\documentclass[a4paper,11pt]{ETHthesis}
\usepackage[innercaption]{sidecap}
\usepackage{amsmath, amsfonts, amssymb, amsthm}
\usepackage[utf8]{inputenc}
\usepackage[english]{babel}
\usepackage{graphicx}
\usepackage{hyperref}
\usepackage{enumerate}
\usepackage{makeidx}
\usepackage{pgf}
\usepackage{tikz}
\usetikzlibrary{arrows,decorations.pathmorphing,backgrounds,positioning,fit}
\usepackage{pgffor}
\makeindex 

\usepackage{fancyhdr}
\def\N{\mathbb{N}}
\def\P{\mathbb{P}}

\def\Xn{(X_{n})}
\def\Zn{(Z_{n})}
\def\Z{\mathbb{Z}}

\def\lgr{\mathbb{Z}_{2}\wr \mathsf{G}}
\def\lgrp{\mathbb{Z}_{2}\wr \Gamma}
\def\0{(\textbf{0},o)}

\newcommand{\entr}{\mathsf{h}}
\newcommand{\eps}{\varepsilon}

\newcommand{\R}{\mathbb{R}}
\newcommand{\C}{\mathcal{C}}
\newcommand\Si{\mathbf{\Sigma}}

\newcommand{\diam}{\operatorname{diam}}
\newcommand{\supp}{\operatorname{supp}}
\newcommand{\AUT}{\operatorname{AUT}}
\newcommand{\AFF}{\mathbb{AFF}}

\newcommand{\abs}[1]{\lvert#1\rvert}

\renewcommand{\theenumi}{(\alph{enumi})}

\theoremstyle{plain}
\newtheorem{theorem}{Theorem}[section]
\newtheorem{lemma}[theorem]{Lemma}
\newtheorem{proposition}[theorem]{Proposition}
\newtheorem{corollary}[theorem]{Corollary}
\newtheorem{conjecture}[theorem]{Conjecture}
\newtheorem{definition}[theorem]{Definition}

\newtheorem{assumption}[theorem]{Assumption}
\newtheorem*{thm1}{Theorem \ref{thm_conv_lrw_general_graphs}}
\newtheorem*{thm2}{Theorem \ref{PoissonTheorem}}
\newtheorem*{thm3}{Summary of Theorems \ref{thm:conv_LRW_fixed_end}, \ref{thm:poisson_lrw_fixed_end} 
and \ref{thm:poisson_fixed_end_zero_drift}}
\newtheorem*{thm4}{Theorem \ref{thm:A}}
\newtheorem*{cor4}{Corollary \ref{cor:B}}

\theoremstyle{definition}
\newtheorem{example}[theorem]{Example}
\newtheorem{remark}[theorem]{Remark}

\begin{document}
\pagenumbering{roman}
{\sffamily%
\title{\huge{Lamplighter Random Walks and Entropy-Sensitivity of Languages}}
\author{\huge{Ecaterina Sava}}
\date{Graz, Oktober 2010}
\maketitle
}
\thispagestyle{empty}
\newpage{}
\thispagestyle{empty}

\clearpage
\mbox{}
\newpage
\thispagestyle{empty}


\begin{center}
{\Large\scshape Statutory Declaration}
\end{center}
I declare that I have authored this thesis independently, that I have not used
other than the declared sources/resources, and that I have explicitely marked
all material which has been quotes either literally or by content from the used
sources.

\vspace*{5cm}

\noindent
\begin{minipage}[h]{0.4\linewidth}
  \begin{center}
    \dotfill\\
    Date
  \end{center}
\end{minipage}
\hspace*{0.1\linewidth}
\begin{minipage}[h]{0.5\linewidth}
  \begin{center}
    \dotfill\\
    (Signature)
  \end{center}
\end{minipage}

\tableofcontents
\clearpage
\mbox{}
\newpage
\pagenumbering{arabic}

\cleardoublepage
\phantomsection
\addcontentsline{toc}{chapter}{Introduction}
\chapter*{Introduction}\label{introduction}

The main purpose of this thesis is to study the interplay between
geometric properties of infinite graphs and analytic and probabilistic 
objects such as transition operators, harmonic functions and random walks 
on these graphs.

Suppose we are given a random walk $\Xn$ on a graph $\mathsf{\mathsf{G}}$. There are many questions
regarding its behaviour, as the discrete parameter $n$ goes to infinity.
Standard questions of this kind are: will the random walk visit some given 
vertex infinitely many times, or will it leave any bounded set after a finite
time with probability one? In the first case the random walk is called \textit{recurrent}
and in the second \textit{transient}.

For a transient random walk, there are several problems one is interested in: 
for instance to study its \textit{convergence} (in a sense to be specified),
to describe the \textit{bounded harmonic functions} for the random walk, to describe
its \textit{Poisson boundary}, or to study the parameter of exponential decay of the
transition probabilities of the random walk (\textit{spectral radius}). 

In the \textit{first part} of the thesis we deal with similar problems in the context of  
random walks on the so-called \textit{lamplighter graphs}, which are
\emph{wreath products} of graphs. Random walks on such graphs will be called \emph{lamplighter
random walks}. All walks we consider will be transient and 
irreducible throughout this part. The convergence and the Poisson boundary of 
\textit{lamplighter random walks} is studied for different underlying graphs, 
and the used methods are mostly of a geometrical nature. Most of the results 
presented here are published in {\sc Sava}~\cite{SavaPoisson2010}.

In the \textit{second part} of the thesis we consider Markov chains
on directed, labelled graphs. With such graphs we associate in a natural way 
a class of infinite languages (sets of labels of paths in the graph) and we study
the \textit{growth sensitivity} (or \textit{entropy sensitivity}) of these languages 
using Markov chains. The growth sensitivity of a language is a well-known and studied 
topic in group theory and symbolic dynamics.  
Under suitable general assumptions on the graphs, we prove that the
associated languages are growth sensitive, by using Markov chains with 
forbidden transitions. This part of the thesis is mainly based on the 
paper by {\sc Huss, Sava and Woess}~\cite{HussSavaWoess2010}.
\addcontentsline{toc}{section}{Overview}
\section*{Overview}

In \textit{Chapter 1} a general introduction to the theory of 
Markov chains, graphs, random walks over graphs and groups, and 
their properties like nearest neighbour type, transience/recurrence, 
spectral radius, rate of escape, is given.
For a Markov chain with transition matrix $P$ over an infinite state space $\mathsf{G}$
we shall either use the notation $(\mathsf{G},P)$ or the sequence $\Xn$
of $\mathsf{G}$-valued random variables, depending on the context. 

\subsection*{Part I}

In \textit{Chapter 2} we introduce a class of random walks
on wreath products $\Z_2\wr \mathsf{G}$ of graphs $\Z_2$ and $\mathsf{G}$, where $\mathsf{G}$ is an
infinite, connected, transitive graph and $\Z_2:=\Z/2\Z$ represents
a finite set with two elements, which encodes the state of
lamps. The intuitive interpretation is as follows. One considers
the base graph $\mathsf{G}$, and at each vertex of $\mathsf{G}$ there is a lamp 
which can be switched on or off. If one defines the set 
\begin{equation*}
\mathcal{C}=\{\eta:\mathsf{G}\to\Z_2,\ \eta \text{ finitely supported}\} 
\end{equation*}
of finitely supported lamp configurations, then the \textit{wreath product $\lgr$} is the graph
with vertex set $\mathcal{C}\times \mathsf{G}$, and adjacency relation
\begin{equation*}
(\eta, x) \sim (\eta', x') :\Leftrightarrow 
\begin{cases}
x\sim x' & \mbox{in} \ \mathsf{G}\  \mbox{and} \  \eta =\eta', \\
x=x' &\mbox{in} \ \mathsf{G} \ \mbox{and}\ \eta \bigtriangleup \eta'=\{x\}.
\end{cases}
\end{equation*}
The space $\lgr$ is called the \textit{lamplighter graph},
and consists of pairs $(\eta,x)$ where $\eta$ is a finitely
supported configuration of lamps and $x$ a vertex in $\mathsf{G}$.
Let $P$ be a transition matrix on $\lgr$ whose entries are of the form 
\begin{equation*}
p\big((\eta,x),(\eta',x')\big), 
\end{equation*}
and $\Zn$ the random walk with transitions in $P$. 
This walk will be called the \textit{lamplighter random walk (LRW)}, and it
can be written as $Z_n=(\eta_n, X_n)$, where $\eta_n$ is the random configuration
of lamps and $X_n$ is the random position in $\mathsf{G}$ at time $n$. 
Then $\Xn$ defines a random walk on $\mathsf{G}$, called the \textit{base random walk},
whose transition matrix $P_{\mathsf{G}}$ is given by the projection of $P$ on the
base graph $\mathsf{G}$. We shall assume that $\Xn$ is transient, and this
implies transience of the lamplighter random walk $\Zn$.

The first part of the thesis is devoted to the description of the behaviour 
of $\Zn$, as $n$ goes to infinity, and the Poisson boundary is the main object
of study in this part. As we shall later see, the behaviour at infinity depends 
on the base random walk $\Xn$ and on the geometry of the transitive base graph $\mathsf{G}$.

Let $\AUT(\mathsf{G})$ be the group of all automorphisms of the graph 
$\mathsf{G}$, and $\Gamma$ a closed subgroup of $\AUT(\mathsf{G})$. 
Also, let $\partial \mathsf{G}$ be a general boundary at infinity of $\mathsf{G}$, 
such that the action of $\Gamma$ on $\mathsf{G}$ extends to this boundary, 
and such that $\Xn$ converges to
$\partial \mathsf{G}$. Then, in \textit{Chapter 3}
the convergence of $\Zn$ to a ``natural'' boundary $\Pi$ of the lamplighter
graph $\lgr$ is proved. The boundary $\Pi$ is defined as
\begin{equation*}
\Pi =\bigcup_{\mathfrak{u} \in \partial \mathsf{G}}\mathcal{C}_{\mathfrak{u}}\times \{\mathfrak{u}\}, 
\end{equation*}
where $\mathcal{C}_{\mathfrak{u}}$ consists of all
configurations $\zeta$ which are either finitely supported,
or infinitely supported with $\mathfrak{u}$ being the
only accumulation point of $\supp(\zeta)$. 
Note that $\supp(\zeta)$ is the subset of $\mathsf{G}$
where the lamps are on. In this settings, we can prove the following.

\begin{thm1}
Let $\Zn$ be an irreducible and homogeneous random walk with finite first 
moment on $\lgr$. Then there exists a $\Pi$-valued random variable 
$Z_{\infty}$, such that $Z_n\to Z_{\infty}$ almost surely
and the distribution of $Z_{\infty}$ is a continuous measure on $\Pi$.
\end{thm1}

We remark that the limit random variable $Z_{\infty}$ is
a pair of the form $(\eta_{\infty},X_{\infty})$, where
$\eta_{\infty}$ is the limit configuration of lamps, which
is not necessary finitely supported, and $X_{\infty}$
is the limit of the base random walk $\Xn$ on $\mathsf{G}$.

\textit{Chapter 4} starts with preliminaries and definitions on
the Poisson boundary of a random walk. Most of them are due 
to {\sc Kaimanovich and Vershik}~\cite{KaimanovichVershik1983},
and {\sc Kaimanovich}~\cite{Kaimanovich2000}. In the most general formulation,
it represents the space of ergodic components of the time shift
in the path space. It is a measure theoretical space, which
gives a representation of the bounded harmonic functions
for the respective random walk in terms of the \textit{Poisson formula}.

The Poisson boundary can also be described using
purely geometric approaches: for instance the \textit{Strip Criterion}
and the \textit{Ray Criterion}, developed by {\sc Kaimanovich}~\cite{Kaimanovich2000}. 
Based on the Strip Criterion (Theorem \ref{thm:strip_crit}), we describe 
the so-called \textit{Half Space Method}
for lamplighter random walks. It requires the existence of a strip 
$\mathfrak{s}(\mathfrak{u},\mathfrak{v})\subset \mathsf{G}$, with 
$\mathfrak{u},\mathfrak{v}\in\partial \mathsf{G}$, which has the properties 
requested by the Strip Criterion. If
such a strip $\mathfrak{s}$ exists for the base random walk 
$\Xn$ on $\mathsf{G}$, then the \textit{Half Space Method}
explains how to construct a ``bigger'' strip $S$
as a subset of $\lgr$, which satisfies again the conditions of 
the Strip Criterion. This method requires that the state space 
$\mathsf{G}$ can be split into ``half spaces'' $\mathsf{G}_+$ and $\mathsf{G}_-$. 

Under suitable assumptions on the transitive base graph $\mathsf{G}$
and on the random walk $\Xn$ on it, we can prove the following 
for lamplighter random walks $\Zn$ on $\lgr$, with $Z_n=(\eta_n,X_n)$.

\vspace{1cm}

\begin{thm2}
Let $\Zn$ be an irreducible, homogeneous random walk with 
finite first moment on $\lgr$. If $\nu_{\infty}$ is the limit distribution of $\Zn$ on 
$\Pi$, then the measure space $(\Pi,\nu_{\infty})$ is the Poisson boundary of $\Zn$.
\end{thm2}

In the remaining chapters of the first part, the base graph $\mathsf{G}$ is replaced 
by the following: a graph with infinitely many ends in \textit{Chapter 5}, a 
hyperbolic graph in \textit{Chapter 6}, and the Euclidean lattice $\Z^d$ in 
\textit{Chapter 7}. For LRW on $\lgr$, the convergence
and the Poisson boundary will be described as an application
of Theorem \ref{thm_conv_lrw_general_graphs}
and Theorem \ref{PoissonTheorem}.

In \textit{Chapter 5} we let $\mathsf{G}$ be a graph with infinitely many ends
and its boundary $\partial \mathsf{G}$ be the space of its ends. 
Like before,  $\Gamma\subset \AUT(\mathsf{G})$ is a group which acts
transitively on $\mathsf{G}$. Two cases should be distinguished: when $\Gamma$
\emph{fixes no end in $\partial \mathsf{G}$}, and when 
$\Gamma$ \textit{fixes one end} in $\partial\mathsf{G}$, which then 
has to be unique. In the first case, Theorem \ref{thm_conv_lrw_general_graphs} and
Theorem \ref{PoissonTheorem} can be adapted in order to prove the convergence and 
to describe the Poisson boundary of LRW $\Zn$ on $\lgr$. The conditions required in
the \textit{Half Space Method} are satisfied for graphs with infinitely many ends and 
random walks on them. The construction of the ``small'' strip $\mathfrak{s}$ as a 
subset of $\mathsf{G}$ is based on the theory of cuts and structure trees of a graph.

The second case, when $\Gamma$ fixes an end of $\mathsf{G}$, is more interesting.
The graph $\mathsf{G}$ can be viewed as an oriented tree $\mathcal{T}$ with a fixed end $\omega$, 
like below.

\begin{minipage}[b]{0.4\linewidth}
The behaviour at infinity of lamplighter random walks $\Zn$
on $\Z_2\wr\mathcal{T}$, depends on the \emph{modular drift}
$\delta(P_{\mathcal{T}})$ of the base random walk $\Xn$
on the oriented tree $\mathcal{T}$. By $P_{\mathcal{T}}$,
we denote the transition matrix of $\Xn$.
We emphasize that the case $\delta(P_{\mathcal{T}})=0$
is the most difficult and interesting one, which is studied in 
\textit{Section \ref{sec:zero_drift}}. 
\end{minipage}
\hspace{0.5cm}
\begin{minipage}[b]{0.6\linewidth}

\begin{tikzpicture}[scale=0.55]

\draw[dashed](0,0)-- (10,0);
\node (a) at (11,0) {$\partial^*\mathcal{T}$};

\draw[thin,->] (0,0.5)--(0,8.5);

\draw[->] (0.5,0.5)--(7.25,8);
\node (b) at (7.6,8.2) {$\omega$};

\draw (5,5.5)--(9.5,0.5);
\draw[loosely dotted] (0,5.5)--(10,5.5);
\draw[loosely dotted] (0,3)--(10,3);
\draw[loosely dotted] (0,1.5)--(10,1.5);
\draw[loosely dotted] (0,0.75)--(10,0.75);

\draw (2.75,3)--(4.75,0.5);
\draw (7.25,3)--(5.25,0.5);

\draw (1.4,1.5)--(2.15,0.5);
\draw (3.95,1.5)--(3.20,0.5);

\draw (6.05,1.5)--(6.80,0.5);
\draw (8.60,1.5)--(7.85,0.5);

\draw (0.74,0.75)--(0.94,0.5);
\draw (1.95,0.75)--(1.75,0.5);

\draw (3.4,0.75)--(3.60 ,0.5);
\draw (4.56,0.75)--(4.36,0.5);

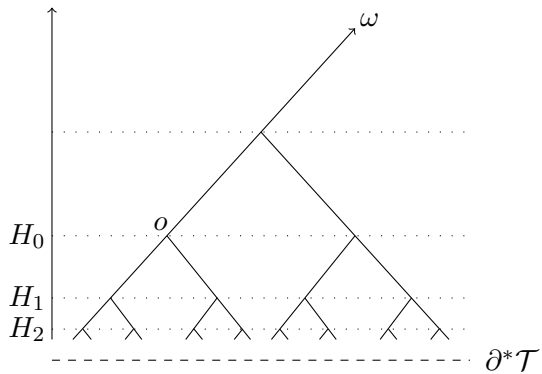
\draw (5.45,0.75)--(5.65 ,0.5);
\draw (6.60,0.75)--(6.40,0.5);

\draw (8.05,0.75)--(8.25 ,0.5);
\draw (9.26,0.75)--(9.06,0.5);

\node (c) at (2.6,3.3){$o$};

\node (g) at (-0.6,0.75) {$H_2$};
\node (h) at (-0.6,1.5) {$H_1$};
\node (i) at (-0.6,3) {$H_0$};

\end{tikzpicture}

\end{minipage}
For this special case, 
the correspondence with a random walk on $\Z$ is used. If we set
\begin{equation*}
\Pi=\Pi^*\cup \omega^*,
\end{equation*}
where 
\begin{equation*}
 \Pi^*=\bigcup_{\mathfrak{u}\in\partial^*\mathcal{T}}\mathcal{C}_{\mathfrak{u}}\times \{\mathfrak{u}\}
\text{ and }\omega^* =\{ (\zeta,\omega):\zeta\in\mathcal{C}_{\omega}\},
\end{equation*}
then we can prove the following.
\vspace{1cm}
\begin{thm3}
Let $\Zn$ be an irreducible and homogeneous random walk with finite 
first moment on $\mathbb{Z}_2\wr\mathcal{T}$, where $\mathcal{T}$ is an
homogeneous tree and $\Gamma$ a subgroup of $\AUT(\mathcal{T})$, which
acts transitively on $\mathcal{T}$ and fixes one end $\omega\in\partial \mathcal{T}$.
Then
\begin{enumerate}[(a)]
\item If $\delta(P_{\mathcal{T}})>0$, then there exists a $\Pi^*$-valued random variable $Z_{\infty}$,
such that $Z_n\to Z_{\infty}$, almost surely. 
If $\nu_{\infty}$ is the limit distribution on $\Pi^*$, 
then $(\Pi^*,\nu_{\infty})$ is the Poisson boundary of $\Zn$.
\item If $\delta(P_{\mathcal{T}})< 0$, then $(Z_n)$ converges almost surely to some random
variable with values in $\omega^*$, and $(\omega^*,\nu_{\infty})$ is the Poisson boundary
of $\Zn$, where $\nu_{\infty}$ is the limit distribution on $\omega^*$.
\item If $\delta(P_{\mathcal{T}})=0$, then $(Z_n)$ converges almost surely to some random
variable in $\omega^*$. Moreover, if $\Xn$ is a nearest neighbour random walk on $\mathcal{T}$
then $(\omega^*,\nu_{\infty})$ is again the Poisson boundary of $\Zn$.
\end{enumerate}  
\end{thm3}
The convergence part follows basically from Theorem
\ref{thm_conv_lrw_general_graphs}, and the description
of the Poisson boundary in the case $\delta(P_{\mathcal{T}})>0$
(and $\delta(P_{\mathcal{T}})<0$, respectively) is an application of Theorem \ref{PoissonTheorem}.

When \textit{$\delta(P_{\mathcal{T}})=0$}, i.e., when the base walk $\Xn$ has zero 
modular drift on $\mathcal{T}$, the Poisson boundary of LRW $\Zn$ on $\Z_2\wr \mathcal{T}$ is 
described in a completely different manner, and uses the existence of cutpoints 
for the random walk $\Xn$ on $\mathcal{T}$. Moreover, the correspondence between the tail 
$\sigma$-algebra of a random walk and its Poisson boundary, which in most cases 
coincide, will be used. This is the content of Section \ref{sec:zero_drift}.

In \textit{Chapter 6}, the base graph $\mathsf{G}$ will be replaced by a 
hyperbolic graph (in the sense of Gromov), 
and its boundary is the hyperbolic boundary $\partial_h \mathsf{G}$.
We are interested only in the case when the boundary is infinite.
Then we can prove again the convergence of LRW to the boundary 
$\Pi$ in Theorem \ref {thm:conv_rw_hyperbolic} and describe the 
Poisson boundary in Theorem \ref{thm:poisson_lrw_hyperbolic_graphs}.

Finally, for the sake of completeness, in \textit{Chapter 7} 
we show how to apply Theorem \ref{thm_conv_lrw_general_graphs}
and Theorem \ref{PoissonTheorem}  to LRW over Euclidean lattices, that is, 
over base graphs $\mathsf{G}=\Z^d$, with $d\geq 3$.  
The results in this chapter were earlier obtained by {\sc Kaimanovich}
\cite{Kaimanovich2001} (for non-zero drift on $\Z^d$) and for the zero-drift case, 
recently by {\sc Erschler} \cite{Erschler2010}.

In the last chapter of the first part several 
open problems and conjectures regarding lamplighter random walks
are stated.

Concluding the overview of the first part of the thesis, let us remark
that {\sc Kaimanovich and Vershik} \cite{KaimanovichVershik1983}
were the first to show that lamplighter groups and graphs are fascinating 
objects in the study of random walks. By now, there is a considerable amount 
of literature on this topis. The paper of {\sc Kaimanovich} \cite{Kaimanovich1991} 
may serve as a major source for earlier literature.
{\sc Lyons, Pemantle and Peres} \cite{LyonsPemantlePeres} investigated the
rate of escape of inward-biased random walks on lamplighter groups.
The lamplighter group $\Z_2\wr\Z$ is one of the examples for which
the entire spectrum for some random walks is known. 
{\sc Grigorchuk and Zuk} \cite{GrigorchukZuk} computed the complete 
spectrum for the random walk, corresponding to a specific generating set of 
the lamplighter group. {\sc Erschler } \cite{ErschlerDrift} proved that
the rate of escape of symmetric random walks on the wreath product 
$\Z_2\wr A$, where $A$ is a finitely generated group, is zero
if and only if the random walk's projection onto $A$ is recurrent.
{\sc Erschler} \cite{Erschler2010} investigated also the Poisson boundary
of lamplighter random walks on $\Z_2\wr\Z^d$, with $d\geq 5$, such that the projection
on $\Z^d$ has zero drift, and she proved that the Poisson boundary is
isomorphic with the space of limit configurations. 

\subsection*{Part II}

In this part, we prove the entropy sensitivity of languages associated in
a natural way with infinite labelled graphs $X$. The proof is based on 
considering Markov chains with forbidden transitions on $X$, and on 
investigating the spectral radius of such chains.

If $\Si$ is a finite \textit{alphabet} and 
$\Si^{*}$ the set of all finite words over $\Si$, then 
a \index{language}\textit{language} $L$ over $\Si$ is a subset 
of $\Si^{*}$. The \index{language!growth}\textit{growth} 
or \index{language!entropy}\textit{entropy} of $L$ is 
\begin{equation*}
\entr(L) 
= \limsup_{n\to\infty}
\frac{1}{n} \log \bigl|\{w \in L:\: \abs{w} = n\}\bigr|.
\end{equation*}
The quantity $\entr(L)$ measures the parameter of exponential decay of $L$.
For a finite, non-empty set $F\subset\Si^*\setminus \{\epsilon\}$ denote
\begin{equation*}
 L^{F}  = \{w\in L:\:\text{no}\; v\in F\; \text{is a subword of}\; w\},
\end{equation*}
where $\epsilon$ is the empty word.

\textit{Question: } For which class of languages
associated with infinite graphs, is $\entr(L^{F})<\entr(L)$?
If this holds for \textit{any} set $F$ of \textit{forbidden subwords}, 
then the language $L$ is called \textit{growth sensitive} 
(or \textit{entropy sensitive}). 

Let $(X,E,\ell)$ be an infinite graph with vertex set $X$, edge
set $E$ and $l:E\to\Si$ a function which associates to each edge 
$e\in E$ its label $l(e)\in\Si$. With $(X,E,\ell)$ we associate the
following languages
\begin{equation*} 
L_{x,y} =\{\text{the labels we read along all paths from } x \text{ to }y \text{ in }X\},
\end{equation*}
for $x,y\in X$. We write 
$
\entr(X) = \entr(X,E,\ell) = 
\sup_{x,y\in X} \entr(L_{x,y})
$
and call this the \index{graph!labelled graph!entropy}\textit{entropy} of 
our oriented, labelled graph. 
Under general assumptions, we can prove the following results.
\begin{thm4}
Suppose that $(X,E,\ell)$ is uniformly connected and deterministic with
label alphabet $\Si$. Let $F \subset \Si^*\setminus \{\epsilon\}$ be a finite, non-empty set
which is relatively dense in $(X,E,\ell)$. Then 
\begin{equation*}
\sup_{x,y\in V}\entr(L_{x,y}^F) < \entr(\mathsf{G}).
\end{equation*} 
\end{thm4}
\begin{cor4}
If $(X,E,\ell)$ is uniformly connected and fully deterministic, then
$L_{x,y}$ is growth sensitive for all $x,y \in X$.
\end{cor4}

What is interesting here is that the proof of the previous results
is based on Markov chains. We consider
a Markov chain with state space $X$ and transition matrix
$P$, whose entries are induced by the labelled edges of $(X,E,\ell)$. 
We then remark that the entropy $\entr(X)$ is in direct
correspondence with the spectral radius $\rho(P)$ of the respective Markov chain.
Moreover, to the restricted language $L^F$ one can also 
associate a ``restricted'' Markov chain, that is, a Markov chain
which is not allowed to cross edges with labels in $F$.
Then the question of growth sensitivity can be interpreted in terms
of Markov chains and its respective spectral radii on $X$. 

Part II is completed with an example where one can apply the 
results developed previously: applications to pairs of groups and
their Schreier graphs.

\chapter{Basic Facts and Preliminaries}\label{preliminaries}

This chapter is devoted to basic definitions and facts connected with the 
theory of Markov chains and random walks on graphs and groups. Moreover, 
we present here some basic tools which are useful for a better understanding of 
the results we are going to present throughout this work. We shall
follow the notations from {\sc Woess} \cite{WoessBook}.

A Markov chain on a state space $\mathsf{G}$, which is adapted to the 
geometry of $\mathsf{G}$, will be called a random walk throughout this thesis.

\section{Markov Chains}\label{Markov_chains}

A \textit{Markov chain} is 
defined by a finite or countable \textit{state space} $\mathsf{G}$ and 
a \textit{transition matrix} (or \emph{transition operator}) $P=\big(p(x,y)\big)_{x,y\in \mathsf{G}}$. In 
addition, one has to specify an initial position, that is, the position at time $0$. 
The entry $p(x,y)$ of $P$ represents the probability to move from $x$ to $y$ in one 
step. This defines a sequence of $\mathsf{G}$-valued random variables $X_{0},X_{1},\ldots$, 
called \index{Markov chain}\textit{Markov chain}, where $X_{n}$ represents the random 
position in $\mathsf{G}$ at time $n$. One can imagine a Markov chain as a walker moving 
randomly in the state space $\mathsf{G}$, according to the probabilities given by the 
transition matrix $P$.

In order to model the random variables $(X_{n})$, one has to find a 
suitable probability space on which the random position after $n$ steps can be 
described as the $n$-th random variable of a Markov chain. The usual choice of 
the probability space is the \index{Markov chain!trajectory space}\textit{trajectory space} $\Omega=\mathsf{G}^{\Z_+}$, 
equipped with the product $\sigma$-algebra arising from the discrete one on $\mathsf{G}$. 
An element $\omega=(x_{0},x_{1},x_{2},\ldots)$ of $\Omega$ represents a possible 
evolution (trajectory), that is, a possible sequence of points visited one 
after the other by the Markov chain. Then, $X_{n}$ is the $n$-th projection 
from $\Omega$ to $\mathsf{G}$. This describes the Markov chain starting at $x$, when $\Omega$ 
is equipped with the probability measure given via the Kolmogorov extension theorem by
\begin{equation*}
\P_{x}[X_{0}=x_{o},X_{1}=x_{1},\ldots,X_{n}=x_{n}]=\delta_{x}(x_{0})p(x_{o},x_{1})\cdots p(x_{n-1},x_{n}). 
\end{equation*}
The associated expectation is denoted by $\mathbb{E}_{x}$. 
Depending on the context, we shall call a Markov chain the 
pair $(\mathsf{G},P)$ or the sequence of random variables $\Xn$. We write 
\begin{equation*}
p^{(n)}(x,y)=\P_{x}[X_{n}=y], 
\end{equation*}
which represents, on one hand the $(x,y)$-entry 
of the matrix power $P^{n}$, with $P^{0}=I$ ($I$ is the identity matrix over $\mathsf{G}$), and 
on the other hand the $n$-step transition probability, that is, the probability the 
get from $x$ to $y$ in $n$ steps.

\begin{definition}
A Markov chain $(\mathsf{G},P)$ is called \index{Markov chain!irreducible}\emph{irreducible}, 
if for every $x,y\in \mathsf{G}$, there is some $n\in\N$ such that $p^{(n)}(x,y)>0$.
\end{definition}
This means that every state $y\in \mathsf{G}$ can be reached from every other state $x\in \mathsf{G}$ with 
positive probability. Throughout this thesis, we shall always require that the state 
space is infinite and all states communicate, i.e., the Markov chain is irreducible.

\section{Random Walks on Graphs}\label{sec:Random walks on graphs}

\paragraph{Graphs.} A \index{graph}\textit{graph} consists of a finite or countable 
set of vertices (points) $\mathsf{G}$, equipped with a symmetric \textit{adjacency relation} 
$\sim$, which defines the set of edges $E(\mathsf{G})$ (as a subset of $\mathsf{G}\times \mathsf{G}$). We write 
$(x,y)$ for the edge between the pair of neighbours $x,y$. For the sake of simplicity, 
we exclude loops, that is $(x,x)\not\in E(\mathsf{G})$, for all $x\in \mathsf{G}$. 

In order to simplify the notation, instead of writing $\big(\mathsf{G},E(\mathsf{G})\big)$ for a 
graph, we shall write only $\mathsf{G}$. It will be clear from the context whether we are 
considering vertices or edges.

A \textit{path} from $x$ to $y$ in $\mathsf{G}$ is a sequence $[x=x_{0},x_{1},\ldots,x_{n}]$ 
of vertices, such that $x_{i}\sim x_{i-1}$, for all $i=1,2,\ldots,n$. The number 
$n\geq 0$ is the \textit{length} of the path. The graph $\mathsf{G}$ is called 
\index{graph!connected}\textit{connected} if every pair of vertices can 
be joined by a path. The usual \index{graph!graph distance}\textit{graph distance} 
$d(x,y)$ is the minimum among the lenghts of all paths from $x$ to $y$. 
A path is called \textit{simple} if it has no repeated vertex, and \textit{geodesic} 
if its length is $d(x,y)$. The degree $\deg(x)$ of a vertex $x$ is the number 
of its neighbours. 

The graph $\mathsf{G}$ is called \index{graph!locally finite}\textit{locally finite} 
if every vertex has finite degree. We say that $\mathsf{G}$ has 
\index{graph!with bounded geometry}\textit{bounded geometry} if it is connected 
with bounded vertex degrees.

Let $\mathsf{G}$ and $\mathsf{G}'$ be two graphs, and $d$ and $d'$ the discrete graph metric
on them, respectively. We say that $\mathsf{G}$ and $\mathsf{G}'$ are \emph{quasi-isometric}
if there exists a mapping $\varphi:\mathsf{G}\to \mathsf{G}'$, such that
\begin{equation*}
A^{-1}d(x,y)-A^{-1}B\leq d'(\varphi x,\varphi y)\leq A d(x,y)+B,
\end{equation*}
for all $x,y\in \mathsf{G}$, and
\begin{equation*}
 d'(x',\varphi \mathsf{G})\leq B,
\end{equation*}
for every $x'\in \mathsf{G}'$, where $A\geq 1$ and $B\geq 0$.

\paragraph*{Graph Automorphisms.}An \index{graph!automorphism}\textit{automorphism} of a graph $\mathsf{G}$ is a self-isometry 
of $\mathsf{G}$ with respect to the graph distance $d$, that is, a bijection
$\varphi :\mathsf{G}\rightarrow \mathsf{G}$ with
\begin{equation*}
d(x,y)=d(\varphi x,\varphi x),\quad \text{ for all } x,y\in \mathsf{G}. 
\end{equation*}
The set of all automorphisms of a graph $\mathsf{G}$ forms a group denoted by $\AUT(\mathsf{G})$.

The graph $\mathsf{G}$ is called \index{graph!transitive}\textit{(vertex)-transitive} 
if for every pair $x,y$ of vertices in $\mathsf{G}$, there exists a graph automorphism 
$\varphi$ with $\varphi x=y$. If $\mathsf{G}$ is transitive, then all vertices have 
the same degree. If there is a \index{group}subgroup $\Gamma$ of $\AUT(\mathsf{G})$, such that, for 
every $x,y\in \mathsf{G}$, there exists $\gamma\in\Gamma$, with $\gamma x=y$, then 
we say that $\Gamma$ \index{group!transitive action}\textit{acts transitively} on $\mathsf{G}$.
Throughout this thesis, we shall only consider transitive graphs.
\paragraph*{The Graph of a Markov Chain.} Every Markov chain $(\mathsf{G},P)$ with 
state space $\mathsf{G}$ and transition matrix $P=\big(p(x,y)\big)_{x,y\in \mathsf{G}}$ 
defines a graph whose vertex set is the state space $\mathsf{G}$ and the (oriented) 
set of edges is given by 
\begin{equation*}
E(\mathsf{G})=\{(x,y): p(x,y)>0\mbox{ with } x,y\in \mathsf{G} \} 
\end{equation*}
When the transition matrix $P$ is adapted to the structure of the underlying 
graph $\mathsf{G}$, then we shall speak of a \index{random walk}\textit{random walk} 
on $\mathsf{G}$ (instead of a Markov chain).

The \index{random walk!simple}\textit{simple random walk} (SRW) on a locally 
finite graph $\mathsf{G}$ is the Markov chain whose state space is $\mathsf{G}$ and the 
transition probabilities are given by
\begin{equation}
p(x,y)= 
\begin{cases}
\dfrac{1}{\deg(x)},&\mbox{ if }y\sim x\\
0,&\mbox{otherwise}.
\end{cases}
\end{equation}
This is the basic example of a Markov chain adapted to the underlying graph $\mathsf{G}$.
Throughout this thesis, we shall consider more general 
types of adaptedness properties of the transition matrix $P$
 to the underlying structure, and we shall speak of random
 walks (instead of Markov chains). We define here some of
 these properties, which will be frequently used.
\begin{definition}
The random walk $(\mathsf{G},P)$ is of \index{random walk!nearest neighbour}
 \emph{nearest neighbour type}, if $p(x,y)>0$ occurs only when $d(x,y)\leq 1$. 
\end{definition}
\begin{definition}
The random walk $(\mathsf{G},P)$ is said to have \index{random walk!finite moment}\emph{$k$-th finite moment}
 with respect to the usual graph distance $d$ on the graph $\mathsf{G}$, if 
\begin{equation*}
\sum_{x\in \mathsf{G}}d(o,x)^{k}p(o,x)<\infty,\quad \text{ for all } x\in \mathsf{G},
\end{equation*}
for some fixed vertex $o$ in $\mathsf{G}$. 
\end{definition}
Further adaptedness conditions of geometric type will be introduced later on.

\paragraph*{Green Function and Spectral Radius. } Assume that the random walk 
$(\mathsf{G},P)$ is irreducible. The \index{Green function}\textit{Green function} 
associated with $(\mathsf{G},P)$ is given by the power series
\begin{equation*}
G(x,y|z)=\sum_{n=0}^{\infty}p^{(n)}(x,y)z^{n},\quad x,y\in \mathsf{G}, z\in \mathbb{C}. 
\end{equation*}
We write $G(x,y)$ for $G(x,y|1)$. This is the expected number of
visits of $\Xn$ to $y$ when $X_{0}=x$. 

\begin{lemma}
For all $x,y\in \mathsf{G}$ the power series $G(x,y|z)$ has the same finite radius of 
convergence $\mathfrak{r}(P)$ given by
\begin{equation*}
\mathfrak{r}(P):=\Big(\limsup_{n\to\infty}\big(p^{(n)}(x,y)\big)^{1/n}\Big)^{-1}<\infty. 
\end{equation*}
\end{lemma}
\begin{proof}
 The fact that the power series defining the functions $G(x,y|z)$ have all the same radius
of convergence follows from a system of Harnack-type inequalities.  Because of the
irreducibility of $(\mathsf{G},P)$, for all $x_1,x_2,y_1,y_2\in \mathsf{G}$ there exist some 
$k,l\in\mathbb{N}$ such that,
\begin{equation*}
 p^{(k)}(x_1,x_2)>0 \text{ and } p^{(l)}(y_2,y_1)>0.
\end{equation*}
Therefore, for every $n\in\mathbb{N}$, we have
\begin{equation*}
p^{(n+k+l)}(x_1,y_1)\geq p^{(k)}(x_1,x_2)p^{(n)}(x_2,y_2)p^{(l)}(y_2,y_1). 
\end{equation*}
Consequently, for every positive argument of the Green function,
\begin{equation*}
 G(x_1,y_1|z)\geq p^{(k)}(x_1,x_2)p^{(l)}(y_2,y_1)z^{k+l}G(x_2,y_2|z).
\end{equation*}
It follows that the radius of convergence of $G(x_1,y_1|z)$ is at least as big as
that of $G(x_2,y_2|z)$. The fact that $\mathfrak{r}(P)<\infty$ follows 
from the irreducibility of $(\mathsf{G},P)$: let $k\in\mathbb{N}$
such that $p^{(k)}(x,x)=\varepsilon>0$, then $p^{(nk)}(x,x)\geq\varepsilon^n>0$
for every $n\geq 0$.
\end{proof}
Hence, the Green function has the following important property: 
if the random walk on $\mathsf{G}$ is irreducible and $z$ is a real number 
greater than zero, then the power series $G(x,y|z)$ either converges 
or diverges simultaneously for all $x,y\in \mathsf{G}$. For more details, 
see {\sc Woess}~\cite{WoessBook}. 
\begin{definition}
The \index{spectral radius}\emph{spectral radius} of the random walk $(\mathsf{G},P)$ is
\begin{equation*}
 \rho(P)=\limsup_{n\to\infty}p^{(n)}(x,y)^{1/n}\in (0,1].
\end{equation*}
\end{definition}
\begin{definition}
The random walk $(\mathsf{G},P)$ is called \index{random walk!recurrent}\emph{recurrent} 
if $G(x,y)=\infty$ for some ($\Leftrightarrow$ every) $x,y\in \mathsf{G}$. Otherwise, 
is called \index{random walk!transient}\emph{transient}.
\end{definition}

\begin{proposition}Further characterizations of recurrence and transience:
\begin{enumerate}
\item If $\rho(P)<1$, then $(\mathsf{G},P)$ is transient. The converse is not true.
\item If $(\mathsf{G},P)$ is recurrent then
\begin{equation*}
\mathbb{P}_{x}[X_n=y \text{ for infinitely many } n]=1,\quad \text{for all }x,y\in \mathsf{G}.
\end{equation*}
\item If $(\mathsf{G},P)$ is transient, then for every finite $A\subset \mathsf{G}$,
\begin{equation*}
 \mathbb{P}_{x}[X_n\in A\text{ for infinitely many }n]=0,\quad \text{for all }x\in \mathsf{G}.
\end{equation*}
\end{enumerate} 
\end{proposition}
In other words, a random walk $(\mathsf{G},P)$ is \textit{recurrent} if every 
element of the state space $\mathsf{G}$ is visited infinitely often with 
probability $1$. Equivalently, in the transient case, each element 
is visited only finitely many times with probability $1$. This is the same 
as saying that $\Xn$ leaves finite subset of $\mathsf{G}$ almost surely after a 
finite time. 

\begin{example}
The SRW on $\mathbb{Z}$ is the Markov chain with state space $\mathsf{G}=\mathbb{Z}$
and transition probabilities 
\begin{equation*}
p(x,x+1)=p(x,x-1)=\frac{1}{2},\quad \text{ for all } x\in \Z.
\end{equation*}
The random walk on $\mathbb{Z}$ with drift to the right is the Markov chain
on $\mathsf{G}=\mathbb{Z}$ with 
\begin{equation*}
p(x,x+1)=1-p(x,x-1)=p,\quad \text{ for all } x\in \Z \text{ and } p\in \Big(\frac{1}{2},1\Big).
\end{equation*}
The SRW on $\mathbb{Z}$ is recurrent and the 
random walk with drift is transient.

\begin{minipage}[b]{0.5\linewidth}
\begin{tikzpicture}
\foreach \x in {-2,-1,0,1,2}
	\filldraw [black] (\x,0) circle (2.5pt) node[below] {$\x$};		
\draw[thick] (-2.5,0)--(2.5,0);
\draw[thick,black,->] (0.1,0.5)--(0.95,0.5);
\draw[thick,black,->] (-0.1,0.5)--(-0.95,0.5);
\draw  node at (0.5,0.8) {$1/2$};
\draw  node at (-0.5,0.8) {$1/2$};
\end{tikzpicture}
\quad \quad \quad Simple random walk on $\mathbb{Z}$
\end{minipage}
\hspace{0.75cm}
\begin{minipage}[b]{0.5\linewidth}
\begin{tikzpicture}
\foreach \x in {-2,-1,0,1,2}
	\filldraw [black] (\x,0) circle (2.5pt) node[below] {$\x$};		
		
\draw[thick] (-2.5,0)--(2.5,0);

\draw[thick,black,->] (0.1,0.5)--(0.95,0.5);
\draw[thick,black,->] (-0.1,0.5)--(-0.95,0.5);
\draw  node at (0.5,0.8) {$p$};
\draw  node at (-0.5,0.8) {$1-p$};
\end{tikzpicture}
\quad \quad Random walk on $\mathbb{Z}$ with drift
\end{minipage}
\end{example}

\paragraph*{Rate of Escape}
\begin{proposition}[Rate of Escape, Drift]
If the random walk $(\mathsf{G},P)$ has finite first moment with respect to $d$, then there exists
a finite number $l=l(P)$ such that
\begin{equation*}
 \frac{d(o,X_n)}{n}\to l, \quad \text {almost surely}.
\end{equation*}
The number $l=l(P)$ is called \index{Markov chain!rate of escape}\emph{rate of escape} 
or \index{Markov chain!drift}\emph{drift} of the random walk $\Xn$ with transition matrix $P$.
\end{proposition}
The rate of escape is only of interest for transient random walks, since
for the recurrent ones it is always zero.
The existence of the rate of escape is a 
consequence of \textit{Kingman's subadditive ergodic theorem} which we formulate
now. See {\sc Kingman}~\cite{Kingman1968} for details. 
\begin{theorem}\label{thm:kingman}
Consider the probability space $\Omega=\mathsf{G}^{\Z_+}$ and let $T$
be the time shift on $\Omega$ with $T(x_0,x_1,\ldots)=(x_1,x_2,\ldots)$.
If $(W_n)$ is a subadditive sequence of non-negative real-valued
random variables, that is, for all $k,n\in\Z_+$ 
\begin{equation*}
W_{k+n}\leq W_n+W_k\circ T^n, 
\end{equation*}
holds, and if $W_1$ is integrable, then there
is a $T$-invariant real-valued integrable random variable $W_{\infty}$
such that
\begin{equation*}
 \lim_{n\to\infty}\frac{1}{n}W_n=W_{\infty},\quad \text{ almost surely}.
\end{equation*}
\end{theorem}

\paragraph*{Reversible Markov Chains}
\begin{definition}
A Markov chain $(\mathsf{G},P)$ is called \index{Markov chain!reversible}\emph{reversible} 
if there exists a measure $m:\mathsf{G}\rightarrow (0,\infty)$ such that
\begin{equation*}
m(x)p(x,y)=m(y)p(y,x),\quad\mbox{ for all } x,y\in \mathsf{G}. 
\end{equation*}
\end{definition}
Also, $m$ is called the \textit{reversible measure} for $(\mathsf{G},P)$.
If $(\mathsf{G},P)$ is the simple random walk on $\mathsf{G}$, then $m(x)=\deg(x)$.
\begin{definition}
A function $f\in \ell^{\infty}(\mathsf{G})$ is called \emph{$P$-harmonic} if $Pf=f$ pointwise,
where the Markov operator $P$ acts on functions $f\in \ell^{\infty}(\mathsf{G})$ by
\begin{equation*}
{P} f(x)=\sum_{y\in \mathsf{G}} p(x,y)f(y).
\end{equation*}
We say that $f$ is \emph{$P$-superharmonic} if $Pf\leq f$ pointwise. 
\end{definition}
Reversibility is the same as saying that the transition matrix
$P$ acts on $l^2(\mathsf{G},m)$ as a self-adjoint operator, that is, 
$(Pf_1,f_2)=(f_1,Pf_2)$, for all $f_1,f_2\in l^2(\mathsf{G},m)$, where
the Hilbert space $l^2(\mathsf{G},m)$ consists of functions $f:\mathsf{G}\to\mathbb{R}$ with
\begin{equation*}
 \sum_{x\in \mathsf{G}}\abs{f(x)}^2 m(x)< \infty.
\end{equation*}
The inner product on $l^2(\mathsf{G},m)$ is given by
\begin{equation*}
(f_1,f_2)=\sum_{x\in \mathsf{G}}f_1(x)f_2(x)m(x).
\end{equation*}
There is another useful characterization of the recurrence of a
random walk in terms of superharmonic functions, which we state now.
\begin{theorem}
 $(\mathsf{G},P)$ is recurrent if and only if all non-negative superharmonic
functions are constant.
\end{theorem}

\paragraph{Markov Chains and Reversed Markov Chains.}
If the Markov chain $(\mathsf{G},P)$ is reversible with respect to the measure $m$,
then one can construct the \index{Markov chain!reversed} \textit{reversed Markov 
chain} $(\mathsf{G},\check{P})$ on $\mathsf{G}$, whose transition probabilities are given by
\begin{equation*}
\check{p}(x,y)=p(y,x)\dfrac{m(y)}{m(x)}. 
\end{equation*}
The reversed Markov chain inherits the properties of the original one.

If $\mathsf{G}$ is a transitive graph, $P$ a transition matrix on $\mathsf{G}$ and
$\Gamma\subset \AUT(\mathsf{G})$ is a group which acts transitively on $\mathsf{G}$,
then one can construct a reversible measure $m$ for the Markov chain $(\mathsf{G},P)$.
For doing this, let $o\in \mathsf{G}$ be a fixed reference vertex, 
whose choice is irrelevant by the transitivity assumption. Denote by 
$\Gamma_{x}$ the \index{group!stabilizor}\textit{stabilizer} of $x$ in 
$\Gamma$, that is, 
\begin{equation*}
\Gamma_{x}=\{\gamma\in\Gamma:\gamma x=x\},\quad \text{ for }x\in\mathsf{G},
\end{equation*}
and by $\Gamma_{o}x$ the \index{group!orbit}\textit{orbit} of $x$ 
under the action of $\Gamma_{o}$, i.e.
\begin{equation*}
\Gamma_{o}x=\{\gamma x: \gamma\in\Gamma_{o}\}. 
\end{equation*}
Then it is easy to check that 
\begin{equation*}
 m(x)=\dfrac{|\Gamma_{o}x|}{|\Gamma_{x}o|}
\end{equation*}
is a \index{Markov chain!reversible measure}reversible measure for $(\mathsf{G},P)$.
By $|\cdot |$ we denote the cardinality of the respective set. 
Note that if the group $\Gamma$ is discrete, then the measure 
$m$ is just the counting measure. 

\paragraph{Homogeneous Markov Chains.}
Let $(\mathsf{G},P)$ be an irreducible Markov chain on $\mathsf{G}$. We denote by
\begin{equation*}
 \AUT(\mathsf{G},P)=\{g\in \AUT(\mathsf{G}):\ p(x,y)=p(gx,gy),\ \text{ for all } x,y\in \mathsf{G}\}
\end{equation*}
the group of automorphisms (isometries) of $\mathsf{G}$ which leaves invariant the 
transition probabilities of $P$.

\begin{definition}
A Markov chain is called \index{Markov chain!homogeneous}\emph{homogeneous}
or \index{Markov chain!transitive}\emph{transitive} if the group $\AUT(\mathsf{G},P)$ 
acts on $\mathsf{G}$ transitively. 
\end{definition}
Throughout this thesis we shall consider transient Markov chains which are
irreducible and homogeneous.

\section{Random Walks on Finitely Generated Groups}

Let $\Gamma$ be a discrete group with unit element $e$, with the group 
operation written multiplicatively, unless $\Gamma$ is abelian. Let also 
$\mu$ be a probability measure on $\Gamma$. The 
\index{random walk!on a group}\textit{(right) random walk} 
on $\Gamma$ with law $\mu$, denoted by $(\Gamma,\mu)$, is the Markov chain with 
state space $\Gamma$ and transition probabilities given by 
\begin{equation}\label{eq:trpb_rw}
 p(x,y)=\mu(x^{-1}y),\quad \text{ for all }\ x,y\in\Gamma.
\end{equation}
In order to obtain an equivalent model of the random walk $(\Gamma,\mu)$
as a sequence of random variables $(S_n)$, 
we use the product space $(\Gamma,\mu)^{\N}$. For $n\geq 1$, the $n$-th 
projections $Y_{n}$ of $\Gamma^{\N}$ onto $\Gamma$
constitute a sequence of independent $\Gamma$-valued random 
variables with common distribution $\mu$, and the right random 
walk starting at $x\in \Gamma$ is given as
\begin{equation*}
 S_{n}=x Y_{1}\cdots Y_{n},\quad n\geq 1.
\end{equation*}
This is a generalization of the scheme of sums of i.i.d. random 
variables on the integers or on the reals. The $n$-step transition 
probabilities are obtained by
\begin{equation*}
p^{(n)}(x,y)=\mu^{(n)}(x^{-1}y), 
\end{equation*}
where $\mu^{(n)}$ is the $n$-fold convolution of $\mu$ with itself, 
with $\mu^{0}=\delta_{e}$, the point mass at the group identity. 
We denote by $\mathbb{P}_{x}$ the measure of the random walk and omit 
the subscript if the random walk starts at the identity $e$.

In order to relate random walks on groups with random walks on graphs, 
let us introduce the notion of \textit{Cayley graphs}, that is, graphs 
that encode the structure of discrete groups. Suppose that the group 
$\Gamma$ is finitely generated, and let $S$ be a symmetric set of 
generators of $\Gamma$. The \index{graph!Cayley graph}\textit{Cayley graph $\mathsf{G}(\Gamma, S)$} 
of $\Gamma$ with respect to the generating set $S$ has vertex set $\Gamma$, 
and two vertices $x,y\in\Gamma$ are connected by an edge, if and only 
if $x^{-1}y\in S$. This graph is connected, locally finite, and regular 
(all vertices have the same degree $|S|$). Notice that Cayley graphs are 
transitive in the sense that they look the same from every vertex. If $e\in S$, 
then $\mathsf{G}(\Gamma,S)$ has a loop at each vertex.

\begin{minipage}[b]{0.6\linewidth}
\begin{example}
The homogeneous tree $\mathcal{T}_{M}$ of degree $M$
is the Cayley graph of the group
$\langle a_{1},a_{2},\ldots ,a_{M}|a_{i}^{2}=e\rangle$
with respect to the 
generators $S=\{a_{1},a_{2},\ldots a_{M}\}$. This group is the free 
product of $M$ copies of the two-element group $\mathbb{Z}_{2}$. 
See Chapter II in {\sc Woess}~\cite{WoessBook} for details on free products.
\end{example} 
\end{minipage}
\hspace{0.5cm}
\begin{minipage}[b]{0.4\linewidth}

\begin{tikzpicture}[scale=0.6]
\draw (-1,0)-- node[above] {$\mathcal{T}_3$} (1,0);
\draw (1,0)--(2,1);
	\draw (2,1)--(3,1);
		\draw (3,1)--(3.5,1.5);
		\draw (3,1)--(3.5,0.5);
	\draw (2,1)--(2,2);
		\draw (2,2)--(2.5,2.5);
		\draw (2,2)--(1.5,2.5);
\draw (1,0)--(2,-1);
	\draw (2,-1)--(3,-1);
		\draw (3,-1)--(3.5,-0.5);
		\draw (3,-1)--(3.5,-1.5);
	\draw (2,-1)--(2,-2);
		\draw (2,-2)--(1.5,-2.5);
		\draw (2,-2)--(2.5,-2.5);
\draw (-1,0)--(-2,1);
	\draw (-2,1)--(-3,1);
		\draw (-3,1)--(-3.5,1.5);
		\draw (-3,1)--(-3.5,0.5);
	\draw (-2,1)--(-2,2);
		\draw (-2,2)--(-1.5,2.5);
		\draw (-2,2)--(-2.5,2.5);
\draw (-1,0)--(-2,-1);
	\draw (-2,-1)--(-3,-1);
		\draw (-3,-1)--(-3.5,-0.5);
		\draw (-3,-1)--(-3.5,-1.5);
	\draw (-2,-1)--(-2,-2);
		\draw (-2,-2)--(-1.5,-2.5);
		\draw (-2,-2)--(-2.5,-2.5);

\foreach \x in {-1,1}
	\filldraw [black] (\x,0) circle (2.5pt);
\foreach \x in {-3,-2,2,3}
	\filldraw [black] (\x,1) circle (2.5pt);
\foreach \x in {-3,-2,2,3}	
	\filldraw [black] (\x,-1) circle (2.5pt);
\foreach \x in {-2,2}	
	\filldraw [black] (\x,2) circle (2.5pt);
\foreach \x in {-2,2}	
	\filldraw [black] (\x,-2) circle (2.5pt);	
\end{tikzpicture}
\end{minipage}

\begin{minipage}{0.6\linewidth}
\begin{example}
Euclidean lattices are the most well-known Cayley graphs. 
In the abelian group $\mathbb{Z}^{d}$, we may choose
the set of all elements with Euclidean 
length $1$ as generating set $S$. The resulting Cayley graph is the
usual grid. The group $\mathbb{Z}^{2}$ can be written as $\langle a,b|ab=ba\rangle$.
\end{example} 
\end{minipage}
\hspace{1cm}
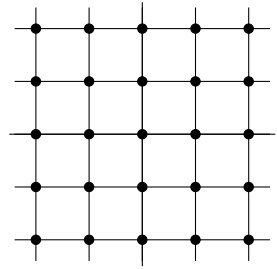
\begin{minipage}{0.4\linewidth}
\begin{tikzpicture}[scale=0.7]
\draw (-2.5,0) -- (2.5,0);
\draw (0,-2.5) -- (0,2.5);
\draw[step=1cm] (-2.4,-2.4) grid (2.4,2.4);

\foreach \x in {-2,-1,0,1,2}
	\foreach \y in {-2,-1,0,1,2}
		\filldraw [black] (\x,\y) circle (2.5pt);

\end{tikzpicture}
\end{minipage}

The \textit{simple random walk} on $\mathsf{G}=\mathsf{G}(\Gamma,S)$ is the right random walk on 
$\Gamma$ whose law $\mu$ is the equidistribution on $S$, i.e., $\mu(s)=1/|S|$ for $s\in S$.

For an arbitrary distribution $\mu$, we write $\supp(\mu)=\{x\in\Gamma:\mu(x)>0\}$. 
Then $\supp(\mu^{(n)})=(\supp(\mu))^n$, and the random walk with 
law $\mu$ is \index{random walk! irreducible}irreducible if and only if
\begin{equation*}
\bigcup_{n=1}^{\infty}(\supp(\mu))^n=\Gamma, 
\end{equation*}
i.e., the set $\supp(\mu)$ generates $\Gamma$ as a group.

\paragraph{Reversed Random Walk.}
If $(S_n)$ is the right random walk on $\Gamma$ with distribution law $\mu$, then the 
\index{random walk!reversed random walk}\textit{reversed random walk} 
$(\check{S}_n)$ on $\Gamma$ has the distribution law $\check{\mu}$ given by
\begin{equation}\label{eq:reversed_rw}
 \check{\mu}(\gamma)=\mu(\gamma^{-1}),\quad \text{for all }\gamma\in\Gamma.
\end{equation}

\paragraph{Random Walks on $\mathsf{G}$ and on $\Gamma\subset \AUT(\mathsf{G},P)$.}
Let $\Gamma$ be a closed subgroup of $\AUT(\mathsf{G},P)$ 
which acts transitively on $\mathsf{G}$. The graph $\mathsf{G}$ should not be
necessary a Cayley graph of $\Gamma$. One can then
define random walks on $\Gamma$ which are in direct 
correspondence with random walks $(\mathsf{G},P)$ on $\mathsf{G}$. Such random walks on $\Gamma$
inherit the properties of $(\mathsf{G},P)$.

The group $\Gamma$ carries a \emph{left Haar measure $d\gamma$}, since
$\Gamma\subset \AUT(X,P)$. The measure $d\gamma$ has the following properties:
every open set has positive measure, every compact set has finite measure
and $d\gamma$ is a left translation invariant measure. Moreover, as a Radon
measure with these properties, $d\gamma$ is unique up to multiplication by constants. 
If $\Gamma$ is discrete, the Haar measure is (a multiple of) the counting measure.
For details concerning integration on locally compact groups and Haar
measures on groups, the reader may have a look at the book
of {\sc Hewitt and Ross}~\cite{HewittRoss1963}.

Let us choose a left Haar measure $d\gamma$ on $\Gamma$, such that 
$\int_{\Gamma_o}d\gamma=1$, where $o\in \mathsf{G}$ is a reference vertex.
With the transition probabilities $P$ of the random walk $\Xn$ on $\mathsf{G}$, 
one can associate a Borel measure $\mu$ on $\Gamma$ by
\begin{equation}\label{eq:correspondence_rw}
 \mu(d\gamma)=p(o,\gamma o)d\gamma.
\end{equation}
The measure $\mu$ is absolutely continuous with respect to $d\gamma$.
One can check that $\mu$ defines a probability measure on $\Gamma$,
and induces the right random walk $(S_n)$ on $\Gamma$. Then $(S_no)$
is a model of the random walk $\Xn$ on $\mathsf{G}$ starting at $o$. In other words,
$(S_no)$ is a homogeneous Markov chain with transition probabilities
$p(x,y)$, for  $x,y\in \mathsf{G}$.  For more details, see also
{\sc Kaimanovich and Woess}~\cite[Section 2]{KaimanovichWoess2002}.

Therefore, whenever one has a Markov chain $(\mathsf{G},P)$ and a closed 
subgroup $\Gamma$ of $\AUT(\mathsf{G},P)$, then one can construct a measure $\mu$ on $\Gamma$ 
which is in direct correspondence with the transition probabilities $P$ by equation
\eqref{eq:correspondence_rw}. If $\mathsf{G}$ is a Cayley graph of $\Gamma$ with respect to some
generating set, then the correspondence between $\mu$ and $P$ 
is given by \eqref{eq:trpb_rw}.

\part{Behaviour at Infinity of Lamplighter Random Walks}
\chapter{Lamplighter Random Walks}

The aim of this chapter is to introduce a class of random walks
on \index{graph!wreath product}\textit{wreath products} of groups and graphs.
Wreath products of groups are the simplest non-trivial case of
semi-direct products, because they essentially arise from the 
action of a group on itself by translation. Such groups
are called \textit{groups with dynamical configuration}
in \cite{KaimanovichVershik1983}.

Random walks on wreath products are known in the literature as
\textit{lamplighter random walks}, because of the intuitive
interpretation of such walks in terms of configuration of lamps. 
Such walks appear also in the paper of {\sc Varopoulos}~\cite{Varopoulos1983}.

We first introduce lamplighter graphs and random walks on them, and
afterwards we consider group actions on lamplighter graphs.
Depending on the group actions, the random walks
will inherit different behaviour at infinity, which will be studied in the sequel. 
For simplicity of notation, we shall mostly write LRW for lamplighter random walks.

\section{Lamplighter Graphs}

Let $\mathsf{G}$ be an infinite, locally finite, transitive, connected graph and let 
$o$ be some reference vertex in $\mathsf{G}$. Imagine that at each 
vertex of $\mathsf{G}$ sits a lamp, which can have different states of intensity, 
but only finitely many. For sake of simplicity, we shall 
consider the case when the lamp has only two states, encoded 
by the elements of the set $\Z_{2}:=\Z/2\Z=\{0,1\}$, where the element 
$0$ represents the state off (the lamp is switched off) and 
the element $1$ represents the state on (the lamp is switched on). 
Anyway, instead of $\Z_{2}$ one can consider any finite set.

One can think of a person starting in $o$ with all 
lamps switched off and moving randomly in $\mathsf{G}$, according to 
some given probability, and switching randomly lamps on or off. 
We investigate the following model: at each step the person 
may walk to some random vertex (situated in a bounded neighbourhood 
of his  current position), and may change the state of some lamps in a 
bounded neighbourhood of his position. At every moment of time 
the lamplighter will leave behind a certain configuration 
of lamps. The configurations of lamps are encoded by functions 
\begin{equation*}
\eta:\mathsf{G}\rightarrow \Z_{2}, 
\end{equation*}
which give, for every $x\in \mathsf{G}$, the state of the lamp sitting there. Denote by 
\begin{equation*}
\hat{\mathcal{C}}=\{\eta : \mathsf{G}\rightarrow\Z_{2}\} 
\end{equation*}
the set of all configurations, and let $\mathcal{C} \subset \hat{\mathcal{C}}\ $ 
be the set of all finitely supported configurations, where a 
configuration is said to have finite support if the set 
\begin{equation*}
\supp(\eta)=\{x \in \mathsf{G}: \eta(x)\ne  0  \} 
\end{equation*}
is finite. Denote by $\bf{0}\bf$ the \textit{zero} or \textit{trivial configuration}, 
i.e. the configuration which corresponds to all lamps switched off, 
and by $\delta_{x}$ the configuration where only the lamp at 
$x\in \mathsf{G}$ is on and all other lamps are off. 

\begin{definition}
The \emph{wreath product} $\lgr$ of graphs $\Z_2$ and $\mathsf{G}$ is defined as the 
graph with vertex set $\mathcal{C}\times \mathsf{G}$ 
and adjacency relation given by 
\begin{equation}\label{neigh_relation}
(\eta, x) \sim (\eta', x') :\Leftrightarrow 
\begin{cases}
x\sim x' & \mbox{in} \ \mathsf{G}\  \mbox{and} \  \eta =\eta', \\
x=x' &\mbox{in} \ \mathsf{G} \ \mbox{and}\ \eta \bigtriangleup \eta'=\{x\},
\end{cases}
\end{equation}
where $\eta \bigtriangleup \eta'$ represents the subset of $\mathsf{G}$, 
where the configurations $\eta$ and $\eta'$ are different.
\end{definition}
The wreath product $\lgr$ will be refered as the \index{graph!lamplighter graph}\textit{lamplighter graph}.
The vertices of $\Z_{2}\wr \mathsf{G}$ are pairs of the form $(\eta,x)$, 
where $\eta$ represents a finitely supported $\Z_{2}$-valued 
configuration of lamps and $x$ some vertex in $\mathsf{G}$. The graph 
$\mathsf{G}$ will be called the \textit{base graph} or
the \textit{underlying graph} for the lamplighter graph $\lgr$. 

\begin{minipage}[b]{0.6\linewidth}
\begin{example}
Consider the wreath product $\Z_2\wr\Z_2$: the base graph 
and the graph of lamp states are both $\Z_2$.
Denote the vertices of the base graph by $\{a,b\}$ and the 
state of lamps by $\{0,1\}$. Then the lamplighter graph 
$\Z_2\wr\Z_2$ has $8$ vertices and it can be represented as in the figure.
\end{example} 
\end{minipage}
\hspace{0.5cm}
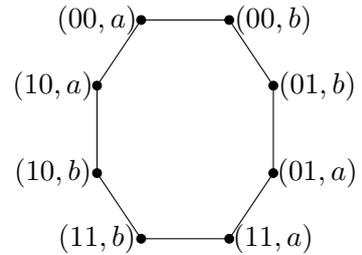
\begin{minipage}[b]{0.4\linewidth}
\begin{tikzpicture}[scale=0.58]

\draw (-1,2.5)--(-2,1);
\node (a) at (-2,2.5) {$(00,a)$};
\draw (-2,1)--(-2,-1);
\node (b) at (-3,1) {$(10,a)$};
\draw (-2,-1)--(-1,-2.5);
\node (c) at (-3,-1) {$(10,b)$};
\draw (-1,-2.5)--(-1,-2.5);
\node (d) at (-2,-2.5) {$(11,b)$};
\draw (-1,-2.5)--(1,-2.5);
\node (e) at (2,-2.5) {$(11,a)$};
\draw (1,-2.5)--(2,-1);
\node (f) at (3,-1) {$(01,a)$};
\draw (2,-1)--(2,1);
\node (g) at (3,1) {$(01,b)$};
\draw (2,1)--(1,2.5);
\node (h) at (2,2.5) {$(00,b)$};
\draw (1,2.5)--(-1,2.5);

\filldraw (-1,2.5) circle (0.1cm);
\filldraw (-2,1) circle (0.1cm);
\filldraw (-2,-1) circle (0.1cm);
\filldraw (-1,-2.5) circle (0.1cm);
\filldraw (1,-2.5) circle (0.1cm);
\filldraw (2,-1) circle (0.1cm);
\filldraw (2,1) circle (0.1cm);
\filldraw (1,2.5) circle (0.1cm);
\end{tikzpicture}

\end{minipage}

The wreath product of graphs of bounded geometry is a graph of bounded geometry,
and the wreath product of regular graphs is also regular. 

Let us now define a metric $d$ on the graph $\Z_{2}\wr \mathsf{G}$. If we denote by $l(x,x')$ the 
smallest length of a ``travelling salesman'' tour from $x$ to $x'$ that 
visits each element of the set $\eta \bigtriangleup \eta '$, then
\begin{equation}\label{LamplighterGraphGmetric}
d\left((\eta,x),(\eta',x')\right)=l(x,x')+|\eta'\bigtriangleup\eta|
\end{equation}
defines a metric on $\Z_{2}\wr \mathsf{G}$. Recall that the travelling 
salesman tour between two given points is the shortest possible 
tour that visits each point exactly once. Above, $|\cdot|$ 
represents the cardinality of the respective set. This metric 
will be called the \textit{lamplighter metric} or \textit{lamplighter distance}. 

\begin{remark} One can also consider a generalization
of wreath products, which are called in {\sc Erschler} \cite{Erschler2006}
\emph{wreath producs of graphs with respect to a family of subsets} and, in particular,
with respect to partitions.
\end{remark}

\section{Random Walks on Lamplighter Graphs}

Consider the state space $\lgr$ defined as in the previous section and an 
irreducible transition matrix $P$ on it, which determines the walk $(\lgr, P)$. 
This random walk will be called the  \index{random walk!lamplighter}
\emph{lamplighter random walk (LRW)}. The entries 
\begin{equation*}
p\big((\eta,x),(\eta',x')\big),\quad \text{ with } (\eta,x), (\eta',x')\in\lgr,
\end{equation*}
of the transition matrix $P$ are the one-step transition probabilities,
while the corresponding $n$-step transition probabilites of the 
random walk  are denoted by $p^{(n)}\big((\eta,x),(\eta',x')\big)$. 
Suppose that the starting point for the LRW is $\0$, 
where $o$ is a fixed vertex in $\mathsf{G}$, and $\textbf{0}$ is the 
trivial configuration.

The random walk on $\lgr$ with transition matrix 
\begin{equation*}
P=\Big(p\big((\eta,x),(\eta',x')\big)\Big)  
\end{equation*}
can also be described by a sequence of $(\lgr)$-valued random variables $\Zn$. 
More precisely, we write $Z_{n}=(\eta_{n},X_{n})$, where $\eta_{n}$ 
is the random configuration of lamps at time $n$, and $X_{n}$ is the 
random vertex in $\mathsf{G}$ where the lamplighter stands at time $n$. 
In the following, when referring to the lamplighter random walk, 
we shall use the sequence of random variables $\Zn$ with $Z_{n}=(\eta_{n},X_{n})$,
whose transitions are given by the matrix $P$.  

Assume that the lamplighter random walk $\Zn$ has finite first 
moment with respect to the lamplighter metric $d$ defined in \eqref{LamplighterGraphGmetric}, that is, 
\begin{equation*}\label{eq:trpb_base_rw}
\sum_{(\eta,x)\in\lgr}d \big((\textbf{0},o),(\eta,x)\big) p\big((\textbf{0},o),(\eta,x)\big)<\infty. 
\end{equation*}
The process $\Zn$ on $\lgr$ projects onto random processes $\Xn$ 
on $\mathsf{G}$ and $(\eta_{n})$ on the space of configurations $\mathcal{C}$.
Note that the stochastic 
process $(\eta_{n})$ is a sequence of random configurations of lamps, 
but not a Markov chain, since the configuration $\eta_{n}$ at time $n$ 
depends on the entire history of the process up to time $n$. 
The projection $\Xn$ is a random walk on the base graph $\mathsf{G}$ with  
starting point $o\in \mathsf{G}$ and one-step transition probabilities given by
\begin{equation}\label{trpb_base_random_walk}
p_{\mathsf{G}}(x,x')=\sum_{\eta'\in\mathcal{C}}p\big((\mathbf{0},x),(\eta',x')\big),\quad\mbox{ for all } x,x'\in \mathsf{G}.
\end{equation}
The corresponding $n$-step transition probabilities are 
denoted by $p_{\mathsf{G}}^{(n)}(x,x')$, and the transition matrix of $\Xn$
is $P_{\mathsf{G}}=(p_{\mathsf{G}}(x,y))$. The process $\Xn$ will be called
\textit{the base random walk} or \textit{projected random walk} 
on $\mathsf{G}$. 

\textbf{Key assumption}: throught this thesis, we shall
consider lamplighter random walks $\Zn$, such that the projection
$\Xn$ on the transitive base graph $\mathsf{G}$ is transient.

This implies that $\Xn$ leaves every finite subset of $\mathsf{G}$ with probability
 $1$ after a finite time. Transience of the base process is a key assumption,
which leads to the transience of the lamplighter random walk $\Zn$. 
In other words, transience of the random walk $\Xn$ on $\mathsf{G}$ implies that almost every 
path of the original random walk $\Zn$ on $\lgr$ will leave behind a 
certain limit configuration of lamps $\mathsf{G}$, which will not be necessarily 
finitely supported. 

{\textbf{Question}}: Does the limit configuration of lamps 
completely describe the behaviour of the lamplighter random walk 
$\Zn$ at ``infinity''?

The behaviour of the lamplighter random walk at ``infinity'' is  
the main topic of the first part of the work. This comprises the study of 
the \textit{convergence} and of the \textit{Poisson boundary} for lamplighter 
random walks $\Zn$.
\begin{remark}
In the study of the behaviour of lamplighter random walks $\Zn$ 
over lamplighter graphs $\lgr$, with $Z_n=(\eta_n,X_n)$, 
the properties of the base random walk $\Xn$ and the geometry of 
the base graph $\mathsf{G}$ play a crucial role. 
\end{remark}

\subsection{Examples of Transition Matrices}

There are different ways one can define transition matrices over
$\lgr$. One can start with transition matrices over $\mathbb{Z}_2$
and  $\mathsf{G}$, and construct a new transition matrix on $\lgr$ like 
in the sequel.

Consider $P_\mathsf{G}$ and $P_{\mathbb{Z}_{2}}$ transition
matrices on $\mathsf{G}$ and $\mathbb{Z}_{2}$, respectively. One can lift
$P_{\mathsf{G}}$ on $\mathsf{G}$ to $\Bar{P}_{\mathsf{G}}$ on $\lgr$ by setting
\begin{equation*}
\bar{p}_{\mathsf{G}}\big((\eta, x),(\eta', x')\big)=
\begin{cases}
p_{\mathsf{G}}(x, x'), & \text{if}\  \eta=\eta'\\
0, & \text{otherwise}
\end{cases}.
\end{equation*}
One can also lift $P_{\mathbb{Z}_{2}}$ on $\mathbb{Z}_{2}$ to 
$\bar{P}_{\mathbb{Z}_{2}}$ on $\lgr$ by setting
\begin{equation*}
\bar{p}_{\mathbb{Z}_{2}}\big((\eta, x),(\eta', x')\big)=
\begin{cases}
p_{\mathbb{Z}_{2}}\big(\eta(x),\eta'(x')\big), & \text{if}\  x=x'\ \text{and}\ \eta\bigtriangleup  \eta'=\{x\}\\
0, & \text{otherwise}
\end{cases}.
\end{equation*}
Using the embeddings of the transition matrices on $\Z_2$ and $\mathsf{G}$
into $\lgr$, one can construct different ``types'' of lamplighter random
walks (transition matrices) on $\lgr$.

\paragraph{Walk or Switch Random Walk.}
Let $a$ be a parameter with $0<a<1$. Define the transition matrix
$P_{a}$ on $\lgr$ by 
\begin{equation*}
P_{a}=a \bar{P}_{\mathsf{G}}+(1-a)\bar{P}_{\mathbb{Z}_{2}}.
\end{equation*}
The interpretation of $P_{a}$ in lamplighter terms is as follows. 
If the lamplighter stands at $x$ and the current configuration
is $\eta$, then he first tosses a coin. 
If "head" comes up (with probability $a$) then he makes 
a random move according to the probability distribution
$p_{\mathsf{G}}(x, \cdot)$ while leaving the lamps unchanged. 
If "tail" comes up, then he makes no move in the graph $\mathsf{G}$, 
but modifies the state of the lamp where he stands according
to the distribution $p_{\mathbb{Z}_2}\big(\eta(x),\cdot\big)$.

\begin{remark}
If the graph $\mathsf{G}$ is regular (i.e. all vertices have the same degree)
and $P_\mathsf{G}$ and $P_{\mathbb{Z}_{2}}$
are the transition matrices of the simple random walks on $\mathsf{G}$ and
$\mathbb{Z}_{2}$ respectively, then the simple random walk on $\lgr$ 
is given by  $P_{a}$ with 
\begin{equation*}
 a=\dfrac{\deg_{\mathsf{G}}}{\deg_{\mathsf{G}}+2},
\end{equation*}
where $\deg_\mathsf{G}$ is the vertices degree in $\mathsf{G}$.
\end{remark}

\paragraph{Switch-Walk-Switch Random Walk.}
Define a transition matrix $Q$ on $\lgr$ by the following matrix product
\begin{equation*}
Q=\bar{P}_{\mathbb{Z}_{2}}\cdot\bar{P}_{\mathsf{G}}\cdot\bar{P}_{\mathbb{Z}_{2}}.
\end{equation*}
The intuitive interpretation is: if the lamplighter stands at $x$ and the current
configuration of lamps is $\eta$, then he first changes the state
of the lamp at $x$ according to the probability distribution 
$p_{\mathbb{Z}_{2}}\big(\eta(x),\cdot\big )$. 
Then he makes a step to some point $x'\in \mathsf{G}$ according to the 
probability distribution $p_{\mathsf{G}}(x, \cdot)$, and at last,
he changes the state of the lamp at $x'$ according to the 
probability distribution $p_{\mathbb{Z}_{2}}\big(\eta(x'),\cdot\big)$.

The \textit{Switch-Walk-Switch} and \textit{Walk or Switch} lamplighter random 
walks are two basic examples of random walks which are well studied in the
literature. Nevertheless, we are not going to use this specific types
of transition matrices, but we shall instead work with general irreducible
transition matrices $P$ over $\lgr$, such that its projection $P_{\mathsf{G}}$ 
onto $\mathsf{G}$ is a transient random walk.

\section{Lamplighter Groups and Random Walks}

Recall that the set $\Z_{2}$ encodes the intensities 
of lamps $\{0,1\}$. If we endow this set with the operation 
of addition modulo $2$, then it becomes a group. Consider 
now the \textit{group $\Z_{2}$  acting transitively on itself} by 
left multiplication, and identify the set of intensities of 
lamps with the group $\Z_{2}$, such that the state $0$ corresponds 
to the group identity. For simplicity of notation we use $\Z_{2}$ for 
both the set of lamp intensities and the group with two elements acting on it. 

The set $\mathcal C$ of all finitely supported $\Z_2$-valued
configurations on $\mathsf{G}$ becomes then a group with the pointwise operation ``$\oplus$''
\begin{equation*}
(\eta\oplus\eta')(x)=\eta(x)\oplus\eta'(x),
\end{equation*}
taken in the group $\Z_{2}$. The unit element of $\mathcal{C}$ is the
zero configuration $\textbf{0}$, which corresponds to all lamps switched off.

Let $\Gamma$ be a closed subgroup of $\AUT(\mathsf{G})$ 
which acts transitively on $\mathsf{G}$. We do not require that $\mathsf{G}$ is
a Cayley graph of $\Gamma$. For instance, $\Gamma$ can also be a 
non-discrete group like in Section \ref{subsec:one_fixed_end}.

\begin{definition}The \index{group!wreath product}\emph{wreath product} of the groups $\Z_{2}$ and $\Gamma$,
denoted by $\Z_{2}\wr\Gamma$, is a 
semidirect product of $\Gamma$ and the direct sum $\sum_{\gamma'\in \Gamma}\Z_{2}$ of copies of $\Z_{2}$ 
indexed by $\Gamma$, where every $\gamma\in\Gamma$ acts
on $\sum_{\gamma'\in\Gamma}\Z_{2}$ by translation $T_{\gamma}$ defined as 
\begin{equation*}
(T_{\gamma}\phi)(\gamma')=\phi(\gamma^{-1}\gamma'),\quad \text{ for all } \gamma'\in \Gamma,\phi\in\mathcal{C}.
\end{equation*}
\end{definition}
The elements of $\Z_{2}\wr\Gamma$ are pairs of the form 
$(\phi,\gamma)\in\mathcal C\times \Gamma$, where $\phi$ 
represents a finitely supported configuration of lamps 
and $\gamma\in\Gamma$. A group operation on $\Z_{2}\wr\Gamma$, 
denoted by ``$\cdot$'' is given by
\begin{equation*}
(\phi,\gamma)\cdot(\phi',\gamma{'})=(\phi\oplus T_{\gamma}\phi',\gamma\gamma'),
\end{equation*}
where $\gamma,\gamma'\in\Gamma$, $\phi,\phi'\in\mathcal C$, and $\oplus$ 
is the componentwise addition modulo $2$. 

We shall call $\lgrp$ 
together with this operation the \index{group!lamplighter group}\textit{lamplighter group} 
over $\Gamma$. The group identity is $(\textbf{0},e)$, where 
$\textbf{0}$ is the zero configuration and $e$ is the unit 
element in $\Gamma$. Finally, define an action of 
elements $(\phi,\gamma)\in\lgrp$ on $\lgr$ by
\begin{equation}\label{grp_action_lrw}
 (\phi,\gamma)(\eta,x)=(\phi\oplus T_{\gamma}\eta,\gamma x),\quad \mbox{ for all } (\eta,x)\in\lgr.
\end{equation}
This action preserves the neighbourhood relation defined 
in \eqref{neigh_relation}. Therefore
\begin{equation*}
(\phi,\gamma)\in \AUT(\lgr), 
\end{equation*}
that is, the lamplighter group $\lgrp$ is a subgroup of $\AUT(\lgr)$. 
Now we have both an \emph{underlying geometric structure} 
$\lgr$ (the lamplighter graph) and an \emph{action of the group} $\lgrp$ 
(the lamplighter group) on it. When $\mathsf{G}$ is a Cayley graph 
of $\Gamma$ with respect to some finite generating set, then these two
structures can be identified, but since we work with more general
groups $\Gamma\subset \AUT(\mathsf{G})$, it is important to distinguish
between the lamplighter graph $\lgr$ and the lamplighter group $\lgrp$.

Since $\Z_{2}$ acts transitively on itself and $\Gamma$ 
acts transitively on $\mathsf{G}$, it is straightforward to see 
that the wreath product $\lgrp$ acts transitively on 
$\lgr$, since it is by construction a subgroup of $\AUT(\lgr)$. 
For more details, see also {\sc Woess}~\cite{WoessNote2005}.  

\textbf{Assumption}: suppose that $\Gamma\subset \AUT(\mathsf{G})$ is chosen
such that $\lgrp$ is a subgroup of $\big(\AUT(\lgr),P\big)$, 
that is, the transition probabilities $p(\cdot,\cdot)$ of $\Zn$ are 
invariant with respect to the action of the  group $\lgrp$. In other words,
$\Zn$ is a homogeneous random walk. This means that
for all $g=(\phi,\gamma)\in\lgrp$, we have
\begin{equation*}
p\Big(g(\eta,x),g(\eta',x')\Big)= p\Big((\eta,x),(\eta',x')\Big),\quad \text{ for all } (\eta,x),(\eta',x')\in\lgr.
\end{equation*}
\begin{corollary}\label{cor:xn_space_hom}
The factor chain $\Xn$ is also a homogeneous random walk on $\mathsf{G}$,
i.e $\Gamma\subset \AUT(\mathsf{G},P_{\mathsf{G}})$.
\end{corollary}
\begin{proof}
Definition \eqref{trpb_base_random_walk} of the 
transition probabilities of the random walk $\Xn$ on $\mathsf{G}$
implies that
\begin{equation*}
p_{\mathsf{G}}(\gamma x,\gamma x')=\sum_{\eta\in\mathcal{C}}p\big((\textbf{0},\gamma x),(\eta,\gamma x')\big),\quad\text{ for all }x,x'\in \mathsf{G},\gamma\in\Gamma.
\end{equation*}
Since $\lgrp$ acts transitively on $\lgr$, it follows that there exists
a configuration $\eta_1\in\mathcal{C}$ such that
\begin{equation*}
\eta=\mathbf{0}\oplus T_{\gamma}\eta_1.
\end{equation*}
Also, we can write $\mathbf{0}=\mathbf{0}\oplus T_{\gamma}\mathbf{0}$.
Using the action \eqref{grp_action_lrw} of the lamplighter group $\lgrp$ on 
the lamplighter graph $\lgr$, we can write 
\begin{equation*}
(\mathbf{0},\gamma x)=(\mathbf{0},\gamma)(\mathbf{0},x) \text{ and } (\eta,\gamma x')=(\mathbf{0},\gamma) (\eta_1,x').
\end{equation*}
Thus
\begin{equation*}
p\big((\mathbf{0},\gamma x),(\eta,\gamma x')\big)=p\big((\mathbf{0},\gamma)(\mathbf{0},x),(\mathbf{0},\gamma) (\eta_1,x')\big).
\end{equation*}
This, together with the invariance of the transition 
probabilities of the lamplighter random walk yields 
\begin{equation*}
p_{\mathsf{G}}(x,x')=p_{\mathsf{G}}(\gamma x, \gamma x'),\quad \text{for all } x,x'\in \mathsf{G},\gamma\in\Gamma,
\end{equation*}
which proves the claim.
\end{proof}

Recall now a simple fact about the asymptotic configuration 
size of the lamplighter random walk, which will be needed later.
\begin{lemma}\label{lem:rate_of_escape_rw}
Let $\Zn$ be a random walk with finite first moment on $\lgr$, 
with $Z_n=(\eta_n,X_n)$. Then there exists a constant $C\geq0$, such that
\begin{equation*}
 lim_{n\to\infty}\dfrac{|\supp(\eta_n)|}{n}=C.
\end{equation*}
In other words, the number of lamps which are turned on increases asymptotically at linear speed. 
\end{lemma}
\begin{proof}
Kingman's subadditive ergodic theorem \cite{Kingman1968} applied to the sequence 
$|\supp(\eta_i)|$, which is subadditive, yields the desired result.
\end{proof}

\begin{remark}
The constant $C$ was studied for a large class of lamplighter 
random walks over discrete graphs and groups. It is greater than zero if 
and only if the factor chain $\Xn$ is transient.
\end{remark}

\paragraph*{Induced Random Walks on $\lgrp$.}
\label{par:correspondence_p_mu}
Let $\nu$ be a probability measure on $\lgrp$, which determines
the right random walk $(\lgrp,\nu)$, and which is uniquely 
induced by the transition matrix $P$
of $\Zn$ like in equation \eqref{eq:correspondence_rw}.

The measure $\mu$ on $\Gamma$, which is induced by the 
transition matrix $P_{\mathsf{G}}$ of the random walk $\Xn$ on $\mathsf{G}$,
is then given by
\begin{equation*}
\mu(x)=\sum_{\eta\in\mathcal{C}}\nu\big((\eta,x)\big).
\end{equation*}
Indeed, this follows from the fact that the transition probabilities of $\Xn$
on $\mathsf{G}$ are projections of the transition probabilities 
of $\Zn$ on $\lgr$ as in equation \eqref{eq:trpb_base_rw}.

For the correspondence between $\mu$ and $P_{\mathsf{G}}$, and $\nu$ and $P$, we shall 
use the notation $\mu\leftrightarrow P_{\mathsf{G}}$ and $\nu\leftrightarrow P$, respectively.

\chapter{Convergence to the Boundary}\label{chap:conv_lrw}

This chapter is devoted to the study of convergence (in a sense
to be specified) of homogeneous lamplighter random walks $\Zn$
on graphs $\lgr$, with $Z_n=(\eta_n,X_n)$, given that the 
base random walk $\Xn$ is transient on $\mathsf{G}$. We emphasize
that the geometry of $\mathsf{G}$ and the action of $\Gamma\subset \AUT(\mathsf{G})$ 
play an important role in the study of the behaviour of $\Zn$,
as $n$ tends to infinity.

We are interested in transitive, infinite base graphs $\mathsf{G}$ endowed with 
a ``rich'' boundary $\partial \mathsf{G}$. Using $\partial \mathsf{G}$, we construct 
a boundary $\Pi$ for the lamplighter graph $\lgr$. We then prove 
that $\Zn$ converges to some random variable $Z_{\infty}$ with
values in the boundary $\Pi$ almost surely, 
under some natural assumptions on the base random walk $\Xn$.
Finally, in the next chapters, the results obtained here will be 
applied to specific base graphs $\mathsf{G}$: graphs with infinitely many
ends, hyperbolic graphs, and Euclidean lattices. 

For lamplighter random walks over Euclidean lattices, the results
proved here were earlier obtained by {\sc Kaimanovich}~\cite{Kaimanovich1991}.
For the sake of completeness, we just show how to apply our results
in this case.

\section{The Boundary of the Lamplighter Graph}\label{sec:bndr_lgr}

\paragraph{The Boundary of the Base Graph.} In order to
``build'' a boundary for the lamplighter graph $\Z_2\wr \mathsf{G}$,
we start with the base graph $\mathsf{G}$, which is assumed to be infinite,
locally finite, connected and transitive. Let $d(\cdot,\cdot)$ be the
discrete graph metric on $\mathsf{G}$, and consider
\begin{equation*}
\widehat{\mathsf{G}}=\mathsf{G}\cup\partial \mathsf{G} 
\end{equation*}
an extended space of $\mathsf{G}$, not necessarily compact, with ideal 
\index{graph!boundary of a graph}\textit{boundary} $\partial \mathsf{G}$ 
(the set of points at infinity), such that $\widehat{\mathsf{G}}$ is \emph{compatible} 
with the group action $\Gamma$ on $\mathsf{G}$. Here, by compatibility between $\Gamma$
and $\widehat{\mathsf{G}}$ we mean that the action of $\Gamma$ on $\mathsf{G}$ extends 
to an action on $\widehat{\mathsf{G}}$ by 
homeomorphisms. Recall that $\Gamma\subset \AUT(\mathsf{G})$.

\paragraph{Convergence to the Boundary.}
In order to introduce the notion of convergence of a random walk $\Xn$ 
on $\mathsf{G}$ to the boundary $\partial \mathsf{G}$, recall first the definition 
of the trajectory space $\Omega$ of $\Xn$
\begin{equation*}
 \Omega=\mathsf{G}^{\Z_+}=\{\omega=(x_0,x_1,x_2, \ldots ):\ x_n\in \mathsf{G} \text{ for all } n\geq 0\}.
\end{equation*}
An element $\omega\in\Omega$ represents a possible evolution, that is, 
a possible sequence of points visited one after the other by $\Xn$. 
For $\omega=(x_0,x_1,x_2,\ldots )\in \Omega$ and $n\geq 0$, define 
the projections $X_n(\omega)=x_n$. 

Suppose that we have the boundary $\partial \mathsf{G}$ and the extended 
space $\widehat{\mathsf{G}}$ of $\mathsf{G}$ defined in \eqref{sec:bndr_lgr} and set
\begin{equation*}
\Omega_{\infty}=\{\omega\in\Omega: X_{\infty}(\omega)=\lim_{n\to\infty} X_n(\omega) \in\partial \mathsf{G} 
\text{ exists in the topology of } \widehat{\mathsf{G}}\}.
\end{equation*}
\begin{definition}
We say that the random walk $\Xn$ on $\mathsf{G}$ \emph{converges to the boundary
$\partial \mathsf{G}$} if  
\begin{equation*}
\mathbb{P}_{x}[\Omega_{\infty}]=1,\quad  \text{ for every } x\in \mathsf{G}. 
\end{equation*}
\end{definition}
For the convergence of $\Xn$ to $X_{\infty}\in \partial\mathsf{G}$, 
the notation $X_n\to X_{\infty}$ will be used. 

The random variable $X_\infty$ is measurable with respect to the Borel 
$\sigma$-algebra of $\partial \mathsf{G}$. If $\Xn$ converges
to the boundary $\partial \mathsf{G}$, the hitting distribution is the measure
$\mu_{\infty}$ defined for Borel sets $B\subset\partial \mathsf{G}$ by
\begin{equation*}
 \mu_{\infty}(B)=\mathbb{P}[X_{\infty}\in B|X_0=o].
\end{equation*}

\begin{definition}\label{definition:projectivity}
The boundary $\partial \mathsf{G}$ is called 
\index{graph!boundary of a graph!projective}\emph{projective}, if the 
following holds for sequences $(x_n),(y_n)$ of vertices in $\mathsf{G}$: 
if $(x_n)$ converges to a boundary point $\mathfrak{u}\in\partial \mathsf{G}$ and
\begin{equation*}
\sup_{n}d(x_n,y_n)<\infty ,
\end{equation*}
then also $(y_{n})$ converges to $\mathfrak{u}$.
\end{definition} 

For our results, a somehow weaker property of the boundary $\partial \mathsf{G}$ is needed.

\begin{definition}\label{def:weak_proj}
The boundary $\partial \mathsf{G}$ is called 
\index{graph!boundary of a graph!weakly projective}\emph{weakly projective} 
if the following holds for sequences $(x_{n}),(y_{n})$ of vertices in 
$\mathsf{G}$: if $(x_{n})$ converges to $\mathfrak{u}\in \partial \mathsf{G}$ and 
\begin{equation*}
\frac{d(x_{n},y_{n})}{d(x_{n},o)}\to 0,\quad \text{ as }n\to\infty, 
\end{equation*} 
then $(y_{n})$ converges also to $\mathfrak{u}$.
\end{definition}
\begin{remark}
Note that a weakly projective boundary is also projective, 
but the other way round does not necessary hold.
\end{remark} 
Indeed, when the sequence of vertices $(x_n)$ accumulates at some 
boundary point $\mathfrak{u}\in\partial \mathsf{G}$ and 
\begin{equation*}
d(x_n,y_n)\approx \log n \text{ and } d(x_n,o)\approx c n, \text{ for some }c>0,
\end{equation*}
then the requirements for $\partial \mathsf{G}$ to be a weakly projective 
boundary are satisfied, but not those for a projective boundary. 
We shall work with weakly projective boundaries.

\paragraph{The Boundary of the Lamplighter Graph $\lgr$.}  The natural 
compactification of the set of finitely supported configurations $\mathcal{C}$ 
in the topology of pointwise convergence is the set $\widehat{\mathcal{C}}$ of all, 
finitely or infinitely supported configurations. 

Since the vertex set of the lamplighter graph $\lgr$ is $\mathcal{C}\times \mathsf{G}$, the space 
\begin{equation*}
\partial (\lgr)=(\widehat{\mathcal{C}}\times \widehat{\mathsf{G}})\setminus (\mathcal{C} \times \mathsf{G}) 
\end{equation*}
is a \textit{natural boundary} at infinity for $\lgr$. Let us write
\begin{equation*}
\widehat{\lgr}=\widehat{\mathcal{C}} \times \widehat{\mathsf{G}}. 
\end{equation*}
The boundary $\partial (\lgr)$ contains all pairs $(\zeta,\mathfrak{u})$, 
where $\mathfrak{u}\in\partial \mathsf{G}$ and $\zeta$ is a finitely or infinitely 
supported configuration. This boundary is so ``rich'' that it gives us 
plenty of information on the behaviour of the lamplighter random walks
at infinity.

\section{Convergence of LRW}

Convergence of LRW $\Zn$ on $\lgr$ follows mainly from the convergence 
of the base random walk $\Xn$ on $\mathsf{G}$. For the time being,
we are still working with general random walks $\Xn$ on $\mathsf{G}$,
where $\mathsf{G}$ is a transitve graph. In order to get some information
about the random walk $\Zn$ on $\lgr$, we have to know something
about $\Xn$.  For this reason, some assumptions on
$\Xn$ and $\mathsf{G}$ are needed.
\begin{assumption}\label{assumptions_brw}
Assume that:
\begin{description}\label{ass}
 \item[(A1)] $\Xn$ has finite first moment on $\mathsf{G}$.
 \item[(A2)] $\Xn$ converges to $\partial \mathsf{G}$, with 
hitting distribution $\mu_{\infty}$.
 \item[(A3)] $\partial \mathsf{G}$ is weakly projective.
\end{description}
\end{assumption}
These assumptions are not very restrictive. We shall give several 
examples of graphs $\mathsf{G}$ and random walks on them where these assumptions hold. 

Given that the base random walk $\Xn$ on $\mathsf{G}$ converges to 
the boundary $\partial \mathsf{G}$, we prove that the lamplighter random
walk $\Zn$ on $\lgr$ converges to random variables with values
on the boundary $\partial(\lgr)$. This boundary is still too big
for our purposes, that is, it contains many points towards $\Zn$
converges with probability $0$. For this reason, let us define a
``smaller'' boundary $\Pi$ for the lamplighter graph, which is
still dense in $\partial(\lgr)$, and we shall show that the
random walk $\Zn$ converges with probability $1$ to a random
variable with values in $\Pi$. Define the subset  $\Pi$ of
$\partial(\lgr)$ by
\begin{equation}\label{eq:pi_boundary}
\Pi =\bigcup_{\mathfrak{u} \in \partial \mathsf{G}}\mathcal{C}_{\mathfrak{u}}\times \{\mathfrak{u}\}, 
\end{equation}
where the set $\mathcal{C}_{\mathfrak{u}}$ consists of all
configurations $\zeta$, which are either finitely supported,
or infinitely supported with $\supp(\zeta)$  accumulating only
at $\mathfrak{u}$. The set $\mathcal{C}_{\mathfrak{u}}$ is dense
in $\widehat{\mathcal{C}}$ because 
$\mathcal{C}\subset\mathcal{C}_{\mathfrak{u}}$ and $\mathcal{C}$
is dense in $\widehat{\mathcal{C}}$. Hence, $\Pi$ is also dense in $\partial(\lgr)$.

The action of the group $\lgrp$ on the lamplighter graph
$\lgr$ extends to an action on $\widehat{\lgr}=\widehat{\mathcal{C}} \times \widehat{\mathsf{G}}$
by homeomorphisms and leaves the Borel subset $\Pi\subset\partial(\lgr)$ invariant.
If we take $(\phi,\gamma)\in \lgrp$ and $(\zeta,\mathfrak{u}) \in \Pi$, then
\begin{equation*} \label{action}
(\phi,\gamma)(\zeta,\mathfrak{u})=(\phi \oplus T_{\gamma}\zeta,\gamma \mathfrak{u}).
\end{equation*}
If $\mathfrak{u}\in\partial \mathsf{G}$ and $\zeta$ is finitely supported or accumulates
only at $\mathfrak{u}$, then $T_{\gamma}\zeta$ can accumulate at most at
$\gamma\mathfrak{u}$. Also the configuration $\eta\oplus T_{\gamma}\zeta$
accumulates again at most at $x\mathfrak{u}$ because $\eta$ is finitely supported,
so that adding $\eta$ modifies $T_{\gamma}\zeta$ only in finitely many points.

\begin{definition}We shall say that a sequence of lamp configurations 
$\eta_n$ \emph{converges} to a configuration $\eta_{\infty}$, if
\begin{equation*}
\lim_{n\to\infty}\eta_n(x)=\eta_{\infty}(x),\quad \text{ for all } x\in \mathsf{G},
\end{equation*}
i.e., for all $x\in \mathsf{G}$ the sequence $\eta_n(x)$ stabilizes.
\end{definition}

For a special case, where the lamplighter changes the lamps configuration
only at the current vertex, it is clear that the lamplighter random walk 
converges to a point in $\Pi$. Indeed, by Assumption \ref{assumptions_brw} (A2), 
the base random walk $\Xn$ converges to a random element $X_{\infty}\in\partial \mathsf{G}$. 
Since only the states of lamps which are visited can be modified, and by 
transience every vertex is visited only finitely many times, after some time every 
vertex is left forever and the state of the lamp sitting there remains unchanged. 
Therefore the random configuration $\eta_n$ must converge pointwise to a random
configuration which accumulates at $X_{\infty}$.

We shall prove the convergence for general homogeneous
random walks on $\lgr$, not only restricted to the situation when the 
configuration can be changed at the current position.
This was also proved by {\sc Karlsson and Woess}~\cite{KarlssonWoess2007} for
a class of lamplighter random walks over homogeneous trees. The following
result is a generalization of \cite[Theorem 2.9]{KarlssonWoess2007}
for lamplighter random walks $\Zn$ over general transitive base graphs $\mathsf{G}$. 
In {\sc Sava}~\cite{SavaPoisson2010} this was proved for lamplighter random walks on 
discrete groups.

Assume that the random walk $\Xn$ on the transitive base graphs 
$\mathsf{G}$ satisfies Assumption \ref{assumptions_brw}. 
Then the following holds for homogeneous lamplighter random walks $\Zn$.

\begin{theorem}\label{thm_conv_lrw_general_graphs}
Let $\Zn$ be an irreducible and homogeneous random walk with finite first 
moment on $\lgr$. Then there exists a $\Pi$-valued random variable 
$Z_{\infty}=(\eta_{\infty},X_{\infty})$, such that $Z_n\to Z_{\infty}$
almost surely in the topology of $\widehat{\lgr}$, 
for every starting point $(\eta_0,x_0)$. Moreover, 
the distribution of $Z_{\infty}$ is a continuous measure on $\Pi$.
\end{theorem}
\begin{proof}
Without loss of generality, we may suppose that the starting point
is $(\textbf{0},o)$, where $\textbf{0}$ is the trivial (zero)
lamps configuration and $o\in \mathsf{G}$ some reference vertex, whose 
choice is irrelevant by the transitivity of $\mathsf{G}$. 

The random walk $\Zn$ is homogeneous, and by Corollary 
\ref{cor:xn_space_hom} the factor chain $\Xn$ is also homogeneous. 
Also, the factor chain $\Xn$ on $\mathsf{G}$ is transient, and it converges almost 
surely to a random variable $X_{\infty}\in\partial \mathsf{G}$ by Assumption
\ref{assumptions_brw}.

Now, assume that $\Zn$ has finite first moment on $\lgr$. 
Then, the configuration $\eta_i$ of lamps at time $i$ can be
obtained by modifying the previous configuration $\eta_{i-1}$ 
in a finite number of vertices in $\mathsf{G}$. This implies that there exists 
a finitely supported configuration $\phi_i$, such that
\begin{equation*}
 \eta_i=\phi_i\oplus\eta_{i-1},\quad \text{ for all } i=1,2,\ldots ,n.
\end{equation*}
The configuration $\phi_i$ is zero in all points which are not touched 
by the random walker. Let now $(y_{n})$ be an unbounded 
sequence of elements in $\mathsf{G}$, with $y_{n}\in\supp(\phi_n)$.
Thus, $y_{n}$ is a sequence of vertices in $\mathsf{G}$ where the lamp is switched on.
Since $\Xn$ has finite first moment on $\mathsf{G}$, the following holds with probability $1$:
\begin{equation*}
\frac{d(y_{n},X_{n})}{n}\to 0,\quad\text{ as }n\to\infty.
\end{equation*}
\textit{Kingman's subadditive ergodic theorem \ref{thm:kingman}} 
(see also {\sc Kingman}~\cite{Kingman1968}) implies that there exists 
finite constant  $l>0$, such that
\begin{equation*}
\frac{d(X_{n},o)}{n}\to l,\quad \text{as}\ n\to\infty,\ \text{almost surely}.
\end{equation*}
Making use of the previous two equations and the triangle inequality, we get
\begin{equation}\label{eq:weak}
\frac{d(X_{n},y_{n})}{d(X_{n},o)}\to 0,\quad \text{as } n\to\infty. 
\end{equation}

Recall that by Assumption \ref{assumptions_brw} we have $X_n\to X_{\infty}$.
By the weakly projectivity of $\partial \mathsf{G}$
and from equation \eqref{eq:weak}, it follows that
$(y_{n})$ converges to $X_{\infty}$. Observe that 
\begin{equation*}
\supp(\eta_{n})\subset \bigcup_{i=1}^{n}\supp(\phi_i),
\end{equation*}
which is a union of finite sets. Since the unbounded 
sequence $y_{n}\in\supp(\phi_n)$ 
converges to $X_{\infty}$, it follows that 
$\supp(\eta_{n})$ must converge to $X_{\infty}$. 
That is, the random configuration $\eta_{n}$ converges pointwise 
to a limit configuration $\eta_{\infty}$, which accumulates at $X_{\infty}$ 
and $Z_{n}=(\eta_{n},X_{n})$ converges to a random element 
$Z_{\infty}=(\eta_{\infty},X_{\infty})\in\Pi$.

When the limit distribution of $\Xn$ is a continuous measure 
on $\partial \mathsf{G}$ (i.e., it carries no point mass), then the same 
is true for the limit distribution of $\Zn$ on $\Pi$. Indeed, 
supposing that there exists some single point in $\Pi$ with 
non-zero hitting probability measure, then a contradiction 
arises since one can find some single point in $\partial \mathsf{G}$ 
with non-zero measure. This is not possible because of the continuity 
of the limit distribution of $\Xn$.

On the other side, when the limit distribution of $\Xn$ is not continuous on 
$\partial \mathsf{G}$, one can use Borel-Cantelli lemma in order to prove that the limit 
distribution of $\Zn$ is still continuous.
\end{proof}

\chapter{Poisson Boundary of LRW}\label{chap:poiss_bndr_lrw}

The \textit{Poisson boundary} of a random walk is a measure space which
describes the stochastically significant behaviour of its paths
at infinity. In this chapter we present a method to identify the
Poisson boundary of lamplighter random walks over graphs
$\lgr$, given that the base walk $\Xn$ over $\mathsf{G}$ 
is transient and satisfies some suitable assumptions.
This method is called the \textit{Half-Space Method}.
The base graph $\mathsf{G}$ will be then replaced in the following chapters
by some specific graphs, and the method described here will 
be applied.

The Poisson boundary of lamplighter random walks over groups $\lgrp$, with 
$\Gamma$ a discrete group endowed with a rich boundary, was determined in 
{\sc Sava}~\cite{SavaPoisson2010}. 

For more information on the Poisson boundary, 
the reader is invited to have a look
at the introductory and complex papers of 
{\sc Kaimanovich}~\cite{Kaimanovich1991}, \cite{Kaimanovich1995}
and \cite{Kaimanovich2000}. The description of the Poisson boundary 
of lamplighter random walks over Euclidean lattices $\Z^d$, with $d\geq 5$, 
such that the base walk has zero drift was an open problem
for a long time, and has been recently solved by {\sc Erschler}~\cite{Erschler2010}.
For the relation between the Poisson boundary and the linear drift
of a random walk, see also {\sc Karlsson and Ledrappier}~\cite{KarlssonLedrappier2007}.
For other problems and methods related to the determination of the Poisson boundary,
see {\sc Ballmann and Ledrappier} \cite{BallmannLedrappier}, and {\sc Ledrappier} \cite{Ledrappier1985}.

\section{Preliminaries}

The \textit{Poisson boundary} of a Markov chain $(\mathsf{G},P)$  is defined
as the space of ergodic components of the time shift in the path space.
Under natural assumptions on the transition matrix $P$ on $\mathsf{G}$, there exists 
a measure space $(\Lambda,\lambda)$, such that the \textit{Poisson formula}
\begin{equation*}
 h_{\varphi}(x)=\int_{\Lambda}\varphi\ d\lambda_x
\end{equation*}
states an isometric isomorphism between the Banach space $H^{\infty}(\mathsf{G},P)$ of
bounded harmonic functions on $\mathsf{G}$ with sup-norm and the space $L^{\infty}(\Lambda,\lambda)$
of $\lambda$-measurable functions on $\Lambda$. The space $(\Lambda,\lambda)$ is 
called the \index{Markov chain!Poisson boundary}\textit{Poisson boundary} 
of the Markov chain $(\mathsf{G},P)$. \textit{Triviality of the Poisson boundary} is 
equivalent to the absence of non-constant bounded harmonic functions for 
the pair $(\mathsf{G},P)$. This is the so-called  
\index{Markov chain!Poisson boundary!Liouville property}\textit{Liouville property}. 

The Poisson formula characterizes the Poisson boundary up to a measure theoretical 
isomorphism. It also has a topological interpretation in terms of the Martin boundary, where
it consists of the set of possible limits of the Markov chain at the boundary together with the 
family of corresponding harmonic hitting distributions. Nevertheless, 
we emphasize that the Poisson boundary is a measure-theoretical object 
and all objects connected with the Poisson boundary are defined modulo subsets of 
measure $0$. For a detailed description, see {\sc Kaimanovich} \cite{Kaimanovich1992}. 

Recall that if $\Gamma$ is a group which acts transitively on $\mathsf{G}$ and 
leaves the transition operator $P$ on $\mathsf{G}$ invariant, that is, 
if $\Gamma\subset \AUT(\mathsf{G},P)$,
then there exists a measure $\mu$ on $\Gamma$, which determines the right 
random walk $(\Gamma,\mu)$. Moreover, the measure $\mu$ is uniquely induced 
by the transition probabilities $P$ of the 
pair $(\mathsf{G},P)$ as in equation \eqref{eq:correspondence_rw}. Recall the notation 
$P\leftrightarrow\mu$ for the respective correspondence.

Homogeneous Markov operators are intermediate between random walks on
countable groups and random walks on general locally compact groups.
Although the state space $\mathsf{G}$ is countable, the Poisson boundary of
the Markov chain $(\mathsf{G},P)$ is isomorphic with the Poisson boundary of
the induced random walk $(\Gamma,\mu)$ on $\Gamma\subset \AUT(\mathsf{G},P)$,
which is not necessarily discrete.

\begin{proposition}\label{prop:coresp_rw_mc}
If $\Gamma\subset \AUT(\mathsf{G},P)$, then the Poisson boundary of the random 
walk $(\Gamma,\mu)$ coincides with the Poisson boundary of the pair $(\mathsf{G},P)$.
\end{proposition}
For the proof see {\sc Kaimanovich and Woess}~\cite[Proposition 3.1]{KaimanovichWoess2002}.
\begin{definition}\label{def:mu_bndr}
A $\mu$-boundary for the random walk $(\Gamma,\mu)$ is a space $(B,\sigma)$
with the following properties:
\begin{enumerate}[(a)]
\item every path of the random walk converges to a limit in $B$ with 
hitting distribution $\sigma$.
\item the measure $\sigma$ is $\mu$-harmonic, i.e. $\mu\ast\sigma=\sigma$.
\item $\Gamma$ acts on $B$ by measurable bijections.
\end{enumerate}
\end{definition}
Due to the coincidence of the Poisson boundaries of $(\Gamma,\mu)$
and $(\mathsf{G},P)$, it follows that a $\mu$-boundary for $(\Gamma,\mu)$ is also 
a $\mu$-boundary for $(\mathsf{G},P)$.

If $\mathsf{G}$ is embedded into a topological space $B$, and every 
path of the Markov chain converges to a limit in $B$, then the space $B$
with the hitting measure $\sigma$ on it, is a \textit{quotient 
of the Poisson boundary}. Such quotients are 
\index{Markov chain!Poisson boundary!$\mu$-boundary}\textit{$\mu$-boundaries}.
Moreover, the topology of $B$ is irrelevant, since any projection
from the path space $\Omega=\mathsf{G}^{\mathbb{Z}_{+}}$ onto the space $(B,\sigma)$
gives rise to a $\mu$-boundary.

\textit{The Poisson boundary is the maximal $\mu$-boundary}. Here,
we mean maximality in a measure theoretic sense, i.e., there is no way
(up to measure $0$) of further splitting the boundary
points of this compactification. Therefore, 
the problem of identifying the Poisson boundary consists of two parts:
\begin{enumerate}
\item To find a $\mu$-boundary $(B,\sigma)$. This space is a \textit{priori}
just a quotient of the Poisson boundary. 
\item To show that this boundary is maximal, i.e., is isomorphic to 
the whole Poisson boundary. 
\end{enumerate}
The identification of a $\mu$-boundary can be done in geometric or 
combinatorial terms. Throughout this thesis, we shall consider a geometric 
approach in order to prove the maximality of a $\mu$-boundary. 
In Section \ref{sec:zero_drift}, the Poisson boundary
is described by a measure theoretical method, namely
the correspondence between the tail $\sigma$-algebra of a random
walk and its Poisson boundary is used. For the geometric approach, 
we shall use one very nice criterion called \textit{Strip Criterion}, 
developed by Kaimanovich in \cite{Kaimanovich2000}. This ``strip''
or ``bilateral'' approximation is inspired by the use of  
bilateral geodesics in cocompact rank $1$ Cartan-Hadamard manifolds 
by {\sc Ledrappier and Ballmann} \cite{BallmannLedrappier}.

There is a second criterion called \textit{Ray Criterion},
due again to {\sc Kaimanovich} \cite{Kaimanovich2000}, which
can also be used in the identification of the Poisson boundary,
and which will be stated below for sake of completeness. 
These criteria are based on entropies of conditional random walks,
and require an approximation of the sample paths of the random 
walk in terms of their limit behaviour.

These criteria allow to identify the Poisson boundary
with natural boundaries for several classes of graphs and groups
with \textit{hyperbolic properties}: hyperbolic graphs (or more
generally, Cayley graphs of groups of isometries of Gromov 
hyperbolic spaces), graphs with infinitely many ends, and some
other semi-direct and wreath products. All these graphs are 
endowed with natural and nice rich geometric boundaries, which will
be explained in what follows. Moreover, it is known that sample paths 
of the Markov chains on these graphs converge to natural boundaries.

Even if the determination of the Poisson boundary in this thesis is 
done by applying the \textit{Strip Criterion}, it is instructive to state 
both criteria here.

\begin{theorem}\textbf{[Ray Criterion]}\label{thm:ray_crit}
Let $P$ be a homogeneous Markov operator with finite first moment on $\mathsf{G}$
and let $(B,\sigma)$ be a $\mu$-boundary. If there exists a sequence
of measurable maps $R_n:B\rightarrow \mathsf{G}$, such that
\begin{equation*}
 d(X_n,R_n(X_{\infty}))=o(n),
\end{equation*}
for almost every path of the random walk $\Xn$ (with transition operator $P$), 
then $(B,\sigma)$ is the whole Poisson boundary of $\Xn$.
\end{theorem}

For the second criterion, we shall assume that simultaneously with a 
$\mu$-boundary $(B_{+},\sigma_{+})$ we are also given a $\check{\mu}$-boundary 
$(B_{-},\sigma_{-})$ of the reversed Markov operator $\check{P}$.
This criterion is symmetric with respect to the time reversal and leads
to a simultaneous identification of the Poisson boundaries of
$(\mathsf{G},P)$ and $(\mathsf{G},\check{P})$, respectively. For the definition of the
transition probabilities of $(\mathsf{G},\check{P})$, see \eqref{eq:reversed_rw}.

\begin{theorem}\textbf{[Strip Criterion]}\label{thm:strip_crit}
Let $P$ be a homogeneous Markov operator with finite first moment 
on $\mathsf{G}$ and let $(B_{+},\sigma_{+})$, $(B_{-},\sigma_{-})$ be a 
$\mu$- and a $\check{\mu}$- boundary, respectively. 
If there exists a measurable $\Gamma$-equivariant map $S$ assigning 
to almost every pair of points $(b_{-},b_{+})\in B_{-} \times B_{+}$ a 
non-empty ``strip'' $ S(b_{-},b_{+})\subset \mathsf{G}$, such that, for the ball 
$B(o,n)$ of  radius $n$ in the metric of $\mathsf{G}$,
\begin{equation}\label{eq:subexp_strip_growth}
\frac{1}{n}\log| S(b_{-},b_{+})\cap B(o,n) | \to 0 ,\ \text{ as }\ n\to\infty,
\end{equation}
for $(\sigma_{-} \times \sigma_{+})$-almost every 
$(b_{-},b_{+})\in B_{-} \times B_{+}$, then $(B_{+},\sigma_{+})$ and 
$(B_{-},\sigma_{-})$ are the Poisson boundaries of the Markov chains
$(\mathsf{G},P)$ and $(\mathsf{G},\check{P})$, respectively. 
\end{theorem}

Recall that $\Gamma\subset \AUT(\mathsf{G},P)$.
Equivariance of the strip $S(b_{-},b_{+})\subset \mathsf{G}$ with respect to the group 
action $\Gamma$, means that, for all $\gamma\in\Gamma$,
\begin{equation*}
 \gamma S(b_{-},b_{+})=S(\gamma b_{-},\gamma b_{+}).
\end{equation*}
In most of the cases, the equivariance of the strip is 
very easy to prove. The ``harder'' part of the theorem is to prove
the subexponential growth of the chosen strip.
The ``thinner'' the strips $S(b_{-},b_{+})$, the larger the class of
Markov operators for which condition \eqref{eq:subexp_strip_growth}
is satisfied. This means that the sample paths of the random walk $\Xn$
on $\mathsf{G}$ go to infinity ``faster''.

In some cases, the existence of such strips is almost evident, whereas 
checking the \textit{Ray Criterion} may be rather complicate, or in some
cases it fails. However, there are also some situations where the 
\textit{Ray Criterion} is more helpful than the \textit{Strip Criterion}.
The ray criterion provides more information than the strip criterion about
the behaviour of sample paths of the random walk, and can also
be useful for other issues than the identification of the Poisson boundary.

See, for example, {\sc Ledrappier} \cite{Ledrappier2001} where it is used for 
estimating the Hausdorff dimension of the harmonic measure. On the
other hand, for checking the ray criterion one often needs rather elaborate 
estimates, whereas existence of strips is easier.

We state here another result which will be useful in the following,
and can be found in {\sc Kaimanovich and Woess}~\cite{KaimanovichWoess2002}.
\begin{proposition}\label{prop:rate_esc_poisson}
 Suppose that $P$ is a homogeneous Markov operator
with finite first moment on $\mathsf{G}$. Then the Poisson boundary
of the random walk $\Xn$ with transition matrix $P$ is trivial if
\begin{enumerate}[(a)]
\item $\mathsf{G}$ has subexponential growth, or
\item the drift $l(P)$ vanishes.
\end{enumerate}
\end{proposition}

\section{Half-Space Method for LRW}\label{sec:half_space_method}

Let us go back to our setting where a transient
and irreducible random walk $\Xn$ with transition matrix 
$P_{\mathsf{G}}$ over the transitive graph $\mathsf{G}$ is given.
We also require Assumption \ref{assumptions_brw} to hold. The corresponding lamplighter
random walk on $\lgr$ is $\Zn$, with transition matrix $P$
and $Z_n=(\eta_n,X_n)$. It converges, by Theorem 
\ref{thm_conv_lrw_general_graphs}, to the geometric boundary $\Pi$
defined in \eqref{eq:pi_boundary}.

As before, $\mu_{\infty}^x$ is the hitting distribution of the 
random walk $\Xn$ starting at $x\in \mathsf{G}$. For Borel sets 
$B\subset \partial \mathsf{G}$, we have
\begin{equation*}
\mu_{\infty}^x(B)=\mathbb{P}[X_{\infty}\in B|X_0=x]. 
\end{equation*}
Factorizing with 
respect to the first step, the Markov property yields
\begin{equation*}
\mu_{\infty}^x=\sum_{y\in \mathsf{G}}p_{\mathsf{G}}(x,y)\mu_{\infty}^y.
\end{equation*}
The Borel probability measures family $\{\mu_{\infty}^x:x\in \mathsf{G}\}$, 
are called \textit{harmonic measures}. In view of the irreducibility 
assumption, all harmonic measures $\mu_{\infty}^x$ are equivalent 
to the measure $\mu_{\infty}=\mu_{\infty}^{o}$. Moreover, the space 
$(\partial \mathsf{G},\mu_{\infty})$ is a factor space of the 
Poisson boundary of the random walk $\Xn$, 
and the Poisson formula permits one to identify the space 
$L^{\infty}(\partial \mathsf{G},\mu_{\infty})$ with a certain subspace of 
the space $H^{\infty}(\mathsf{G},P_{\mathsf{G}})$ of bounded harmonic functions.

We are interested in describing the Poisson boundary of lamplighter random walks
$\Zn$ over $\lgr$, with the base random walk $\Xn$ on $\mathsf{G}$ being an irreducible,
transient random walk which satisfies Assumption \ref{assumptions_brw}.

Under the assumptions of Theorem \ref{thm_conv_lrw_general_graphs}, 
let $\nu_{\infty}$ be the distribution of $Z_{\infty}=(\eta_{\infty},X_{\infty})$ 
on $\Pi$, given that the position of the random walk $\Zn$ at time $n=0$ is $(\textbf{0},o)$.
This is a probability measure on $\Pi$ defined for Borel sets $U\subset\Pi$ by
\begin{equation*}
\nu_{\infty}(U)=\P[Z_{\infty}\in U\vert Z_{0}=(\textbf{0},o)].
\end{equation*}
Let $\nu$ be the unique probability measure on $\lgrp$ induced by $P$
($P\leftrightarrow \nu$) as in equation \eqref{eq:correspondence_rw}.
Then the measure $\nu_{\infty}$ is a \textit{$\nu$-harmonic measure} 
for $\Zn$. This means that it satisfies the convolution equation 
$\nu \ast\nu_{\infty}=\nu_{\infty}$. 

Since $\lgrp\subset \AUT(\lgr)$ acts on $\Pi$ by 
measurable bijections and the measure $\nu_{\infty}$ is stationary 
with respect to $\nu$, by Definition \ref{def:mu_bndr} the space
 $(\Pi,\nu_{\infty})$ is a \textit{$\nu$-boundary} for the random walk $\Zn$ 
with transition matrix $P$. We want to prove that this $\nu$-boundary
is indeed the maximal one, that is, the Poisson boundary.

We state a general method to describe the Poisson boundary
of LRW $\Zn$ on $\lgr$ under some reasonable assumptions on the 
base walk $\Xn$.

\subsection{The Half-Space Method}
\label{sec:Method for constructing the strip}
Assume that:
\begin{enumerate}
\item \textbf{Assumption} \ref{assumptions_brw} holds for $\Xn$ and $(\check{X}_{n})$. 
Let $\mu_{\infty}$ and $\check{\mu}_{\infty}$ be the respective 
hitting distributions on $\partial \mathsf{G}$. 
\item For $\mu_{\infty}\times\check{\mu}_{\infty}$-almost 
every pair $(\mathfrak{u},\mathfrak{v})\in\partial \mathsf{G}\times\partial \mathsf{G}$, 
one has a strip $\mathfrak{s}(\mathfrak{u},\mathfrak{v})$, which 
satisfies the conditions from Theorem \ref{thm:strip_crit}:
it is a subset of $\mathsf{G}$, it is $\Gamma$-equivariant, 
and it has subexponential growth, that is,
\begin{equation}\label{eq:subexp_base_strip}
\dfrac{1}{n}\log|\mathfrak{s}(\mathfrak{u},\mathfrak{v})\cap B(o,n)|\to 0,
\mbox{ as }n\to\infty,
\end{equation}
where $B(o,n)=\{x\in \mathsf{G}:\ d(o,x)\leq n\}$ is the ball with 
center $o$ and radius $n$ in $\mathsf{G}$. 
\item For every $x\in\mathfrak{s}(\mathfrak{u},\mathfrak{v})$, one 
can assign to the triple $(\mathfrak{u},\mathfrak{v},x)$ a partition 
of $\mathsf{G}$ into \textit{half-spaces} $\mathsf{G}_{\pm}$, such that 
$\mathsf{G}_{+}$ (respectively, $\mathsf{G}_{-}$) contains a neighbourhood 
of $\mathfrak{u}$ (respectively, $\mathfrak{v}$), and the assignments 
\begin{equation*}
(\mathfrak{u},\mathfrak{v},x)\mapsto \mathsf{G}_{\pm}(x) 
\end{equation*}
are $\Gamma$-equivariant.
\end{enumerate}
As a matter of fact, one can partition $\mathsf{G}$ in more than 
two subsets and the method can still be applied. However, the 
relevant subsets are the ones containing a neighbourhood of 
$\mathfrak{u}$ (respectively, $\mathfrak{v}$). Indeed, by Theorem 
\ref{thm_conv_lrw_general_graphs}, the LRW converges to the boundary 
$\Pi$, which is defined as the set of all pairs $(\phi,\mathfrak{u})$
with $\mathfrak{u}\in\partial \mathsf{G}$ and the only accumulation point of 
the configuration $\phi$ is $\mathfrak{u}$. Therefore, only there 
may be infinitely many lamps switched on (because $\mathfrak{u}$ 
and $\mathfrak{v}$ are the respective boundary points toward the 
random walks $(X_{n})$ and $(\check{X}_{n})$ converge). 

We want to build a finitely supported configuration associated to 
pairs $(\phi_{+},\phi_{-})$ of limit configurations (of the 
lamplighter random walk and of the reversed random walk) accumulating 
at $\mathfrak{u}$ and $\mathfrak{v}$, respectively. In order to do this, 
we restrict $\phi_{+}$ and $\phi_{-}$ on $\mathsf{G}_{-}$ and $\mathsf{G}_{+}$, 
respectively, and then we ``glue together'' the restrictions. Since 
the new configuration depends on the partition of $\mathsf{G}$, we cannot 
choose the same partition for  all $x$, because we will have a constant 
configuration which is not equivariant. Therefore, the partition of 
$\mathsf{G}$ should depend on $x\in s(\mathfrak{u},\mathfrak{v})$. 

Let us now state one of the main results on this thesis, regarding
the Poisson boundary of lamplighter random walks over general base graphs $\mathsf{G}$.
For discrete groups $\Gamma$, the result was published in {\sc Sava}~\cite{SavaPoisson2010}.

\begin{theorem}\label{PoissonTheorem}
Let $\Zn$ be an irreducible, homogeneous random walk with 
finite first moment on $\lgr$. If $\Pi$ is defined as in \eqref{eq:pi_boundary}
and the above assumptions are satisfied, 
then the measure space $(\Pi,\nu_{\infty})$ is the Poisson boundary of $\Zn$, 
where $\nu_{\infty}$ is the limit distribution on $\Pi$ of $Z_{n}$.
\end{theorem}
\begin{proof}
In order to apply the Strip Criterion (Theorem \ref{thm:strip_crit}), we need to find 
$\nu$- and $\check{\nu}$-boundaries for the lamplighter random walk $\Zn$ 
and the reversed lamplighter random walk $(\check{Z}_{n})$, respectively. By 
Theorem \ref{thm_conv_lrw_general_graphs} each of the random walks $(Z_{n})$ and 
$(\check{Z}_{n})$ starting at $(\textbf{0},o)$ converges almost surely to a $\Pi$-valued 
random variable. If $\nu_{\infty}$  and $\check{\nu}_{\infty}$ are their 
respective limit distributions on $\Pi$, then the spaces $(\Pi,\nu_{\infty})$ 
and $(\Pi,\check{\nu}_{\infty})$ are $\nu$- and $\check{\nu}$- boundaries of 
the respective random walks. 
 
Let us take $b_{+}=(\phi_{+},\mathfrak{u})$, $b_{-}=(\phi_{-},\mathfrak{v})\in\Pi$, 
where $\phi_{+}$ and $\phi_{-}$ are the limit configurations of $(Z_{n})$ and $(\check{Z}_{n})$, 
respectively, and $\mathfrak{u},\mathfrak{v}\in\partial \mathsf{G}$ are 
their only respective accumulation points. By the continuity of 
$\nu_{\infty}$  and $\check{\nu}_{\infty}$, the set 
\begin{equation*}
\{(b_{+},b_{-})\in\Pi\times\Pi:\mathfrak{u}=\mathfrak{v}\} 
\end{equation*}
has $(\nu_{\infty}\times\check{\nu}_{\infty})$-measure $0$, 
so that, in constructing the strip $S(b_{+},b_{-})$ we shall consider 
only the case $\mathfrak{u}\neq\mathfrak{v}$.

Use the third assumption in the \textit{Half-Space Method},
and for each $x\in\mathfrak{s}(\mathfrak{u},\mathfrak{v})$ 
consider a partition of $\mathsf{G}$ into $\mathsf{G}_{+}(x)$, $\mathsf{G}_{-}(x)$, and eventually 
$\mathsf{G}\setminus(\mathsf{G}_{+}\cup \mathsf{G}_{-})$. The set $\mathsf{G}_{+}$ 
(respectively, $\mathsf{G}_{-}$) contains a neighbourhood of $\mathfrak{u}$ 
(respectively, $\mathfrak{v}$), and $\mathsf{G}\setminus(\mathsf{G}_{+}\cup \mathsf{G}_{-})$ 
is the remaining subset (which may be empty). The set 
$\mathsf{G}\setminus(\mathsf{G}_{+}\cup \mathsf{G}_{-})$ contains 
neither $\mathfrak{u}$ nor $\mathfrak{v}$. The restriction 
of $\phi_{+}$ on $\mathsf{G}_{-}$ (respectively, of $\phi_{-}$ on $\mathsf{G}_{+}$) 
is finitely supported, since its only accumulation point is $\mathfrak{u}$ 
(respectively, $\mathfrak{v}$), which is not in a neighbourhood of $\mathsf{G}_{-}$ 
(respectively, $\mathsf{G}_{+}$). 
Now ``put together'' the restriction of 
$\phi_{+}$ on $\mathsf{G}_{-}$ and of $\phi_{-}$ on $\mathsf{G}_{+}$ in order to get 
the new configuration
\begin{equation}\label{StripConfiguration}
\Phi(b_{+},b_{-},x)=
\begin{cases}
\phi_{-}, & \mbox{on}\  \mathsf{G}_{+}\\
\phi_{+}, & \mbox{on}\ \mathsf{G}_{-}\\
0, & \mbox{on}\ \mathsf{G}\setminus(\mathsf{G}_{+}\cup \mathsf{G}_{-})
\end{cases} 
\end{equation}
on $\mathsf{G}$, which is, by construction, finitely supported.
For a graphic visualisation of the above construction, see Figure \ref{fig:half_spaces}.

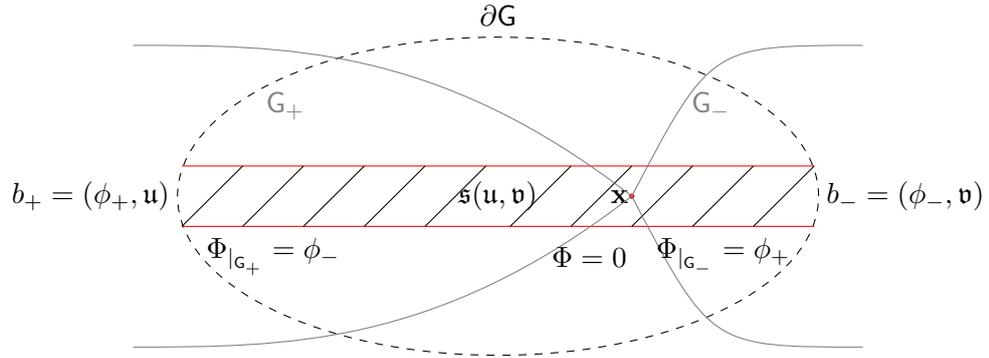
\begin{figure}[h]
\begin{tikzpicture}[scale=0.8]

\draw[dashed,thin] (0,0) ellipse (150pt and 75pt);
\node at (6.7,0) {$b_{-}=(\phi_{-},\mathfrak{v})$};
\node at (-6.7,0) {$b_{+}=(\phi_{+},\mathfrak{u})$};
\node at (0,3) {$\partial \mathsf{G}$};

\draw[thin, red] (-5.2,0.5)--(5.2,0.5);
\draw[thin, red] (-5.2,-0.5)--(5.2,-0.5);
\draw[thin] (-4.2,0.5)--(-5.2,-0.5);
\draw[thin] (-3.2,0.5)--(-4.2,-0.5);
\draw[thin] (-2.2,0.5)--(-3.2,-0.5); 
\draw[thin] (-1.2,0.5)--(-2.2,-0.5);
\draw[thin] (-0.2,0.5)--(-1.2,-0.5);
\draw[thin] (1.2,0.5)--(0.2,-0.5);
\draw[thin] (2.2,0.5)--(1.2,-0.5);
\draw[thin] (3.2,0.5)--(2.2,-0.5);
\draw[thin] (4.2,0.5)--(3.2,-0.5);
\draw[thin] (5.2,0.5)--(4.2,-0.5);

\node at (0,0) {$\mathbf{\mathfrak{s}(\mathfrak{u},\mathfrak{v})}$};

\node at (2,0) {$\mathbf{x}$};
\filldraw[red] (2.2,0) circle (1pt);

\draw[gray] (6,2.5) .. controls (3.5,2.5) .. (2.2,0)..controls (3.5,-2.5) .. (6,-2.5);
\draw[gray] (-6,2.5) .. controls (-3.5,2.5) and (-1,2.5).. (2.2,0)..controls (-1,-2.5) and (-3.5,-2.5) .. (-6,-2.5);
\node[gray] at (3.5,1.5) {$\mathsf{G}_{-}$};
\node at (3.7,-1) {$\Phi_{|_{\mathsf{G}_{-}}}=\phi_{+}$};
\node[gray] at (-3.5,1.5) {$\mathsf{G}_{+}$};
\node at (-3.7,-1) {$\Phi_{|_{\mathsf{G}_{+}}}=\phi_{-}$};
\node at (1.5,-1) {$\Phi=0$};
\end{tikzpicture}

\caption{The construction of the Half-Spaces}\label{fig:half_spaces}
\end{figure}

The sought for 
the ``bigger'' strip $S(b_{+},b_{-})\subset \lgr$ is the set
\begin{equation}\label{eq:lamplighter_strip}
S(b_{+},b_{-})=\{\left(\Phi,x\right) :\  x\in\mathfrak{s}(\mathfrak{u},\mathfrak{v})\}
\end{equation}
of all pairs $(\Phi,x)$, where $\Phi=\Phi(b_{+},b_{-},x)$ is the configuration 
defined above and $x$ runs through the strip $\mathfrak{s}(\mathfrak{u},\mathfrak{v})$ in $\mathsf{G}$. 
This is a subset of $\lgr$. We prove that the map 
\begin{equation*}
(b_{+},b_{-})\mapsto S(b_{+},b_{-})
\end{equation*}
is $\lgrp$-equivariant, i.e., for $g=(\phi,\gamma)\in \lgrp$
\begin{equation*}
gS(b_{+},b_{-})=S(gb_{+},gb_{-}).
\end{equation*}
Next,
\begin{equation*}
gS(b_{+},b_{-})=(\phi,\gamma)\cdot\{\left(\Phi,x\right) 
:\  x\in\mathfrak{s}(\mathfrak{u},\mathfrak{v})\}=\left\lbrace 
(\phi\oplus T_{\gamma}\Phi,\gamma x),\ x\in\mathfrak{s}(\mathfrak{u},\mathfrak{v})\right\rbrace .
\end{equation*}
If $x\in\mathfrak{s}(\mathfrak{u},\mathfrak{v})$, then 
$\gamma x\in\mathfrak{s}(\gamma \mathfrak{u},\gamma\mathfrak{v})$, 
since $\mathfrak{s}(\gamma \mathfrak{u},\gamma\mathfrak{v})$ is $\Gamma$-equivariant. Also,
\begin{equation*}
\phi\oplus T_{\gamma}\Phi=
\begin{cases}
\phi\oplus T_{\gamma}\phi_{-}, & \mbox{on}\  \gamma \mathsf{G}_{+} \\
\phi\oplus T_{\gamma}\phi_{+}, & \mbox{on}\  \gamma \mathsf{G}_{-} \\
	0, & \mbox{on}\ \mathsf{G}\setminus(\gamma \mathsf{G}_{+}\cup \gamma \mathsf{G}_{-}).
\end{cases}
\end{equation*}
This means that 
\begin{equation*}
\phi\oplus T_{\gamma}\Phi(b_{+},b_{-},x)=\Phi(gb_{+},gb_{-},\gamma x), 
\end{equation*}
for all $x\in\mathfrak{s}(\mathfrak{u},\mathfrak{v})$. On the other hand,
\begin{eqnarray*}
S(gb_{+},gb_{-}) & =  & S\big((\phi\oplus T_{\gamma}\phi_{+},\gamma\mathfrak{u}),(\phi\oplus T_{\gamma}\phi_{-},\gamma\mathfrak{v})\big)\\
&  = & \left\{\big(\Phi(gb_{+},gb_{-},\gamma x),\gamma x\big),\ x\in\mathfrak{s}(\mathfrak{u},\mathfrak{v})\right\},
\end{eqnarray*}
that is, $gS(b_{+},b_{-})=S(gb_{+},gb_{-})$, and this  
proves the $\lgrp$-equivariance of the strip $S(b_{+},b_{-})$.

Finally, let us prove that the strip $S(b_{+},b_{-})$ has subexponential growth. 
For this, let $(\eta,x)\in S(b_{+},b_{-})$ such that
\begin{equation*}
d\big((\textbf{0},o),(\eta,x)\big)\leq n. 
\end{equation*}
Definition \ref{LamplighterGraphGmetric} of the metric $d(\cdot,\cdot)$ on 
$\lgr$ implies that $d(o,x)\leq n$. Therefore, if 
\begin{equation}\label{eq:small_strip}
(\eta,x)\in S(b_{+},b_{-})\cap B\big((\textbf{0},o),n\big),\quad \text{then}
\quad x\in\mathfrak{s}(\mathfrak{u},\mathfrak{v})\cap B(o,n),
\end{equation}
where $B\big((\textbf{0},o),n\big)$ (respectively, $B(o,n)$) is the ball with center 
$(\textbf{0},o)$ (respectively, $o$) and of radius $n$ in $\lgr$ (respectively, $\mathsf{G}$). 
Since for every $x\in\mathfrak{s}(\mathfrak{u},\mathfrak{v})$ we associate 
only one configuration $\Phi$ in $S(b_{+},b_{-})$, equation \eqref{eq:small_strip} 
implies that 
\begin{equation*}
|S(b_{+},b_{-})\cap B((\textbf{0},o),n)|\leq |\mathfrak{s}(\mathfrak{u},\mathfrak{v})\cap B(o,n)|.
\end{equation*}
Finally, the assumption \eqref{eq:subexp_base_strip} that 
the $\mathfrak{s}(\mathfrak{u},\mathfrak{v})$ has subexponential growth leads to 
\begin{equation*}
\dfrac{\log|S(b_{+},b_{-})\cap B((\textbf{0},o),n)|}{n}\to 0,\mbox{ as }n\to\infty, 
\end{equation*}
and this proves the subexponential growth of the strip $S(b_+,b_{-})$.

Since for almost every pair of points $(b_{+},b_{-})\in\Pi\times\Pi$, we 
have assigned a strip $S(b_{+},b_{-})$, which satisfies the conditions from 
Theorem \ref{thm:strip_crit}, it follows that the measure space $(\Pi,\nu_{\infty})$ 
is the Poisson boundary of the lamplighter random walk $\Zn$.
\end{proof}

As an application of the \textit{Half Space Method}, we consider 
several classes of transitive base graphs $\mathsf{G}$: graphs with infinitely many ends,
hyperbolic graphs in the sense of Gromov and Euclidean lattices. For
random walks $\Xn$ on these types of graphs Assumption \ref{assumptions_brw} 
holds.

\begin{remark}
In principle, one can apply the construction from the previous proof
to ``iterated'' lamplighter graphs, which are defined as follows. We consider 
$\lgr$ as our base graph, and construct the lamplighter
graph $\Z_2\wr(\lgr)$ over $\lgr$, and so on. The ``smaller''
strip $\mathfrak{s}$ will be in this case a subset of $\lgr$, which has
subexponential growth, and using the Half Space Method,
we lift it to a ``bigger'' strip which satisfies the conditions in 
Theorem \ref{thm:strip_crit}.

\end{remark}

\chapter{Graphs with Infinitely many Ends}\label{inf_many_ends}

In this chapter we introduce transitive graphs $\mathsf{G}$ with rich geometric boundaries,
such as graphs with infinitely many ends, and we study the
behaviour at infinity of random walks $\Xn$ on $\mathsf{G}$ and of lamplighter random 
walks $\Zn$ on $\lgr$, given that $\mathsf{G}$ is a graph with 
infinitely many ends.

The behaviour at infinity of random walks over graphs $\mathsf{G}$ with infinitely
many ends depends on the action of the group $\Gamma\subset \AUT(\mathsf{G},P)$.
If $\Gamma$ does not fix any element of $\partial\mathsf{G}$, then it is
easy to study the convergence and the Poisson boundary of lamplighter random walks
$\Zn$ on $\lgr$ just using Theorem \ref{thm_conv_lrw_general_graphs} and 
Theorem \ref{PoissonTheorem}. On the other hand, if $\Gamma$ 
fixes one end of $\mathsf{G}$, which then has to be uniques, then we encounter one 
interesting situation: when the base random walk $\Xn$ has zero 
drift on $\mathsf{G}$, then one cannot describe the Poisson boundary of $\Zn$
by making use of the \textit{Half Space Method}. This case will be
treated separately in Section \ref{sec:zero_drift},
by considering the correspondence between the 
\textit{tail-algebra} and the Poisson boundary of the respective 
walk. The proof makes also use of \textit{cutpoints} and other results of {\sc James
and Peres}~\cite{JamesPeres1997}. 

\section{Ends of Graphs}\label{sec:space_of_ends}

The basic idea behind the concept of \textit{an end} is to distinguish between 
different ways of going to infinity. Ends carry a natural topology which is 
often not mentioned explicitely. \textit{Ends of graphs} and the 
\textit{end compactification} were originally introduced by 
{\sc Freudenthal}~\cite{Freudenthal1944}, who considered only locally finite graphs. 
{\sc Halin}~\cite{Halin1964} was the first to consider ends of non-locally finite graphs. 

Let $\mathsf{G}$ be an infinite, locally finite, connected graph. 
An \textit{infinite path} or \textit{ray} without self-intersections
 is a sequence $\pi=[x_{0},\ x_{1},\ldots]$ of distinct vertices, 
such that $x_{i}\sim x_{i-1}$ for all $i$, where $\sim$ denotes 
the neighbourhood relation. In $F$ is a finite set of vertices 
and/or edges of $\mathsf{G}$, then the (induced) graph $\mathsf{G}\setminus F$ has
finitely many connected components. Every ray $\pi$ must have all
but finitely many points in precisely one of them, and we say that
$\pi$ ends up in that component. Two rays are called \textit{equivalent} if,
for any finite set of edges $F$, they end up in the same component
of $\mathsf{G}\setminus F$. An \index{graph!end}\textit{end} of $\mathsf{G}$ is 
an equivalence class of rays. 

We write \textit{$\partial \mathsf{G}$ for the space of ends}, and 
$\widehat{\mathsf{G}}=\mathsf{G}\cup\partial \mathsf{G}$ for the 
\index{graph!end!end compactification}\textit{end compactification}
of $\mathsf{G}$. If $C$ is a component of $\mathsf{G} \setminus F$, then we write 
$\partial C$ for the set of those ends whose rays end up in $C$ and
 $\widehat{C}=C \cup \partial C$ for the resulting completion of $C$. 

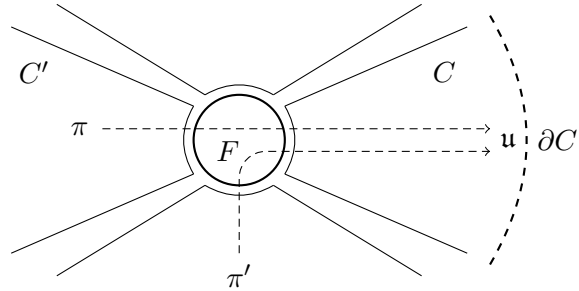
\begin{SCfigure}[1][h]
\begin{tikzpicture}[scale=0.6]

\draw[thick] (0,0) circle (1cm);

\draw (1,0.75)--(5,2.5);
 	\draw [bend right=30] (1,-0.75) to (1,0.75) ;
\draw (1,-0.75)--(5,-2.5);
\draw[thick,dashed] [bend right=30] (5.5,-2.75) to node [right] {$\partial C$} (5.5,2.75);

\draw (-1,0.75)--(-5,2.5);
	\draw [bend left=30] (-1,-0.75) to (-1,0.75);
\draw (-1,-0.75)--(-5,-2.5);

\draw (0.75,1)--(4,3);
	\draw [bend right=30] (0.75,1) to (-0.75,1);
\draw (-0.75,1)--(-4,3);

\draw (0.75,-1)--(4,-3);
	\draw [bend left=30] (0.75,-1) to (-0.75,-1);
\draw (-0.75,-1)--(-4,-3);

\draw[densely dashed, black,->] (-3,0.25)  --  (5.5,0.25);
\draw[densely dashed, black,rounded corners=10pt,->] (0,-2.5) -- (0,-0.25) -- (5.5,-0.25);

\node (f) at (-0.25,-0.25) {$F$};
\node (e) at (4.5,1.5) {$C$};
\node (d) at (-4.5,1.5) {$C'$};
\node (c) at (-3.5,0.25) {$\pi$};
\node (b) at (0,-3) {$\pi'$};
\node (a) at (5.9,0) {$\mathfrak{u}$};
\end{tikzpicture}
\caption{two equivalent rays $\pi$ and $\pi'$ 
ending up in the same component $C$.} 
\end{SCfigure}

Let us now explain the topology of $\widehat{\mathsf{G}}$. If $F$ is a 
finite set and $w\in\widehat{\mathsf{G}}$, then there is exactly one 
component of $\mathsf{G}\setminus F$ whose completion contains $w$. 
We denote the latter by $\widehat{C}(w,F)$. If we vary $F$ 
(finite, with $w\not\in F$), we obtain a neighbourhood base of 
$w$. If $x\in \mathsf{G}$, we can take for $F$ the finite set of neighbours 
of $x$ to see that the topology is discrete on $\mathsf{G}$. It has a countable 
base and it is Hausdorff. When $\mathfrak{u}\in\partial \mathsf{G}$, we can 
find a \textit{standard neighbourhood base}, that is, 
one of the form $\widehat{C}(\mathfrak{u},F_{k})$, $k\in\mathbb{N}$, 
where the finite sets $F_{k}\subset \mathsf{G}$ are such that
\begin{equation*}
F_{k}\cup \widehat{C}(\mathfrak{u},F_{k})\subset 
\widehat{C}(\mathfrak{u},F_{k-1}),\mbox{ for all }k. 
\end{equation*}

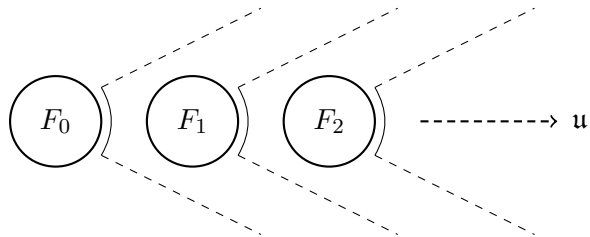
\begin{SCfigure}[1][h]
\begin{tikzpicture}[scale=0.6]

\draw[thick] (0,0) node {$F_{0}$} circle (1cm);
\draw[thick] (3,0) node {$F_{1}$} circle (1cm);
\draw[thick] (6,0) node {$F_{2}$} circle (1cm);

\draw[dashed] (1,0.75)--(4.5,2.5);
 	\draw [bend right=30] (1,-0.75) to (1,0.75) ;
\draw[dashed] (1,-0.75)--(4.5,-2.5);

\draw[dashed] (4,0.75)--(7.5,2.5);
 	\draw [bend right=30] (4,-0.75) to (4,0.75) ;
\draw[dashed] (4,-0.75)--(7.5,-2.5);

\draw[dashed] (7,0.75)--(10.5,2.5);
 	\draw [bend right=30] (7,-0.75) to (7,0.75) ;
\draw[dashed] (7,-0.75)--(10.5,-2.5);

\draw[thick,densely dashed, black,->] (8,0) --(11,0);
\node (a) at (11.5,0) {$\mathfrak{u}$};
\end{tikzpicture} 
\caption{a standard neighbourhood base of $\mathfrak{u}$, 
with $F_1,F_2,F_3$ finite sets} 
\end{SCfigure}

As the number of components of $\mathsf{G}\setminus F$ is finite, 
each $\widehat{C}(\mathfrak{u},F)$ is open and compact, 
and $\widehat{\mathsf{G}}$ is totally disconnected. The space 
$\widehat{\mathsf{G}}$ is called the \textit{end compactification} of $\mathsf{G}$. 

\begin{definition}
An end $\mathfrak{u}$ is called \index{graph!end!thin end}\emph{thin}, if it has a 
standard neighbourhood base with $F_{k}\subset \mathsf{G}$, such that 
\begin{equation*}
\diam(F_{k})=m<\infty , \quad \text{ for all } k,
\end{equation*}
where $\diam(F_{k})$ is the diameter of the set $F_{k}$ 
in the graph metric. The minimal $m$ with this property is the \emph{diameter} of 
$\mathfrak{u}$. Otherwise, $\mathfrak{u}$ is called \index{graph!end!thick end}\emph{thick}.
\end{definition}

Consider the ball $B(o,r)$ of center $o$ (some fixed vertex of $\mathsf{G}$) and 
radius $r$ in $\mathsf{G}$, with respect to the natural graph metric $d$.
In order to prove that the space of ends 
$\partial \mathsf{G}$ is a weakly projective space (see Definition \ref{def:weak_proj}) 
we shall use the following result which can be found in {\sc Woess}~\cite[page 231]{WoessBook}.
\begin{lemma}\label{lem:same_comp_x_y}
Let $r\geq 1$ and $x,y\in \mathsf{G}\setminus B(o,r)$. If 
\begin{equation*}
d(o,x)+d(o,y)-d(x,y)>2r, 
\end{equation*} 
then $x$ and $y$ belongs to the same component of $\mathsf{G}\setminus B(o,r)$.
\end{lemma}
In the results we are going to prove, it is enough for $\partial \mathsf{G}$
to be a weakly projective boundary. 

\begin{lemma}\label{lem:projective_bndr_gr_inf_ends}
The space of ends $\partial \mathsf{G}$ of a locally finite graph $\mathsf{G}$ is a weakly projective boundary.
\end{lemma}
\begin{proof}
Let $(x_n)$ and $(y_n)$ be sequences of vertices in $\mathsf{G}$, such that $(x_n)$ 
converges to some end $\mathfrak{u}\in\partial \mathsf{G}$ in the topology of $\widehat{\mathsf{G}}$, and 
\begin{equation}\label{eq:quotient_dist}
\frac{d(x_{n},y_{n})}{d(o,x_{n})}\to 0 ,\quad \text{ as }n\to\infty. 
\end{equation} 
Assume that the sequence $(y_n)$ converges to another 
end $\mathfrak{v}\in\partial \mathsf{G}$, with $\mathfrak{u}\neq\mathfrak{v}$. 
Then there exists a $r\in\mathbb{N}$, such that, 
for all $n\geq n_r$,
\begin{equation*}
x_n\in \widehat{C}\big(\mathfrak{u},B(o,r)\big) \text{ and }
y_n\in \widehat{C}\big(\mathfrak{v},B(o,r)\big). 
\end{equation*}
This means that for $n$ big enough, the sequences $(x_n)$ and $(y_n)$ 
belong to different components of $\mathsf{G}\setminus B(o,r)$. From Lemma 
\ref{lem:same_comp_x_y}, it follows that
\begin{equation*}
d(o,x_n)+d(o,y_n)-2r<d(x_n,y_n).
\end{equation*}
Dividing both sides of the previous equation through $d(o,x_n)$, 
and using \eqref{eq:quotient_dist}, we get
\begin{equation*}
\frac{d(o,x_n)+d(o,y_n)}{d(o,x_n)}\to 0 ,\quad \text{ as }n\to\infty,
\end{equation*}
which is a contradiction. Therefore, the sequence $(y_n)$ 
converges to the same end $\mathfrak{u}$ and $\partial \mathsf{G}$ is 
a weakly projective boundary.
\end{proof}

\begin{proposition}
 The space of ends $\partial \mathsf{G}$ of a locally finite graph $\mathsf{G}$ is also projective space.
\end{proposition}
For the proof see {\sc Woess}~\cite[page 233]{WoessBook}. 

\section{The Structure Tree of a Graph}\label{sec:structure_tree}

The theory of cuts and structure trees was first 
developed by Dunwoody; see the book by 
{\sc Dicks and Dunwoody}~\cite{DicksDunwoody1989}, or for another
detailed description see {\sc Woess}~\cite{WoessBook} and {\sc Thomassen and Woess}
\cite{ThomassenWoess1993}. 
A detailed study of structure theory may be very fruitful 
for obtaining information on the behaviour of random walks.

A \index{graph!cut of a graph}\textit{cut} of a connected graph $\mathsf{G}$ is a set $F$ of 
edges whose deletion disconnects $\mathsf{G}$. If it disconnects 
$\mathsf{G}$ into precisely two connected components $A=A(F)$ and 
$A^{*}=A^{*}(F)=\mathsf{G}\setminus A$, then we call $F$ \textit{tight}, 
and $A,A^{*}$ are the \textit{sides of $F$}. In \cite{ThomassenWoess1993}, 
Thomassen has proved the following.
\begin{lemma}\label{StructureLemma}
For any $k\in \mathbb{N}$, there are only finitely 
many tight cuts $F$ with $|F|=k$ that contain a given edge of $\mathsf{G}$.
\end{lemma}
Two cuts $F,F'$ are said \textit{to cross} if all four sets
\begin{equation*}
A(F)\cap A(F'),\ A(F)\cap A^{*}(F'),\ A^{*}(F)\cap A(F'),\ A^{*}(F)\cap A^{*}(F')
\end{equation*}
are nonempty. {\sc Dunwoody}~\cite{Dunwoody82} proved the following important theorem.

\begin{theorem}\label{thm:dcut}
Every infinite, connected graph with more than one end 
has a finite tight cut $F$ with infinite sides, such 
that $F$ crosses no $\gamma F$, where $\gamma\in \AUT(\mathsf{G})$. A cut 
with these properties will be called a \emph{$D$-cut}.
\end{theorem}

Let $F$ be a $D$-cut of the locally finite, connected 
graph $\mathsf{G}$ and $\Gamma$ be a closed subgroup of $\AUT(\mathsf{G})$. Define 
\begin{equation*}
\mathcal{E}=\{ A(\gamma F),A^{*}(\gamma F): \gamma\in\Gamma\}.
\end{equation*}

This set has the following properties:
\begin{enumerate}
\item All $A\in \mathcal{E}$ are infinite and connected.
\item If $A\in \mathcal{E}$, then $A^{*}=\mathsf{G}\setminus A\in \mathcal{E}$.
\item If $A,B \in \mathcal{E}$ and $A\subset B$, then there 
are only finitely many $C\in \mathcal{E}$, such that $A\subset C\subset B$.
\item If $A,B \in \mathcal{E}$, then one of 
$A\subset B,\ A\subset B^{*},\ A^{*}\subset B\  \mbox{or}\  A^{*}\subset B^{*}$ holds.
\end{enumerate}

This properties can now be used to construct a 
tree $\mathcal{T}$, called the \textit{structure tree} of 
$\mathsf{G}$ with respect to $\Gamma$ and the $D$-cut $F$. One can think 
of an unoriented edge of $\mathcal{T}$ as a pair of oriented 
edges, where the second edge points from the endpoint to the 
initial point of the first one. The oriented edge set of 
$\mathcal{T}$ is $\mathcal{E}$. That is, if $A\in \mathcal{E}$, 
then $(A,A^{*})$ constitutes a pair of oppositely oriented 
edges between the same two vertices. If $A,B \in\mathcal{E}$ and 
$B\neq A^{*}$, then the endpoint of $A$ is the initial 
point of $B$ if $A\supset B$, and there is no 
$C\in \mathcal{E}$, such that $A\supset C\supset B$ properly.

In this way, we have defined $\mathcal{T}$ in terms of 
its edges and their incidence, contrary to the usual 
approach of defining a graph by starting with its 
vertices. A vertex of $\mathcal{T}$ is an equivalence 
class of edges ``with the same endpoint'', that is, 
$A,B\in\mathcal{E}$ are \textit{equivalent} in this 
sense, if $A=B$, or else, if $A\supset B^{*}$ properly 
and no $C\in\mathcal{E}$ satisfies $A\supset C\supset B^{*}$ properly. 
One can check that this is indeed an equivalence relation.

The vertex set is the set of all equivalence classes $[A]$, 
where $A\in\mathcal{E}$. Neighbourhood in $\mathcal{T}$ is 
described by $[A]\sim [A^{*}]$. One can show that $\mathcal{T}$ 
is a tree: it is connected because of the property $(3)$, 
it has no cycles, since the neighbourhood relation is defined 
in terms of inclusion of sets. 

The tree $\mathcal{T}$ is 
countable by Lemma \ref{StructureLemma}, but not necessarily 
locally finite. One can still define the set $\partial\mathcal{T}$ 
of ends of $\mathcal{T}$ as equivalence classes of rays, as 
in Section \ref{sec:space_of_ends}, and 
$\widehat{\mathcal{T}}=\mathcal{T}\cup\partial\mathcal{T}$. 
The group $\Gamma$ acts by automorphisms on $\mathcal{T}$ via 
$A\mapsto\gamma A$, where $\gamma\in\Gamma$ and $A\in \mathcal{E}$. 
The action has one or two orbits on $\mathcal{E}$ according 
to wheather $\gamma A(F)=A^{*}(F)$ for some $\gamma\in\Gamma$ or not. 
Consequently, $\Gamma$ acts transitively on $\mathcal{T}$ or else 
acts transitively on each of the two bipartite classes of $\mathcal{T}$ 
(that is, the sets of vertices at even/odd distance from a chosen origin).

\begin{example}If $\mathsf{G}$ is the homogeneous tree $\mathcal{T}_{M}$ of degree $M$, 
then any single edge constitutes a $D$-cut. Moreover, if $\Gamma$ is the 
whole automorphism group of $\mathcal{T}_{M}$, then the structure tree is 
again $\mathcal{T}_{M}$. This is not the case, when 
$\Gamma=\mathbb{Z}_{2}\ast \mathbb{Z}_{2} \cdots\ast \mathbb{Z}_{2}$ ($M$ times), 
or $\Gamma$ is the free group.
\end{example}

\begin{example}As an example with thick ends, consider the standard 
Cayley graph $\mathsf{G}$ of the free product
\begin{equation*}
\Gamma=\mathbb{Z}^{2}*\mathbb{Z}_{2}=\langle a,b,c | ab=ba,c^{2}=o\rangle, 
\end{equation*} 
acting by $\Gamma$ on itself. Each copy $\gamma \mathbb{Z}^{2}$ of the square 
grid within $\mathsf{G}$, where $\gamma\in\Gamma$, gives rise to a thick end 
(as an equivalence class of rays that end up in such a copy). The other ends 
are all thin, they have zero diameter. Let $F$ consist of the single edge $[o,c]$. 
This is a $D$-cut, and the structure tree has infinite vertex degrees.\end{example}

Next, let us introduce the \textit{structure map} 
$\varphi :\widehat{\mathsf{G}} \rightarrow \widehat{\mathcal{T}}$. 
In order to understand it, let $x \in \widehat{\mathsf{G}}$. Then there is some 
$A_{0}\in \mathcal{E}$ which contains $x$. If there is a minimal $A\in \mathcal{E}$ 
with this property, then we define $\varphi(x)$ as the end vertex of $A$ as an 
edge of $\mathcal{T}$. If there is no minimal 
$A$ with this property, then there must be a maximal strictly descending 
sequence $A_{0}\supset A_{1}\supset A_{2}\supset \cdots$ in $\mathcal{E}$, 
such that $x\in A_{n}$, for all $n$. As edges of $\mathcal{T}$, the $A_{n}$ 
constitute a path which defines an end in $\partial \mathcal{T}$. 
This end is $\varphi(x)$. The image of $x$ does not depend on 
the particular choice of the initial $A_{0}\in \mathcal{E}$ 
containing $x$. Via $(\gamma,A)\mapsto \gamma A$ for $A\in \mathcal{E}$, 
the group $\Gamma$ acts on $\mathcal{T}$ and $\varphi$ 
commutes with the actions of $\Gamma$ on $\widehat{\mathsf{G}}$ and on $\widehat{\mathcal{T}}$.

If $x$ is a vertex of $\mathsf{G}$, then $\varphi(x)$ is a vertex of $\mathcal{T}$. 
Given an end of $\mathcal{T}$, its preimage under $\varphi$ consists 
of a single end of $\mathsf{G}$. However, there are ends of $\mathsf{G}$ that are 
mapped to vertices of $\mathcal{T}$ under $\varphi$. We write 
\begin{equation*}
\partial^{(0)}\mathsf{G}=\varphi^{-1}(\partial \mathcal{T}). 
\end{equation*}
These are thin ends of $\mathsf{G}$, i.e., ends with finite diameter.


\paragraph{Group Actions.}\label{par:gr_act_gr_inf_ends}
If $\mathsf{G}$ is an infinite, locally finite, connected  and transitive graph
(i.e. there exists $\Gamma$  subgroup of $\AUT(\mathsf{G})$, which acts transitively on $\mathsf{G}$), 
then $\mathsf{G}$ has \textit{one, two or infinitely many ends}. For details, 
see {\sc Woess}~\cite{Woess1989}. 

If $\mathsf{G}$ has one end, then the end compactification is not suitable 
for a good description of the structure of $\mathsf{G}$ at infinity. 

If $\mathsf{G}$ has two ends, then it is isometric with the two-way-infinite 
path. Moreover, the behaviour at infinity of lamplighter random walks over 
the integer line (two-way-infinite path) is well-studied.

For the rest of this chapter we suppose that $\mathsf{G}$ has infinitely many ends, 
since we want a base structure (as a base graph for the lamplighter graph) 
which is endowed with a rich boundary.
The space $\widehat{\mathsf{G}}$ is the end compactification and the boundary 
$\partial \mathsf{G}$ is the space of ends of $\mathsf{G}$, with $|\partial \mathsf{G}|=\infty$.

The simplest examples of graphs with infinitely many ends are 
free products of graphs (with the exception of the Cayley graph 
of $\mathbb{Z}_2*\mathbb{Z}_2$, which has two ends). 
More generally, the free product of two (or finitely many) rooted graphs has 
infinitely many ends, unless both have only two elements. The tree is the 
nicest example of a graph with infinitely many ends. Anyway, 
there are graphs with infinitely many ends which have a very complicated structure. 
For a detailed exposition, the reader is invited to have a 
look at the book of {\sc Dicks and Dunwoody} \cite{DicksDunwoody1989}.
\paragraph{Ends and Fixed Sets.}\label{par:ends_fixed_sets}
Let $\mathsf{G}$ be a graph with infinitely many ends
and $\Gamma\subset \AUT(\mathsf{G})$ which acts transitively
on $\mathsf{G}$, and $\partial^{(0)}\mathsf{G}$ the set of thin ends, i.e., 
those ends with finite diameter. From {\sc Woess} \cite{Woess1989} it 
is known that the set of thin ends $\partial^{(0)}\mathsf{G}$ is dense in $\partial \mathsf{G}$.

\begin{definition}
A subset $B\subset\partial \mathsf{G}$ is fixed by the action of $\Gamma$, if 
$\gamma B=B$ (pointwise action) for every $\gamma\in\Gamma$.
\end{definition}

We also recall the following result, which will be relevant in the sequel. For
a proof, see {\sc Soardi and Woess}~\cite{SoardiWoess1989}.

\begin{theorem}\label{thm:fixed_end_amenable}
The group $\Gamma$ cannot fix a finite subset of $\partial \mathsf{G}$ other than a singleton.
This happens if and only if $\Gamma$ is amenable. 
\end{theorem}

Therefore, when studying random walks on graphs with infinitely many
ends, one has to distinguish two substantially different cases:
\textit{when $\Gamma\subset \AUT(\mathsf{G})$ doesn't fix any end in $\partial \mathsf{G}$ and 
when one end in $\partial \mathsf{G}$ is fixed under the action of $\Gamma$}. 

\section{LRW over Graphs with Infinitely many Ends}

The goal of this section is to apply the results developed in 
Chapter \ref{chap:conv_lrw} and Chapter \ref{chap:poiss_bndr_lrw}
for lamplighter random walks $\Zn$ on $\lgr$,
where $\mathsf{G}$ is a transitive graph with infinitely many ends, i.e., 
$|\partial \mathsf{G}|=\infty$. 

Let us now recall the setting we are working on: 
we are interested in the behaviour at infinity of
homogeneous random walks $\Zn$ on $\lgr$, 
with $\Z_n=(\eta_n,X_n)$. The factor chain $\Xn$ is a 
random walk over a transitive graph $\mathsf{G}$ with infinitely many ends. 
Note that when $\mathsf{G}$ has infinitely many ends, 
the lamplighter graph $\lgr$ over $\mathsf{G}$ has only one end. 
The group $\Gamma$ acts transitively on $\mathsf{G}$ (with or without fixed end) 
and the wreath product $\lgrp$ (the lamplighter group) acts transitively 
on the lamplighter graph $\lgr$.

\subsection{No Fixed End}\label{subsec:no_fixed_end}

When $\Gamma\subset \AUT(\mathsf{G},P)$ acts transitively on $\mathsf{G}$ and does not fix any 
end in $\partial \mathsf{G}$, it is known from {\sc Woess}~\cite{Woess_Amenable1989}
that $\Gamma$ is nonamenable. 

The following holds for random walks $\Xn$ on transitive graphs $\mathsf{G}$
with infinitely many ends. For the proof, see {\sc Woess}~\cite{Woess1989}.

\begin{theorem}
\label{thm:no_fixend_base_random_walk}
The random walk $\Xn$ over $\mathsf{G}$ converges almost surely in the end topology
to a random end $X_{\infty}\in\partial \mathsf{G}$. Denoting by $\mu_{\infty}$
the hitting distribution on $\partial \mathsf{G}$ we have:
	\begin{enumerate}[(a)]
              \item The support of $\mu_{\infty}$ is the whole $\partial \mathsf{G}$.
	      \item $\mu_{\infty}$ is continuous on $\partial \mathsf{G}$, that is 
		$\mu_{\infty}(\{\mathfrak{u}\})=0, \ \text{ for all }\mathfrak{u}\in \partial \mathsf{G}$.
	      \item The mass of thick ends is zero, i.e., $\mu_{\infty}(\partial \mathsf{G} \setminus \partial ^{(0)}\mathsf{G})=0$.
        \end{enumerate}
\end{theorem}
Recall that $\mu_{\infty}$ is the probability distribution defined 
for Borel subsets $B\subset\partial \mathsf{G}$ by
	\begin{equation}\label{eq:lim_dist_brw}
 	\mu_{\infty}(B)=\mathbb{P}[X_{\infty}\in B|X_0=o].
	\end{equation}

For convergence of $\Xn$ to $X_{\infty}\in\partial \mathsf{G}$ the finite
first moment assumption is not needed.
Suppose now that $\Xn$ has finite first moment on $\mathsf{G}$ 
and recall that the space of ends $\partial \mathsf{G}$ is a weakly projective 
boundary. Then Assumption \ref{assumptions_brw} holds.

Let us now state the result regarding the convergence of 
lamplighter random walks $\Zn$ over $\lgr$, whose projection
is the random walk $\Xn$ over the graph with infinitely many ends $\mathsf{G}$.
The boundary $\partial \mathsf{G}$ is the space of ends and the 
boundary $\Pi$ of $\lgr$ is defined in \eqref{eq:pi_boundary} as
\begin{equation*}
\Pi =\bigcup_{\mathfrak{u} \in \partial \mathsf{G}}\mathcal{C}_{\mathfrak{u}}\times \{\mathfrak{u}\}, 
\end{equation*}
where $\mathcal{C}_{\mathfrak{u}}$ is the set of configurations
which are either finitely supported or accumulate at $\mathfrak{u}$.
\begin{theorem}
\label{thm:conv_lrw_no_fixedend}
Let $\Zn$ be an irreducible, homogeneous random walk 
with finite first moment on $\lgr$, where $\mathsf{G}$ is a graph 
with infinitely many ends and $\Gamma\subset \AUT(\mathsf{G})$ 
does not fix any end in $\partial \mathsf{G}$.
Then there exists a $\Pi$-valued random variable $Z_{\infty}=(\eta_{\infty},X_{\infty})$,
such that $Z_n\to Z_{\infty}$ almost surely, in the topology of $\widehat{\lgr}$
for every starting point.
Moreover the distribution of $Z_{\infty}$ is a continuous measure on $\Pi$. 
\end{theorem}
\begin{proof}
The proof of this result follows basically the proof of 
Theorem \ref{thm_conv_lrw_general_graphs},
which holds for general base graphs $\mathsf{G}$ endowed with a rich boundary and such 
that Assumption \ref{assumptions_brw} holds. 

The general boundary $\partial \mathsf{G}$ 
of the base graph is here the space of ends and 
the boundary of the lamplighter graph 
$\lgr$ is $\Pi$ defined like in equation \eqref{eq:pi_boundary}.

Since, by Theorem \ref{thm:no_fixend_base_random_walk}, the 
limit distribution $\mu_{\infty}$ is a continuous measure
on $\partial \mathsf{G}$, the same is true (using a Borel-Cantelli argument)
for the limit distribution of $\Zn$ on $\Pi$.
\end{proof}

\paragraph{The Poisson Boundary.} First of all, let
us apply the Strip criterion due to {\sc Kaimanovich}~\cite[Thm. $6.5$ on p. 677]{Kaimanovich2000} 
in order to determine the Poisson boundary of random walks $\Xn$
over graphs with infinitely many ends $\mathsf{G}$. This  was done in 
{\sc Kaimanovich and Woess}~\cite[Thm. $5.19$]{KaimanovichWoess2002}.
Nevertheless, we give a complete proof here, since the 
geometric construction will be of interest for further results.

\begin{theorem}
\label{thm:poisson_rw_inf_ends}
If $\Gamma$ does not fix any end in $\partial\mathsf{G}$ and the random walk $\Xn$ has finite first 
moment on $\mathsf{G}$, then the Poisson boundary of $\Xn$ is the
measure space $(\partial \mathsf{G},\mu_{\infty})$, where $\mu_{\infty}$ 
is the limit distribution of $\Xn$ on $\partial \mathsf{G}$.
\end{theorem}

\begin{proof}
By Theorem \ref{thm:no_fixend_base_random_walk} and Definition \ref{def:mu_bndr}, 
the space $(\partial \mathsf{G},\mu_{\infty})$ is a $\mu$-boundary for the random walk
$\Xn$ with transition matrix $P_{\mathsf{G}}$, 
where $\mu$ is associated with $P_{\mathsf{G}}$ as in \eqref{eq:correspondence_rw}.
Theorem \ref{thm:no_fixend_base_random_walk} applies also to the reversed 
random walk $(\check{X}_{n})$, which is the random walk on $\mathsf{G}$ with transition 
matrix $\check{P}_{\mathsf{G}}$. If the corresponding limit distribution on $\partial \mathsf{G}$ is 
$\check{\mu}_{\infty}$, then $(\partial \mathsf{G},\check{\mu}_{\infty})$ is a 
$\check{\mu}$-boundary for $(\check{X}_{n})$.

Apply the Strip Criterion \ref{thm:strip_crit}. By Theorem \ref{thm:dcut},
there exists a $D$-cut $F\subset \mathsf{G}$, whose removal disconnects $\mathsf{G}$ into finitely 
many infinite connected components. The $D$-cut was used for the construction
of the structure tree $\mathcal{T}$ of $\mathsf{G}$. Denote by $F^{0}$ the set of
end vertices of $F$.

Let now $\mathfrak{u},\mathfrak{v}$ be ends of $\partial \mathsf{G}$. 
By continuity of $\mu_{\infty}$ and $\check{\mu}_{\infty}$, we have 
\begin{equation*}
\mu_{\infty}\times\ \check{\mu}_{\infty}\big(\{\mathfrak{u},\mathfrak{v}\in\partial \mathsf{G}: 
\mathfrak{u}=\mathfrak{v}\}\big) =0.
\end{equation*}
Therefore, we have to construct the strip $\mathfrak{s}(\mathfrak{u},\mathfrak{v})$
only in the case $\mathfrak{u} \neq \mathfrak{v}$. Define the ``small'' strip
\begin{equation}\label{strip_base_rw_gr_inf_ends}
\mathfrak{s}(\mathfrak{u},\mathfrak{v})=\bigcup_{\gamma\in\Gamma}
\big\{\gamma F^{0}:\widehat{\mathcal{C}}(\mathfrak{u},\gamma F)
\neq\widehat{\mathcal{C}}(\mathfrak{v},\gamma F)\big\}.
\end{equation}
The set $\mathcal{C}(\mathfrak{u},F)$ is the connected component of $\mathsf{G}$ 
which represents the end $\mathfrak{u}$ when we remove the finite set $F$ from $\mathsf{G}$, 
and $\widehat{\mathcal{C}}(\mathfrak{u},F)$ its completion (which contains 
$\mathfrak{u}$) in $\widehat{\mathsf{G}}$. The strip
$\mathfrak{s}(\mathfrak{u},\mathfrak{v})$ is a subset of 
$\mathsf{G}$.

\begin{figure}[h]

\begin{tikzpicture}[scale=0.9]

\path[shade] (0,0) circle (0.8cm);
\draw[thin, red] (0,0) circle (1cm); 
\node (a) at (0,1.2) {$F$};
\draw[->] (0,-1.5) -- (0,-1);
\node (a) at (0,-2) {$F^{0}$};

\draw[densely dashed] (1,0.75)--(6,2.5);
 	\draw [bend right=30] (1,-0.75) to (1,0.75) ;
	\node at (6,2) {$\widehat{\mathcal{C}}(\mathfrak{u},F)$};
	\node at (6.5,0) {$\mathfrak{u}$};
\draw[densely dashed] (1,-0.75)--(6,-2.5);

\draw[densely dashed] (-1,0.75)--(-6,2.5);
 	\draw [bend left=30] (-1,-0.75) to (-1,0.75) ;
	\node at (-6,2) {$\widehat{\mathcal{C}}(\mathfrak{v},F)$};
	\node at (-6.5,0) {$\mathfrak{v}$};
\draw[densely dashed] (-1,-0.75)--(-6,-2.5);

\draw[thin, red] (1.5,0)  circle (1cm);
\draw[->] (1.5,-1.5) -- (1.5,-1);
\node at (1.5,-2) {$\gamma_{1}F^{0}$};
\draw[thin, red] (-1.5,0) circle (1cm);
\draw[->] (-1.5,-1.5) -- (-1.5,-1);
\node  at (-1.5,-2) {$\gamma_{2}F^{0}$};

\draw[densely dashed] (2.5,0.75)--(6.5,1.5);
 	\draw [bend right=30] (2.5,-0.75) to (2.5,0.75) ;
	\node at (6,0.9) {$\widehat{\mathcal{C}}(\mathfrak{u},\gamma_{1}F)$};
\draw[densely dashed] (2.5,-0.75)--(6.5,-1.5);

\draw[densely dashed] (-2.5,0.75)--(-6.5,1.5);
 	\draw [bend left=30] (-2.5,-0.75) to (-2.5,0.75) ;
	\node at (-6,0.9) {$\widehat{\mathcal{C}}(\mathfrak{u},\gamma_{2}F)$};
\draw[densely dashed] (-2.5,-0.75)--(-6.5,-1.5);

\end{tikzpicture}

\caption{The construction of the strip $\mathfrak{s}(\mathfrak{u},\mathfrak{v})$}
\label{fig:strip_inf_ends}
\end{figure}
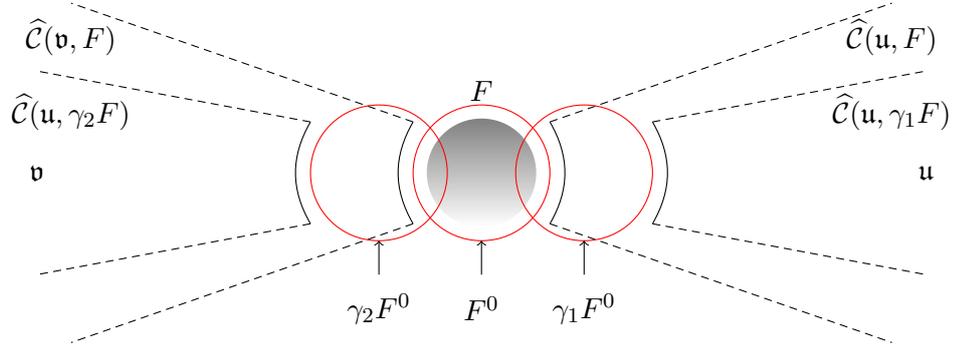

In other words, the strip $\mathfrak{s}(\mathfrak{u},\mathfrak{v})$ is the union of 
all translates $\gamma F^{0}$, with the property that the components, which
contain the ends $\mathfrak{u}$ and $\mathfrak{v}$ after removing the finite
set $F$ are different. See Figure \ref{fig:strip_inf_ends} for the construction. 

We have to check that the strip is $\Gamma$-equivariant and it is ``thin'' 
enough, i.e., it has subexponential growth.

Since $\Gamma$ acts transitively on $\mathsf{G}$, it is clear that $\gamma \mathfrak{s}(\mathfrak{u},\mathfrak{v})=\mathfrak{s}(\gamma\mathfrak{u},\gamma\mathfrak{v})$,
for every $\gamma\in\Gamma$. The strip 
$\mathfrak{s}(\mathfrak{u},\mathfrak{v})$ is the union of all $\gamma F^{0}$, such that the
sides of $\gamma F$, seen as edges of the structure tree $\mathcal{T}$, lie
on the geodesic between $\varphi{\mathfrak{u}}$ and $\varphi{\mathfrak{v}}$. Recall
that $\varphi$ is the function (the structure map) which maps $\widehat{\mathsf{G}}$ onto 
$\widehat{\mathcal{T}}$. 

From Theorem \ref{thm:no_fixend_base_random_walk}, the random walk 
$\Xn$ on $\mathsf{G}$ converges to a thin end in $\partial \mathsf{G}$. If $\mathfrak{u}$ 
and $\mathfrak{v}$ are thin ends in $\partial \mathsf{G}$, then $\varphi{\mathfrak{u}}$ 
and $\varphi{\mathfrak{v}}$ are ends in the 
structure tree $\mathcal{T}$ and the strip $\mathfrak{s}(\mathfrak{u},\mathfrak{v})$ 
is a two way infinite geodesic in $\mathcal{T}$. It can be empty in the case 
$\mathfrak{u}=\mathfrak{v}$, but we already excluded this case. 

From Lemma \ref{StructureLemma} and the fact that $F$ is a $D$-cut, there is
an integer $k>0$, such that the following holds: if $A_1,A_2,\ldots,A_k\in\mathcal{E}$ 
are egdes in $\mathcal{T}$ and $A_0\supset A_1 \supset \cdots \supset A_k$ properly,
then $d(A_k,A^{*}_0)\geq 2$. In other words, if $A_k$ is one of the sides of $\gamma F$,
with $\gamma\in\Gamma$, then $\gamma F^{0}$ is entirely contained in $A_0$.
The finiteness of $F^{0}$ implies the existence of a constant $c>0$, such that,
for the ball $B(o,n)$ of center $o$ and radius $n$ in the graph metric of $\mathsf{G}$,
\begin{equation*}
|\mathfrak{s}(\mathfrak{u},\mathfrak{v})\cap B(o,n)|\leq cn, 
\end{equation*}
for all n and distinct $\mathfrak{u},\mathfrak{v}\in\partial^{(0)}\mathsf{G}$. Applying the
logarithm and dividing through $n$ we get,
\begin{equation*}
\frac{1}{n}\log| \mathfrak{s}(\mathfrak{u},\mathfrak{v})\cap B(o,n) | \to 0 ,\ \text{ as }\ n\to\infty,
\end{equation*}
and this proves the subexponential growth of the strip. 
Therefore, by Strip Criterion \ref{thm:strip_crit},
the space $(\partial \mathsf{G},\mu_{\infty})$ is the Poisson 
boundary of the random walk $\Xn$ on $\mathsf{G}$.
\end{proof}

Now we are able to apply the \textit{Half Space Method} explained in 
Section \ref{sec:half_space_method} in order to describe 
the Poisson boundary of lamplighter random walks $\Zn$ over $\lgr$, 
where $\mathsf{G}$ is a transitive graph with infinitely many ends. Assumption 
\ref{assumptions_brw} holds for random walks $\Xn$ on $\mathsf{G}$ by 
Theorem \ref{thm:no_fixend_base_random_walk}. 

Similar results on the Poisson boundary of lamplighter random walks 
on groups with infinitely many ends were considered in {\sc Sava}~\cite{SavaPoisson2010}.
\begin{theorem}
\label{thm:poisson_lrw_gr_infends_nofixed_end}
Let $\Zn$ be an irreducible, homogeneous random walk with finite first moment 
on $\lgr$, where $\mathsf{G}$ is a transitive graph with infinitely many ends.
If $\Gamma\subset \AUT(\mathsf{G})$ does not fix any end in $\partial \mathsf{G}$, then 
$(\Pi,\nu_{\infty})$ is the Poisson boundary of $\Zn$, starting at 
$(\mathbf{0},o)$, where $\Pi$ is defined in \eqref{eq:pi_boundary}, and
$\nu_{\infty}$ is the limit distribution on $\Pi$.
\end{theorem}
\begin{proof}
In order for Theorem \ref{PoissonTheorem} to be applicable, 
we need the conditions required in the \textit{Half Space Method} to 
be satisfied for the base graph $\mathsf{G}$ and random walks $\Xn$
and $(\check{X}_n)$ on it. 

By assumption, $\Xn$ has finite first moment on $\mathsf{G}$, and
Theorem \ref{thm:no_fixend_base_random_walk}
and Lemma \ref{lem:projective_bndr_gr_inf_ends} imply that
Assumption \ref{assumptions_brw} hold for $\Xn$ and also
for $(\check{X}_n)$. As usual, $\mu_{\infty}$ and 
$\check{\mu}_{\infty}$ are the respective limit distributions on $\partial \mathsf{G}$. 

Next, assign a strip 
$\mathfrak{s}(\mathfrak{u},\mathfrak{v})\subset \mathsf{G}$ to almost 
every pair of ends $(\mathfrak{u},\mathfrak{v})\in\partial \mathsf{G}\times\partial \mathsf{G}$.
Consider the strip $\mathfrak{s}(\mathfrak{u},\mathfrak{v})\subset \mathsf{G}$
defined in the proof of Theorem \ref{thm:poisson_lrw_gr_infends_nofixed_end},  
by equation \eqref{strip_base_rw_gr_inf_ends}, which satisfies the Strip Criterion
conditions in Theorem \ref{thm:strip_crit}.

The next step is to partition $\mathsf{G}$ into half-spaces. By  construction of the
``small'' strip $\mathfrak{s}(\mathfrak{u},\mathfrak{v})$, 
every $x\in\mathfrak{s}(\mathfrak{u},\mathfrak{v})$ is contained in some 
cut $\gamma F$, for some $\gamma\in\Gamma$. Nevertheless, there
are finitely many $\gamma\in\Gamma$, such that, for 
$x\in \mathfrak{s}(\mathfrak{u},\mathfrak{v})$, we have $x\in\gamma F$,
since $F$ is finite.

The partition of $\mathsf{G}$ is done as follows: 
for each $x\in\mathfrak{s}(\mathfrak{u},\mathfrak{v})$, look 
at the $D$-cuts $\gamma F$ containing $x$, pick one of them 
and remove it from $\mathsf{G}$. 
Then the set $(\mathsf{G}\setminus\gamma F)$ contains finitely many 
connected components. This follows from the definition of a $D$-cut, 
and from the finiteness of the removed set $F$. Moreover, the connected 
components containing $\mathfrak{u}$ and $\mathfrak{v}$ are different, 
by the definition of the strip $\mathfrak{s}(\mathfrak{u},\mathfrak{v})$. 
Let $\mathsf{G}_{+}=\mathsf{G}_{+}(x)$ be the connected component 
of $(\mathsf{G}\setminus\gamma F)$, which contains $\mathfrak{u}$, and $\mathsf{G}_{-}=\mathsf{G}_{-}(x)$ 
be its complement in $\mathsf{G}$, which contains $\mathfrak{v}$. One can see here 
that the partition of $\mathsf{G}$ into the half-spaces $\mathsf{G}_{+}$ and 
$\mathsf{G}_{-}$ depends on the cut $\gamma F$ containing $x$, that is, 
depends on $x$. The sets $\mathsf{G}_{+}$ and $\mathsf{G}_{-}$ are $\Gamma$-equivariant. 

From the above, it follows that all the assumptions needed in the \textit{Half Space Method} 
hold in the case of a graph with infinitely many ends $\mathsf{G}$. 
Apply now Theorem \ref{PoissonTheorem}.

By Theorem \ref{thm_conv_lrw_general_graphs} each of the random walks $(Z_{n})$ and 
$(\check{Z}_{n})$ starting at $(\mathbf{0},o)$ converges almost surely to a 
$\Pi$-valued random variable. If $\nu_{\infty}$ and $\check{\nu}_{\infty}$ are 
their respective limit distributions on $\Pi$, then the spaces $(\Pi,\nu_{\infty})$ and 
$(\Pi,\check{\nu}_{\infty})$ are $\nu$- and $\check{\nu}$- boundaries 
of the respective walks. Take
\begin{equation*}
b_{+}=(\phi_{+},\mathfrak{u})\in\Pi,\text{ and } b_{-}=(\phi_{-},\mathfrak{v})\in\Pi 
\end{equation*} 
where $\phi_{+}$ and $\phi_{-}$ are the limit configurations of $(Z_{n})$ and 
$(\check{Z}_{n})$, respectively, and $\mathfrak{u},\mathfrak{v}\in\partial \mathsf{G}$ 
are their only respective accumulation points. Define the configuration 
$\Phi(b_{+},b_{-},x)$ like in \eqref{StripConfiguration}, that is,
\begin{equation*}
\Phi(b_{+},b_{-},x)=
\begin{cases}
\phi_{-}, & \mbox{on}\  \mathsf{G}_{+}\\
\phi_{+}, & \mbox{on}\ \mathsf{G}_{-}
\end{cases} 
\end{equation*}
where $\mathsf{G}\setminus(\mathsf{G}_{+}\cup \mathsf{G}_{-})$ is the empty set. Consider 
the strip $S(b_{+},b_{-})$ exactly like in \eqref{eq:lamplighter_strip}, i.e.,
\begin{equation*}
S(b_{+},b_{-})=\{\left(\Phi,x\right) :\  x\in\mathfrak{s}(\mathfrak{u},\mathfrak{v})\}.
\end{equation*}
By Theorem \ref{PoissonTheorem}, $S(b_{+},b_{-})$ satisfies the 
conditions from Strip Criterion \ref{thm:strip_crit}, and this implies 
that the space $(\Pi,\nu_{\infty})$ is the Poisson boundary of 
the lamplighter random walk $\Zn$ over $\lgr$.
\end{proof}

\subsection{One Fixed End}\label{subsec:one_fixed_end}

Let $\xi\in\partial \mathsf{G}$ be an end of $\mathsf{G}$, which is fixed under the action 
of the group $\Gamma\subset \AUT(\mathsf{G},P)$, and $\partial^*\mathsf{G}=\partial \mathsf{G}\setminus\{\xi\}$ 
the set of the remaining ends. 
Recall a result which can be found in {\sc Möller}~\cite{Moeller1992} and 
{\sc Woess}~\cite{Woess_Amenable1989}, \cite{Woess1989}. 
\begin{theorem}
\mbox{The following hold.}
 \begin{enumerate}[(a)]
  \item $\Gamma$ is amenable and acts transitively on $\partial^{*}\mathsf{G}$.
  \item The structure tree $\mathcal{T}$ of $\mathsf{G}$ is a homogeneous tree with finite degree  $q+1\geq 3$.
  \item The structure map $\varphi:\widehat{\mathsf{G}}\to\widehat{\mathcal{T}}$ is onto, and its restriction
	to $\partial \mathsf{G}$ is a homeomorphism $\partial \mathsf{G}\to\partial \mathcal{T}$. There is an integer $a>0$,
	such that
	\begin{equation*}
	 d_{\mathcal{T}}(\varphi x,\varphi y)\leq d_{\mathsf{G}}(x,y)\leq a\big(d_{\mathcal{T}}(\varphi x,\varphi y)+1\big).
	\end{equation*}
         Therefore, $\mathsf{G}$ and $\mathcal{T}$ are quasi-isometric graphs, whose ends are in bijection. 
 \end{enumerate}
\end{theorem}
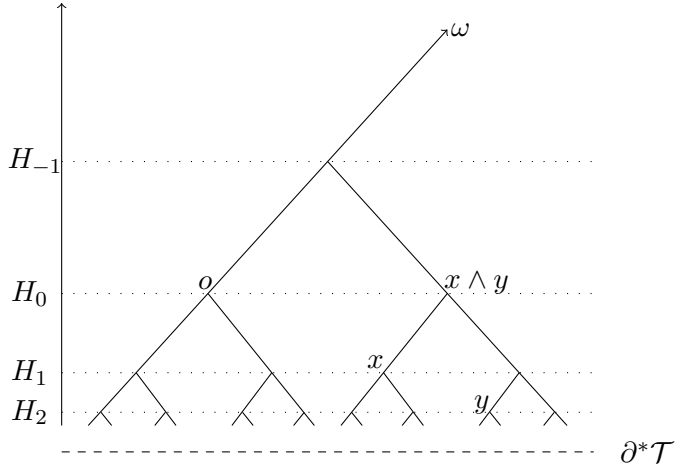
\begin{figure}[h]
\centering

\begin{tikzpicture}[scale=0.7]

\draw[dashed](0,0)-- (10,0);
\node (a) at (11,0) {$\partial^*\mathcal{T}$};

\draw[thin,->] (0,0.5)--(0,8.5);

\draw[->] (0.5,0.5)--(7.25,8);
\node (b) at (7.5,8) {$\omega$};

\draw (5,5.5)--(9.5,0.5);
\draw[loosely dotted] (0,5.5)--(10,5.5);
\draw[loosely dotted] (0,3)--(10,3);
\draw[loosely dotted] (0,1.5)--(10,1.5);
\draw[loosely dotted] (0,0.75)--(10,0.75);

\draw (2.75,3)--(4.75,0.5);
\draw (7.25,3)--(5.25,0.5);

\draw (1.4,1.5)--(2.15,0.5);
\draw (3.95,1.5)--(3.20,0.5);

\draw (6.05,1.5)--(6.80,0.5);
\draw (8.60,1.5)--(7.85,0.5);

\draw (0.74,0.75)--(0.94,0.5);
\draw (1.95,0.75)--(1.75,0.5);

\draw (3.4,0.75)--(3.60 ,0.5);
\draw (4.56,0.75)--(4.36,0.5);

\draw (5.45,0.75)--(5.65 ,0.5);
\draw (6.60,0.75)--(6.40,0.5);

\draw (8.05,0.75)--(8.25 ,0.5);
\draw (9.26,0.75)--(9.06,0.5);

\node (c) at (2.7,3.2){$o$};
\node (d) at (7.8,3.2){$x\wedge y$};
\node (e) at (7.9,0.9){$y$};
\node (f) at (5.9,1.7) {$x$};

\node (g) at (-0.6,0.75) {$H_2$};
\node (h) at (-0.6,1.5) {$H_1$};
\node (i) at (-0.6,3) {$H_0$};
\node (i) at (-0.5,5.5) {$H_{-1}$};

\end{tikzpicture}
\caption{The homogeneous tree $\mathcal{T}$ with a fixed end $\omega$}\label{fig:hom_tree_fixed_end}
\end{figure}
The interpretation of this result is that $\mathsf{G}$ is described - up to small modifications 
encoded by the structure map $\varphi$ - by its structure tree which looks like
in Figure \ref{fig:hom_tree_fixed_end}, with $\varphi(\xi)=\omega$. 

Let $\mathcal{T}$ be the homogeneous tree with a fixed end $\omega$, like in
Figure \ref{fig:hom_tree_fixed_end}, with $\partial \mathcal{T}$
its space of ends and $\partial ^{*}\mathcal{T}=\partial \mathcal{T}\setminus\{\omega\}$.  
Random walks on homogeneous trees with a fixed end are well studied by {\sc Cartwright, Kaimanovich 
and Woess}~\cite{Cartwright_Kaimanovich_Woess_1994}. In this section we want to
extend their results for lamplighter random walks on such graphs: we study the
convergence and the Poisson boundary of homogeneous random walks $\Zn$ on $\Z_2\wr\mathcal{T}$.  

Due to the quasi-isometry between $\mathsf{G}$ and its structure tree $\mathcal{T}$, 
and also to the fact that $\mathcal{T}$ can be represented
as a homogeneous tree with a fixed end, like in 
Figure \ref{fig:hom_tree_fixed_end}, one can replace $\mathsf{G}$ with $\mathcal{T}$ in $\lgr$.
The geometry of $\mathcal{T}$ is quite simple and the behaviour of random walks
on such trees is completely understood. 

\textit{Throughout this section, $\mathcal{T}$ represents the base graph for the lamplighter 
graph $\Z_2\wr\mathcal{T}$, and the goal is the study of random walks on $\Z_2\wr\mathcal{T}$.}
We reconsider briefly the homogeneous tree and its affine group and recall the main known results.

\paragraph*{Geometry of the Oriented tree.} 
If $\omega$ is the fixed reference end in $\partial\mathcal{T}$, 
set for all $x\neq y$ in $\widehat{\mathcal{T}}=\mathcal{T}\cup\partial \mathcal{T}$,
\begin{equation*}
 x\wedge y=\text{ first common vertex of } \pi(x,\omega) \text{ and } \pi(y,\omega),
\end{equation*}
where $\pi(x,\omega)$ is the geodesic ray starting at $x$ and ending at the fixed end $\omega$.
See once again Figure \ref{fig:hom_tree_fixed_end} for a graphic visualization.

Let $o$ be a reference vertex in $\mathcal{T}$ called \textit{origin}. Define the 
\textit{height function} $h:\mathcal{T}\to\mathbb{Z}$ by
\begin{equation}\label{eq:height_fc}
 h(x)=d(x,x\wedge o)-d(o,x\wedge o).
\end{equation}
This function is known in the literature as the \textit{Busemann function}, 
and represents the generation number of $x$. For $m\in\mathbb{Z}$,
the \textit{horocycle} at level $m$ is the infinite set
\begin{equation*}
 H_m=\{x\in\mathcal{T}:h(x)=m\}.
\end{equation*}
One can imagine the oriented tree $\mathcal{T}$ in Figure \ref{fig:hom_tree_fixed_end}
as an infinite genealogical tree, where $\omega$ represents the 
``mythical ancestor'' and every vertex $x\in H_m$ has an
unique ``father'' $x^{-}$ in $H_{m-1}$ and $q\geq 2$ ``sons'' $x_j$,
for $j=1,\ldots ,q$
in $H_{m+1}$. 
Using the height function, one can express the distance between $x$ and $y$ in
$\mathcal{T}$ as
\begin{equation*}
 d(x,y)=h(x)+h(y)-2h(x\wedge y).
\end{equation*}
Finally, define a bounded metric $\rho$ on $\widehat{\mathcal{T}}$, which
is an \textit{ultrametric} (that is, $\rho(x,y)\leq\max\{\rho(x,z),\rho(z,y)\}$, 
for all $x,y,z\in\widehat{\mathcal{T}}$):
\begin{equation*}
 \rho(x,y):=
\begin{cases}
q^{-d(o,x\wedge y)}  , & \text{ if } x\neq y \\
0 , & \text{ if } x=y
\end{cases}.
\end{equation*}
Then a sequence $(x_n)$ converges to $x$ in $\widehat{\mathcal{T}}$ if $\rho(x_n,x)$
tends to zero as $n$ tends to infinity.

\paragraph{The Affine Group $\AFF(\mathcal{T})$ of a Tree.}
The \textit{affine group} of a tree $\mathcal{T}$ is the group $\AFF(\mathcal{T})$
of all isometries $\gamma$ which fix $\omega$. Changing the reference 
end $\omega$ means passing to a conjugate of this group. The name is chosen because
of the analogy with the Poincar\'{e} upper half plane, where the group of
all isometries which fix the point at infinity coincides with the affine group of the
real line. The affine group of an oriented tree and random walks on it was also
studied by {\sc Brofferio}~\cite{Brofferio2004}. 

\paragraph{Random Walks on $\mathcal{T}$ and on its Affine Group $\AFF(\mathcal{T})$.}
Random walks on the oriented tree and on its affine group are well studied by
{\sc Cartwright, Kaimanovich and Woess} \cite{Cartwright_Kaimanovich_Woess_1994}. 
For random walks on these type of trees, they have obtained the following results:
\begin{itemize}
\item Convergence to the boundary $\partial\mathcal{T} $, 
and hence, existence of a harmonic measure on $\partial\mathcal{T} $.
\item The solution of the Dirichlet problem at infinity.
\item Law of large numbers and central limit theorem, formulated with 
respect to two natural length functions on the affine group of $\mathcal{T}$.
\item Identification of the Poisson boundary, that is, a description
of bounded harmonic functions for random walks on $\mathcal{T}$.
\end{itemize}

In order to state the results on convergence and Poisson boundary
of lamplighter random walks $\Zn$ on $\mathbb{Z}_2\wr\mathcal{T}$, 
with  $Z_n=(\eta_n, X_n)$, similar results for random 
walks $\Xn$ on $\mathcal{T}$ are needed. For complete proofs and more details
see once again {\sc Cartwright, Kaimanovich and Woess}~\cite{Cartwright_Kaimanovich_Woess_1994}.

Consider an irreducible random walk $\Xn$ with transition matrix $P_{\mathcal{T}}$ on $\mathcal{T}$, 
and $\Gamma\subset \AFF(\mathcal{T})$ which fixes the end $\omega\in\partial \mathcal{T}$.
Like before, let $\mu$ be the probability measure on $\Gamma$, which 
is uniquely induced (recall the notation $P_{\mathcal{T}}\leftrightarrow\mu$) 
by the transition probabilities
$P_{\mathcal{T}}$ of $\Xn$, like in \eqref{eq:correspondence_rw}. The measure $\mu$ determines
the right random walk $(\Gamma,\mu)$ on $\Gamma$.

Note that if $(X_n)$ is a random walk on the oriented tree $\mathcal{T}$, then
$h(X_n)$ is a random walk on the integer line $\mathbb{Z}$, where $h$ is the height
function defined in \eqref{eq:height_fc}. Indeed, the mapping
\begin{align*}
 \Psi :\mathcal{T} & \to \mathbb{Z}\\
 x & \mapsto h(x)
\end{align*}
induces a projection of any random walk on $\mathcal{T}$ onto a walk
on $\mathbb{Z}$.

Define the \index{random wakl!drift}\emph{modular drift} of the random walk 
$(X_n)$ on $\mathcal{T}$ with transition matrix $P_{\mathcal{T}}$, as
\begin{equation*}
 \delta(P_{\mathcal{T}})=\sum_{x\in\mathcal{T}}h(x)p_{\mathcal{T}}(o,x).
\end{equation*}
The reversed random walk $(\check{X}_n)$ on $\mathcal{T}$ has the transition matrix 
$\check{P}_{\mathcal{T}}$, whose entries are given by
\begin{equation*}
 \check{p}_{\mathcal{T}}(x,y)=p_{\mathcal{T}}(y,x)q^{h(x)-h(y)}.
\end{equation*}
Note that if $\delta(P_{\mathcal{T}})<0$, than $\delta(\check{P}_{\mathcal{T}})>0$ and the other way round.
The following is known from {\sc Cartwright et. al}~\cite{Cartwright_Kaimanovich_Woess_1994}
and for the case $\delta(P_{\mathcal{T}})=0$ from {\sc Brofferio}~\cite{Brofferio2004}.

\begin{theorem}
\label{thm:conv_poisson_tree_fixed_end}
If $(X_n)$ is irreducible and has finite first moment on $\mathcal{T}$, then:
\begin{enumerate}[(a)]
\item If $\delta(P_{\mathcal{T}})>0$, then $(X_n)$
converges almost surely to a random end 
$X_\infty\in\partial^*\mathcal{T}=\partial\mathcal{T}\setminus \{\omega\}$.
If $\mu_{\infty}$ is the distribution
of $X_{\infty}$ on $\partial^* \mathcal{T}$, we have that 
$\supp(\mu_{\infty})=\partial^* \mathcal{T}$
and $(\partial^* \mathcal{T},\mu_{\infty})$ is the Poisson boundary of $(X_n)$.
\item If $\delta(P_{\mathcal{T}})\leq 0$, then $(X_n)$ converges to the fixed end $\omega$ almost surely, 
and the Poisson boundary of $(X_n)$ is trivial.
\end{enumerate}
\end{theorem}
The case $\delta(P_{\mathcal{T}})=0$ is a special one, since the projection $h(X_n)$ of
$X_n$ is the simple random walk on $\mathbb{Z}$, which is reccurent. Nevertheless,
the random walk $(X_n)$ is transient and it converges to the fixed end
$\omega$, but its Poisson boundary is trivial.

The Poisson boundary in the previous result can be described by using
the Strip Criterion \ref{thm:strip_crit}, in the case of positive 
(negative, respectively) drift. For completeness,
we shall give here the idea of the realization of the Poisson boundary 
for random walks on trees $\mathcal{T}$ with a fixed end $\omega$.
For more details, see once again 
{\sc Cartwright, Kaimanovich and Woess}~\cite{Cartwright_Kaimanovich_Woess_1994}.

\begin{proof}[Idea of the proof of Theorem \ref{thm:conv_poisson_tree_fixed_end}.]
If the modular drift $\delta(P_{\mathcal{T}})=0$, then also the drift (rate of escape) $l(P_{\mathcal{T}})=0$, 
and the triviality of the Poisson boundary follows from 
Proposition \ref{prop:rate_esc_poisson}. 
\begin{enumerate}[(a)] 
\item If $\delta(P_{\mathcal{T}})>0$, then $\delta(\check{P}_{\mathcal{T}})<0$. Also 
$X_n\to X_{\infty}\in\partial^{*}\mathcal{T}$ and $\check{X}_{n}\to\omega$, almost
surely. Then $(\partial^{*}\mathcal{T},\mu_\infty)$ is a $\mu$-boundary,
where, as usual $\mu$ is associated with $P_{\mathcal{T}}$ like in \ref{eq:correspondence_rw},
and $(\{\omega\},\delta_{\omega})$ is a $\check{\mu}$-boundary.

Thus, we can apply the Strip Criterion \ref{thm:strip_crit}, and choose
the geodesic lines between $\mathfrak{u}\in\partial^{*}\mathcal{T}$
and $\omega$ as the strips $\mathfrak{s}(\mathfrak{u},\omega)$. Measurability of
the map $\mathfrak{u}\mapsto \mathfrak{s}(\mathfrak{u},\omega)$ is obvious,
and 
\begin{equation*}
\gamma \mathfrak{s}(\mathfrak{u},\omega)=\mathfrak{s}(\gamma\mathfrak{u},\omega),
\end{equation*}
for every $\gamma\in \Gamma$ and $\mathfrak{u}\in\partial^*\mathcal{T}$.
Therefore the strips $\mathfrak{s}(\mathfrak{u},\omega)$ satisfy the conditions
required in Theorem \ref{thm:strip_crit}, and the measure space $(\partial^*\mathcal{T},\mu_{\infty})$
is the Poissson boundary of the random walk $\Xn$ with transition matrix $P_{\mathcal{T}}$ over $\mathcal{T}$.

\item If $\delta(P_{\mathcal{T}})< 0$, we exchange the roles of $P_{\mathcal{T}}$ and $\check{P}_{\mathcal{T}}$
in the first case and we get the triviality of the Poisson boundary of
$\Xn$.
\end{enumerate} 
\end{proof}
\paragraph{Lamplighter Random Walks on $\mathbf{\Z_2}\wr \mathbf{\mathcal{T}}$.}
Consider lamplighter random walks $\Zn$ on $\mathbb{Z}_2\wr \mathcal{T}$,
with $Z_n=(\eta_n,X_n)$ and $\mathcal{T}$ is the tree with the fixed end $\omega$ 
represented in Figure \ref{fig:hom_tree_fixed_end}. Assume that $\Xn$
has finite first moment on $\mathcal{T}$.
The convergence of $\Zn$ is a simple consequence of Theorem 
\ref{thm_conv_lrw_general_graphs}.

Recall the Definition \ref{eq:pi_boundary} of the geometric
boundary $\Pi$ of the lamplighter graph $\lgr$. 
We replace in this section the graph $\mathsf{G}$ with the tree $\mathcal{T}$
with a fixed end $\omega$, and the boundary $\partial \mathsf{G}$
with $\partial\mathcal{T}=\partial^*\mathcal{T}\cup\{\omega\}$.
We can then rewrite $\Pi$ as
\begin{equation*}
 \Pi=\Pi^*\cup \omega^*,
\end{equation*}
where $\Pi^*$ is the set of all pairs $(\zeta,\mathfrak{u})$,
that is,
\begin{equation}\label{eq:pi*_boundary}
 \Pi^*=\bigcup_{\mathfrak{u}\in\partial^*\mathcal{T}}\mathcal{C}_{\mathfrak{u}}\times \{\mathfrak{u}\},
\end{equation}
and $\mathcal{C}_{\mathfrak{u}}$ is the set of all configurations $\zeta$ which are either
finitely supported, or infinitely supported with $\supp(\zeta)$
accumulating only at $\mathfrak{u}$. Also
\begin{equation}\label{eq:omega*}
\omega^* =\{ (\zeta,\omega):\zeta\in\mathcal{C}_{\omega}\}.
\end{equation}
\begin{theorem}
\label{thm:conv_LRW_fixed_end}
Let $\Zn$ be an irreducible, homogeneous random walk with finite first 
moment on $\mathbb{Z}_2\wr\mathcal{T}$, where $\mathcal{T}$ is an 
homogeneous tree and $\Gamma\subset \AFF(\mathcal{T})$.
\begin{enumerate}[(a)]
\item If $\delta(P_{\mathcal{T}})>0$, then there exists a $\Pi^*$-valued random variable $Z_{\infty}$,
such that $Z_n\to Z_{\infty}$ almost surely, in the topology of $\widehat{\mathbb{Z}_2\wr \mathcal{T}}$.
\item If $\delta(P_\mathcal{T})\leq 0$, then $\Zn$ converges almost surely to some $\omega^*$-valued 
random variable, in the topology of $\widehat{\mathbb{Z}_2\wr \mathcal{T}}$.
\end{enumerate}  
The distribution of $Z_{\infty}$ on $\Pi^*$ (on $\omega^*$ respectively) is a 
continuous measure.
\end{theorem}
\begin{proof}
(a) The result is an application of 
Theorem \ref{thm_conv_lrw_general_graphs},
which holds for general transitive base graphs $\mathsf{G}$ endowed with a rich boundary and such that
Assumption \ref{assumptions_brw} holds. 
From Theorem \ref{thm:conv_poisson_tree_fixed_end} it follows
that $\Xn$ satisfies Assumption \ref{assumptions_brw},
with the boundary $\partial \mathsf{G}:=\partial^*\mathcal{T}$.
Then $\Pi^*$ is the boundary for the lamplighter random walk
and $\Zn$ converges to $\Pi^*$ almost surely.

The limit distribution of $X_n$ is a continuous measure on
$\partial^*\mathcal{T}$, and this implies the continuity
of the limit distribution of $Z_n$ on $\Pi^*$.

(b) This is again an application of Theorem
\ref{thm_conv_lrw_general_graphs} when $\delta(P_{\mathcal{T}})<0$, with $\omega$ in the place of 
$\partial \mathsf{G}$ and $\omega^*$ defined in \eqref{eq:omega*} in the place of $\Pi$.
Assumption \ref{assumptions_brw} holds once again for $(X_n)$
and the requirements of Theorem \ref{thm_conv_lrw_general_graphs}
are fulfilled. This proves the desired.

For $\delta(P_{\mathcal{T}})=0$, also the drift of the base random walk is zero,
and Theorem \ref{thm_conv_lrw_general_graphs} can be easily adapted to prove the
convergence of the LRW $\Zn$ to a random variable in $\omega^*$.

When the base random walk $\Xn$ 
converges to $\omega$ almost surely, then the limit distribution is 
the point mass $\delta_{\omega}$ at $\omega$. Nevertheless, 
using Borel-Cantelli lemma, 
one can prove that the limit distribution $\nu_{\infty}$ of $\Zn$ 
on $\omega^*$ is a continuous measure. 
\end{proof}

\paragraph{The Poisson Boundary.}
The main goal of this subsection is to describe 
the Poisson boundary of
lamplighter random walks $\Zn$ on $\mathbb{Z}_2\wr \mathcal{T}$, where
$\mathcal{T}$ is a homogeneous tree with a fixed end $\omega$, like in
Figure \ref{fig:hom_tree_fixed_end}. This will be done
by making use of \textit{Half Space Method}
in the case when the base random walk $\Xn$ has non-zero 
modular drift. 

The case $\delta(P_{\mathcal{T}})=0$ is slightly different
and is also the most difficult one, and it will be treated
separately in the following section. We emphasize that
the proof is completely different from all previous ones and is based 
on the existence of cutpoints for random walks. I am
very grateful to Vadim Kaimanovich for useful disscutions
on this problem, and for the main idea of Section \ref{sec:zero_drift}. 

The correspondence between the tail $\sigma$- algebra 
of a random walk and its Poisson boundary will be used.
Nevertheless, we are able to solve this problem only when
the base random walk $\Xn$ is of nearest neighbour type.
The general case of random walks over $\mathbb{Z}_2\wr \mathcal{T}$ 
such that the projected random walk 
$(X_n)$ on $\mathcal{T}$ has zero drift and bounded range
(not only range $1$ like in our approach) is awaiting future work.

\begin{remark}
{\sc Erschler}~\cite{Erschler2010} proved recently 
that the Poisson boundary of lamplighter random walks over Euclidean 
lattices $\mathbb{Z}^d$, with $d\geq 5$, such that the projection on $\mathbb{Z}^d$ has 
zero drift, is isomorphic with the space of infinite limit configurations
of lamps. She uses a modified version of the Ray Criterion \ref{thm:ray_crit}. 

Her methods does not apply in our case, when the underlying tree 
$\mathcal{T}$ has a fixed end, and the base walk has zero drift.
\end{remark}

For the Poisson boundary of lamplighter random walks over homogeneous
trees with the action of a transitive group without any fixed ends,
see the paper of {\sc Karlsson and Woess}~\cite{KarlssonWoess2007}.
They proved that the Poisson boundary is the space of
infinite limit configurations accumulating at boundary points (ends) 
of the tree. 

Recall that $P_{\mathcal{T}}$ is the transition matrix of the base random walk 
$\Xn$ on $\mathcal{T}$, and $\delta(P_{\mathcal{T}})$ its modular drift.

\begin{theorem}
\label{thm:poisson_lrw_fixed_end}
Let $\Zn$ be an irreducible, homogeneous random walk with 
finite first moment on $\mathbb{Z}_2\wr\mathcal{T}$,
and $\Gamma\subset\AFF(\mathcal{T})$.
\begin{enumerate}[(a)]
\item If $\delta(P_{\mathcal{T}})>0$, then $(\Pi^*,\nu_{\infty})$
is the Poisson boundary of $\Zn$, where $\Pi^*$ is given as in \eqref{eq:pi*_boundary} 
and $\nu_{\infty}$ is the limit distribution of $\Zn$ on $\Pi^*$.
\item If $\delta(P_{\mathcal{T}})<0$, then $(\omega^*,\nu_{\infty})$  is the Poisson boundary
of $\Zn$, where $\omega^*$ is given as in \eqref{eq:omega*} and $\nu_{\infty}$ is the limit
distribution of $\Zn$ on $\omega^*$.
\end{enumerate} 
\end{theorem}
The case $\delta(P_{\mathcal{T}})=0$ is excluded here, and considered separately in 
Section \ref{sec:zero_drift}.

Note that this result gives both the Poisson boundary of $\Zn$ 
(with transition matrix $P$) and of the reversed random walk $(\check{Z}_n)$
(with transition matrix $\check{P}$). The finite
first moment condition, the irreducibility and the homogeneity
hold for both $\Zn$ and $(\check{Z}_n)$, simultaneously.

\begin{proof}[Proof of Theorem \ref{thm:poisson_lrw_fixed_end}]
Apply \textit{Half Space Method} \ref{sec:half_space_method} and Theorem \ref{PoissonTheorem}.

(a) Case $\delta(P_{\mathcal{T}})>0$: In order to apply Theorem 
\ref{PoissonTheorem}, we check again that the conditions required 
in \textit{Half Space Method} \ref{sec:half_space_method} are satisfied
for the oriented tree $\mathcal{T}$ and random walks $\Xn$ and $(\check{X}_n)$ on it. 

From the finite first moment assumption and 
Theorem \ref{thm:conv_poisson_tree_fixed_end} it follows that
Assumption \ref{assumptions_brw} hold for $\Xn$, $(\check{X}_n)$ respectively.
By Theorem \ref{thm:conv_poisson_tree_fixed_end}, 
$\Xn$ converges to a random end in $\partial^*\mathcal{T}$
and $(\check{X}_n)$ converges to the fixed end $\omega$. 
Denote by $\mu_{\infty}$ the limit distribution of $\Xn$ on $\partial^*\mathcal{T}$.
The limit distribution of $(\check{X}_n)$ is the point mass $\delta_{\omega}$
at $\omega$.

Next, assign a strip 
$\mathfrak{s}(\mathfrak{u},\omega)\subset \mathcal{T}$ to almost 
every end $\mathfrak{u}\in\partial^*\mathcal{T}$.
Consider the strip $\mathfrak{s}(\mathfrak{u},\omega)\subset \mathcal{T}$
\begin{equation*}
\mathfrak{s}(\mathfrak{u},\omega)=\pi(\mathfrak{u},\omega),
\end{equation*}
where $\pi(\mathfrak{u},\omega)$ is the unique two-sided infinite geodesic between 
$\mathfrak{u}$ and $\omega$, which has linear growth and
is equivariant with respect to the action of $\Gamma$
on $\mathcal{T}$. Therefore it satisfies the conditions required in Strip
Criterion \ref{thm:strip_crit}.

The partition of $\mathcal{T}$ in half-spaces is done like this:
for every $x\in\mathfrak{s}(\mathfrak{u},\omega)$ let
$\mathcal{T}_{+}(x)$ be the unique connected component
which contains the end $\mathfrak{u}$ after the removal of
$x$ from $\mathcal{T}$, and $\mathcal{T}_{-}(x)$ be its complement in $\mathcal{T}$.
Then
\begin{equation*}
\mathcal{T}_{+}=\mathcal{T}_{+}(x) \text{ and } \mathcal{T}_{-}=\mathcal{T}\setminus \mathcal{T}_{+},
\end{equation*}
are $\Gamma$-equivariant sets. Thus, all requirements
for the \textit{Half Space Method} \ref{sec:half_space_method} hold
and we can apply Theorem \ref{PoissonTheorem}. 

By Theorem \ref{thm:conv_LRW_fixed_end}
the random walk $\Zn$ ($(\check{Z}_n)$, respectively)
converges to a random element in $\Pi^*$ ($\omega^*$, respectively)
with limit distributions $\nu_{\infty}$ ($\check{\nu}_{\infty}$, respectively).
Then $(\Pi^*,\nu_{\infty})$ and $(\omega^*,\check{\nu}_{\infty})$
are $\nu$- and $\check{\nu}$-boundaries of the respective random walks.
Take
\begin{equation*}
b_{+}=(\phi_{+},\mathfrak{u})\in\Pi^*,\text{ and } b_{-}=(\phi_{-},\omega)\in\omega^* ,
\end{equation*} 
where $\phi_{+}$ and $\phi_{-}$ are the limit configurations of $\Zn$ and 
$(\check{Z}_{n})$, respectively, and $\mathfrak{u}\in\partial^*\mathcal{T},\ \omega$ 
are their only respective accumulation points. Define the configuration 
$\Phi(b_{+},b_{-},x)$ by
\begin{equation*}
\Phi(b_{+},b_{-},x)=
\begin{cases}
\phi_{-}, & \mbox{on}\  \mathcal{T}_{+}\\
\phi_{+}, & \mbox{on}\ \mathcal{T}_{-}
\end{cases} 
\end{equation*}
and the strip $S(b_{+},b_{-})$ like in equation \eqref{eq:lamplighter_strip}:
\begin{equation*}
S(b_{+},b_{-})=\{\left(\Phi,x\right) :\  x\in\mathfrak{s}(\mathfrak{u},\omega)\}.
\end{equation*}
By Theorem \ref{PoissonTheorem}, $S(b_{+},b_{-})$ satisfies the 
conditions from Theorem \ref{thm:strip_crit}, and it follows 
that the space $(\Pi^*,\nu_{\infty})$ is the Poisson boundary of 
$\Zn$ and $(\omega^*,\check{\nu}_{\infty})$ is the Poisson boundary of
$(\check{Z}_n)$ over $\Z_2\wr\mathcal{T}$.

(b) Case $\delta(P_{\mathcal{T}})<0$: The proof is like before
with the roles of $P_{\mathcal{T}}$ (respectively $P$) and $\check{P}_{\mathcal{T}}$ (respectively
$\check{P}$) exchanged.
\end{proof}

\section{Zero-Drift Random Walks on the Oriented Tree}
\label{sec:zero_drift}

In this section we describe the Poisson boundary of
lamplighter random walks $\Zn$ on $\Z_2\wr \mathcal{T}$, 
where $\mathcal{T}$ is the oriented tree 
of degree $q+1\geq 3$ given in Figure \ref{fig:hom_tree_fixed_end} (with $q=2$ in the picture),
such that the base random walk $\Xn$ on $\mathcal{T}$ has modular drift 
$\delta(P_{\mathcal{T}})=0$. 

Suppose that $\Xn$ is a nearest neighbour random walk on 
$\mathcal{T}$ with transition probabilities $P_{\mathcal{T}}=\big(p_{\mathcal{T}}(x,y)\big)$ given by
\begin{equation}\label{eq:tr_pb_fixed_end}
 p_{\mathcal{T}}(x,y)=
\begin{cases}
\frac{1}{2}  , &  \text{ if } y=x^{-} \\
\frac{1}{2q} , &  \text{ if } y=x_j,
\end{cases}
\end{equation}
where $x_j$ is one of the q ``sons'' of $x$, and $x^-$ is the father of $x$. Then the horocyclic projection 
$h(X_n)$ of $X_n$ on the integer line $\mathbb{Z}$ is the simple random walk
on $\mathbb{Z}$, which is recurrent. Nevertheless, the random walk
$\Xn$ on $\mathcal{T}$ is transient and converges to $\omega$ by Theorem 
\ref{thm:conv_poisson_tree_fixed_end}.
Also, the random walk $\Zn$ on $\mathbb{Z}_2\wr\mathcal{T}$ converges to 
$(\zeta_{\infty},\omega)\in\omega^*$
by Theorem \ref{thm:conv_LRW_fixed_end}, where $\zeta_{\infty}$ is 
the limit configuration (not necessary with finite support) of the LRW. 

We shall prove that the Poisson boundary
of LRW $\Zn$, with $Z_n=(\eta_n,X_n)$, such that the  base
random walk $\Xn$ has transition probabilities
given by \eqref{eq:tr_pb_fixed_end}, is described by the space of limit
configurations $\zeta_{\infty}$ accumulating at $\omega$. 
For doing this, the Strip and Ray Criterion
(which use in some sense the entropy of the conditional
random walk) are not suitable and another approach is needed.

The proof will be done in several steps, and the first one
is to give to the oriented tree in Figure \ref{fig:hom_tree_fixed_end}
another geometric interpretation, which will make things easier.

Let us consider the reference point $o\in\mathcal{T}$ and the 
one-sided infinite geodesic $\pi(o,\omega)$ joining $o$
with the fixed end $\omega$. Then one can interpret the tree
$\mathcal{T}$ in Figure \ref{fig:hom_tree_fixed_end} as the 
infinite geodesic $\pi(o,\omega)$, which is isomorphic with $\mathbb{Z}_+$, 
with a tree $\mathcal{T}_k$ attached at each point 
$k\in\mathbb{Z}_{+}\cong \pi(o,\omega)$, so that
$k$ is the origin of $\mathcal{T}_k$. For a graphic visualization, see Figure \ref{fig:tree_zero_drift}.

\begin{figure}[h]
\begin{tikzpicture}[scale=0.95]

\draw[->] (0,0) -- (11,0);
\filldraw (0,0) circle (1.5pt);
\filldraw (3,0) circle (1.5pt);
\filldraw (6,0) circle (1.5pt);
\filldraw (9,0) circle (1.5pt);
\draw (-0.25,0) node {$0$};
\draw (11.25,0) node {$\infty$};
\draw (0,-1) node {$\mathcal{T}_0$};
\draw (2.7,-1) node {$\mathcal{T}_1$};
\draw (5.7,-1) node {$\mathcal{T}_2$};
\draw (8.7,-1) node {$\mathcal{T}_3$};
\draw (0,0)--(-1.7,-2.5);
\draw (0,0)--(1.7,-2.5);
\draw (-1,-1.5)--(-0.30,-2.5);
\draw (1,-1.5)--(0.30,-2.5);
\draw[thin,dashed,->] (3.2,0.2)--node[above] {$1/2$}(4.1,0.2);
\draw[thin,dashed,->] (2.8,0.2)--node[above] {$1/4$}(1.9,0.2);
\draw[thin,dashed,->] (3.2,-0.2)--node[right] {$1/4$}(3.2,-1.2);
\draw (3,0.3) node {$1$};
\draw (6,0.3) node {$2$};
\draw (9,0.3) node {$3$};
\draw (3,0)--(3,-1.5);
\draw (3,-1.5)--(1.8,-3);
\draw (3,-1.5)--(4.2,-3);
\draw (2.2,-2.5)--(2.6,-3);
\draw (3.8,-2.5)--(3.4,-3);

\draw (6,0)--(6,-1.5);
\draw (6,-1.5)--(4.8,-3);
\draw (6,-1.5)--(7.2,-3);
\draw (5.2,-2.5)--(5.6,-3);
\draw (6.8,-2.5)--(6.4,-3);

\draw (9,0)--(9,-1.5);
\draw (9,-1.5)--(7.8,-3);
\draw (9,-1.5)--(10.2,-3);
\draw (8.2,-2.5)--(8.6,-3);
\draw (9.8,-2.5)--(9.4,-3);
\end{tikzpicture}

\caption{Another interpretation of the oriented tree}\label{fig:tree_zero_drift}
\end{figure}
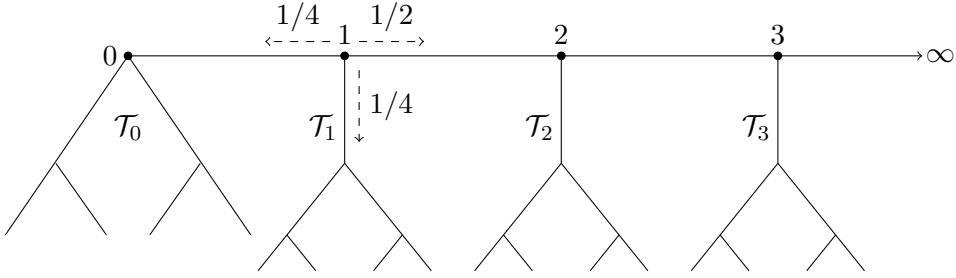
Consider now the following stopping times (exit times)
\begin{equation}\label{eq:stop_times}
\tau_0=0 \text{ and }\tau_{m+1}=\min\{n>\tau_m:X_n\in\mathcal{T}\setminus\mathcal{T}_{\tau_m}\}, \text{ for }m\geq 1.
\end{equation}
If $X_{\tau_m}=k$, this means that for $n\in[ \tau_m,\tau_{m+1}]$, the random 
walk $\Xn$ will move only in the tree $\mathcal{T}_k$, and $X_{\tau_{m+1}}\in\{k-1,k+1\}$.
In other words, the exit time $\tau_{m+1}$ represents the first time when the 
random walk $\Xn$ leaves the attached tree $\mathcal{T}_k$, with $X_{\tau_m}=k$. 

The random walk $\Xn$ restricted to the stopping times $\tau_m$
is again a nearest neighbour random walk (Markov chain) on the positive 
integer line $\mathbb{Z}_{+}$. For simplicity of notation, let us denote by 
$Y_n=X_{\tau_{n}}$ and by 
\begin{equation*}
Q=\big(q(i,j)\big)_{i,j\in\Z_{+}} 
\end{equation*}
the transition matrix of $(Y_n)$. Then the entries of $Q$ are 
\begin{equation*}
q(i,j)=\mathbb{P}[Y_{n+1}=i|Y_n=j]=\mathbb{P}[X_{\tau_{n+1}}=i|X_{\tau_n}=j], \text{ for }i,j\in\Z_+.
\end{equation*}
We have $Y_0=0$ and $Y_1=1$, and the transition probabilities of $(Y_n)$ can be easily computed as
\begin{equation}\label{eq:tr_pb_z}
 q(i,i+1)=\frac{q}{q+1},\text{ and } q(i,i-1)=\frac{1}{q+1}, \text{ for } i\geq 1
\end{equation}
where $q+1$ is the degree of $\mathcal{T}$, with $q\geq 2$, and
\begin{equation*}
q(0,1)=1.
\end{equation*}
All the other entries $q(i,j)$ with $j \notin \{i-1,i+1\}$ are zero, since $(X_n)$
is a nearest neighbour random walk. This random walk is transient and 
almost every path of $(Y_n)$ converges almost surely to $\infty$. 
This gives also another explanation of the transience of the random walk 
$\Xn$ on $\mathcal{T}$ and of the convergence to the fixed end $\omega$.

Next, we introduce some definitions and facts,
which can be found in {\sc James and Peres}~\cite{JamesPeres1997}.

\paragraph{Cut Points.}
Let $(X_j)$ be a Markov chain on a countable state space $\mathsf{G}$.
A \index{Markov chain!cut time}\emph{cut time} for the Markov chain $(X_j)$
is an integer $n$ with
\begin{equation*}
 X_{[0,n]}\cap X_{[n+1,\infty)}=\emptyset,
\end{equation*}
where $X_{[0,n]}=\{X_j:0\leq j\leq n\}$, and in this case the 
random variable $X_n$ is called a 
\index{Markov chain!cut point}\emph{cut point}. 
A cutpoint is a point in the state space $\mathsf{G}$ which is visited exactly once
by the random walk. If $X_n$ is a cut point, then deleting $X_n$ from the trajectory
cuts the trajectory into two disjoint components, 
$X_{[0,n)}$ and $X_{[n+1,\infty)}$, with no possible transition from the
first to the second.

The following result will be needed. For details, see {\sc James and Peres}
\cite[Theorem 1.2 (a)]{JamesPeres1997}.

\begin{theorem}\label{thm:inf_cutpoints}
Any transient random walk with bounded increments on the lattice $\mathbb{Z}^d$
has infinitely many cut points with probability $1$.
\end{theorem}
For the random walk $(Y_n)$ on $\mathbb{Z}_+$ with transition probabilities
$q(i,j)$ given in \eqref{eq:tr_pb_z}, consider the random time
\begin{equation*}
 \xi_n=\inf\{j:|Y_j|\geq n\}.
\end{equation*}
Then one can show that with probability one $\xi_n$ is a cut time for
infinitely many $n$. Thus, the random walk $(Y_n)$ on $\Z_+$
has infinitely many cut times, and implicitely also infinitely many cut points. 

For the random walk $(Y_n)$, let $I_n$ be the indicator function
of the event $\{ n \text{ is a cut time}\}$ and let 
\begin{equation*}
R_n=\sum_{j=0}^n I_{j} 
\end{equation*}
be the number of cut times encountered up to time $n$. Then $R_n$ is a stationary
process and \textit{Birkhoff Ergodic Theorem} implies that 
there exists a constant $p>0$, such that 
\begin{equation*}
\lim_{n\to\infty}\frac{R_n}{n}=p>0.
\end{equation*}
The constant $p$ can be explicitely computed as $p=\mathbb{P}[Y_n\geq 1,\forall n\geq 1]$. 
The number $p>0$ is called the 
\index{Markov chain!cut point!density of}\textit{density of the cut points}
for the random walk $(Y_n)$ on $\mathbb{Z}_+$.

\begin{remark}\label{rem:inf_cut_points}
The set of cut points of $(Y_n)$ has strictly positive 
density $p$ and this implies that for any integer $N>1/p$, 
one can choose a set of trajectories, such that, in any collection 
of $N$ trajectories, at least $2$ have infinitely many common cut points. 
\end{remark}

\paragraph*{The Poisson Boundary.}
Assume that the lamplighter random walk $\Zn$ is of \textit{nearest neighbour type} 
on $\Z_2\wr\mathcal{T}$, that is, the following
\textit{local condition} is satisfied:
\begin{equation}\label{eq:local_cond}
p\big((\eta,x)(\eta',x')\big)>0,
\end{equation}
only if $x$ and $x'$ are neighbours in $\mathcal{T}$ and the configurations $\eta,\eta'$
may only differ at the point $x$.
 
Let us now state the main result on the description of the Poisson boundary
of LRW $\Zn$, with $Z_n=(\eta_n,X_n)$, when the base random walk $\Xn$ on the tree $\mathcal{T}$
with a fixed end $\omega$ has zero modular drift $\delta(P_\mathcal{T})$.

\begin{theorem}\label{thm:poisson_fixed_end_zero_drift}
Let $\Zn$ be an irreducible, nearest neighbour random walk on 
$\mathbb{Z}_2\wr\mathcal{T}$, such that the base walk $\Xn$ on $\mathcal{T}$
is of nearest neighbour type  with $\delta(P_{\mathcal{T}})=0$. 
Then the Poisson boundary of $\Zn$ is given by
the space of limit configurations
endowed with the corresponding hitting distribution $\nu_{\infty}$.
\end{theorem}
The first step in the proof of this theorem is to construct
another ``lamplighter type'' random walk $(\tilde{Z}_n)$ - on a modified 
state space - with the same Poisson boundary as $\Zn$. Then, Theorem 
\ref{thm:poisson_fixed_end_zero_drift} is
just a consequence of three other results which we will state
and prove in what follows.

\paragraph{Extended ``Lamplighter Type'' Random Walk.}

Using the stopping times $\tau_m$ defined in \eqref{eq:stop_times}, 
we construct another ``lamplighter type '' random walk $(\tilde{Z}_n)$, 
similar to $\Zn$. 

For all $k\in\Z_+$, denote by $\mathcal{C}_k$ the set of all
$\Z_2$-valued lamp configurations on the tree $\mathcal{T}_k$, rooted
at $k$. Let also $\tilde{\mathcal{C}}$ be the space of all \emph{finitely 
supported ``generalized'' configurations} $\Phi$ over $\Z_+$, where $\Phi$ is given by
\begin{align*}
 \Phi :\Z_+ & \to \mathcal{C}_k\\
  k & \mapsto \Phi(k)\in\mathcal{C}_k.
\end{align*}
This means that the value $\Phi(k)$ of the  
configuration $\Phi$ at a point $k\in\Z_+$ is a configuration 
 on the whole tree $\mathcal{T}_k$ rooted at $k$,
i.e., a configuration in $\mathcal{C}_k$.

Let the ``new'' state space $\tilde{Z}_+$ be the product of the space 
$\tilde{\mathcal{C}}$ of ``generalized'' configurations $\Phi$
with $\Z_+$, that is,
\begin{equation*}
\tilde{Z}_+=\tilde{\mathcal{C}}\times\Z_+ .
\end{equation*}
Consider now the random walk $(\tilde{Z}_n)$ on $\tilde{Z}_+$,
with $\tilde{Z}_n=(\Phi_n, Y_n)$, such that $\Phi_n$ is a generalized configuration 
over $\Z_+$, and the projection of $(\tilde{Z}_n)$ on $\Z_+$ is the
random walk $(Y_n)$ on $\Z_+$, with transition probabilities given 
as in \eqref{eq:tr_pb_z}. This description of the chain $(\tilde{Z}_n)$
follows from the fact that in between of two consecutive stopping times
$\tau_m,\tau_{m+1}$, the random walk $\Xn$ moves only in the tree
$\mathcal{T}_k$, with 
\begin{equation*}
Y_m=X_{\tau_m}=k \quad \text{ and }\quad \Phi_m=\eta_{\tau_m}. 
\end{equation*}
Moreover, because of the local condition we have assumed in \eqref{eq:local_cond},
the lamplighter configuration may change only in points of $\mathcal{T}_k$.

The walk $(\tilde{Z}_n)$, with $\tilde{Z}_n=(\Phi_n, Y_n)$,
can be viewed as an ``extended random walk'' of $(Y_n)$.
In other words, we add to the states of the chain $(Y_n)$ the lamp configuration
$\Phi_n$. The lamp configurations $\Phi_n$ are similar to the \textit{occupation numbers} 
defined in {\sc James and Peres}~\cite{JamesPeres1997}.

We denote by $\tilde{Q}$ the transition  matrix of $(\tilde{Z}_n)$. Its entries are
of the form
\begin{equation*}
 \tilde{q}\big((\phi_1,k_1),(\phi_2,k_2)\big).
\end{equation*}
For $k_2\notin\{k_1-1,k_1+1\}$, the entries are zero.
Let $n_1$ and $n_2$ be such that 
\begin{equation*}
 k_1=Y_{n_1}=X_{\tau_{n_1}} \text{ and }k_2=Y_{n_2}=X_{\tau_{n_2}}.
\end{equation*}
Then
\begin{equation*}
 \phi_1=\eta_{n_1} \text{ and } \phi_2=\eta_{n_2},
\end{equation*}
and the transtion probabilities of $(\tilde{Z}_n)$ are given by
\begin{equation*}
 \tilde{q}\big((\phi_1,y_1),(\phi_2,y_2)\big)=p\big((\eta_{n_1},X_{\tau_{n_1}}),(\eta_{n_2},X_{\tau_{n_2}})\big).
\end{equation*}

\begin{remark}With this construction in hand, the identification problem 
for the Poisson boundary of $\Zn$ on $\Z_2\wr\mathcal{T}$ reduces to the identification 
problem for the Poisson boundary of the new walk $(\tilde{Z}_n)$ on 
$\tilde{Z}_+=\tilde{\mathcal{C}}\times\Z_+$. 
\end{remark}

Recall that $(Y_n)$ is a random walk with drift to the right on $\Z_+$, and from 
Theorem \ref{thm:inf_cutpoints}, it has infinitely many cutpoints,
with probability $1$. 

\paragraph{Another Description of the Poisson Boundary.}
For the random walk $\Xn$ on the state space $\mathsf{G}$, let
like before $\Omega= \mathsf{G}^{\Z_+}$ be its path space. 
The shift $T:\Omega\to \Omega$ in the path space is defined as 
\begin{equation*}
 T(x_0,x_1,x_2,\ldots)=(x_1,x_2,\ldots).
\end{equation*}
Recall that the Poisson boundary is the space of ergodic components of the time shift
in the path space $\mathsf{G}^{\Z_+}$, endowed with the probability measure $\mathbb{P}$. 
In a more detailed way, denote by $\sim$ the 
\textit{orbit equivalence relation} of the shift $T$ in the path space:
\begin{equation*}
 (\mathbf{x})\sim(\mathbf{x'}) \quad \Leftrightarrow \quad \exists\ n,n'\geq 0: 
T_n\mathbf{x}=T_{n'}\mathbf{x'},\text{ for }\mathbf{x},\mathbf{x'}\in \Omega,
\end{equation*}
and by $\mathcal{A}_T$ the $\sigma$-algebra of all measurable unions of
$\sim$-classes (mod $0$) in the space $(\mathsf{G}^{\Z_+},\mathbb{P})$. 

Then the Poisson boundary of $\Xn$ is the quotient of the path 
space $(\mathsf{G}^{\Z_+},\mathbb{P})$ with respect to the equivalence relation $\sim$.

\paragraph{Tail $\sigma$-Algebra.}
The \textit{tail boundary} is analogous to the definition
of the Poisson boundary, with the equivalence relation
$\sim$ being now replaced with the \textit{tail equivalence relation} 
$\approx$, which is defined as
\begin{equation*}
 (\mathbf{x})\approx(\mathbf{x'}) \quad \Leftrightarrow \quad \exists\ N \geq 0: 
T_n\mathbf{x}=T_{n}\mathbf{x'},\text{ for }\mathbf{x},\mathbf{x'}\in \Omega,\forall n\geq N.
\end{equation*}
An important difference is that unlike the $\sigma$-algebra $\mathcal{A}_T$
from the definition of the Poisson boundary, the \textit{tail $\sigma$-algebra}
$\mathcal{A}^{\infty}$ of all measurable unions of $\approx$-classes
can be presented as the limit of the decreasing sequence of $\sigma$-algebras
$\mathcal{A}_n^{\infty}$ determined by the positions of sample paths of the random walk $(X_n)$
at times $\geq n$. In other words,
\begin{equation*}
 \mathcal{A}^{\infty}=\bigcap_n \mathcal{A}_n^{\infty},
\end{equation*}
where $\mathcal{A}_n^{\infty}$ are the coordinate $\sigma$-algebras
generated by the random variables $(X_k)_{k\geq n}$.

The Poisson and the tail boundaries are sometimes confused, and indeed
they do coincide for ``most common'' random walks. Their coincidence
for random walks on groups is a key ingredient of the entropy theory
of random walks. See {\sc Kaimanovich} \cite{Kaimanovich1992}
for more details and for criteria of triviality of these boundaries
and of their coincidence provided by $0-2$ laws.

Consider the extended random walk $\tilde{Z}_n=(\Phi_n,Y_n)$
defined above on the state space $\tilde{Z}_+$. 
Then the \textit{tail $\sigma$-algebra} of $(\tilde{Z}_n)$
can be described as the $\sigma$-algebra generated by the 
\textit{tail equivalence relation} $\approx$:
\begin{equation}
 \big((\phi_n,y_n)\big)\approx \big((\phi_n',y_n')\big)\Leftrightarrow \exists N\geq 0:
y_{n}  =y_{n}',\ \phi_{n} =\phi_{n}',\forall n\geq N,
\end{equation}
where $(\phi_n,y_n)$ and $(\phi'_n,y'_n)$ are two trajectories 
of the random walk $(\tilde{Z}_n)$. The \textit{tail boundary} of 
$(\tilde{Z}_n)$ is the quotient of the path space 
with respect to the tail equivalence relation $\approx$.

Recall that $(Y_n)$ is a transient random walk
on $\Z_+$, with transition probabilities given in \eqref{eq:tr_pb_z},
and $(\tilde{Z}_n)$ is the extended ``lamplighter type'' random
walk on $\tilde{\Z_+}$, which is also transient.

If  $\Phi_n(k)$ is the random configuration of lamps at time
$n$ on the tree $\mathcal{T}_k$ rooted at $k\in \Z_+$, let 
\begin{equation*}
\Phi_{\infty}(k)=\lim_{n\to\infty}\Phi_n(k), \text{ for every } k\in\Z_+. 
\end{equation*} 
The existence of the limit follows from the fact
that a transient process visits every state infinitely often.

\begin{lemma}\label{lem:sigma_alg}
For $n\geq 1$ consider the event $C_n=\{n\text{ is a cut time}\}$.
For $n\geq 1$ this event is in the $\sigma$-algebra $\sigma(\Phi_{\infty},Y_0)$,
and for trajectories in $C_n$ the value of $\tilde{Z}_n=(\Phi_n,Y_n)$
can be approximated in terms of the limit configuration $\Phi_{\infty}$ only. 
\end{lemma}
\begin{proof}
Let $(\phi_n,y_n)$ be a trajectory 
of the walk $\tilde{Z}_n=(\Phi_n,Y_n)$ and $\phi_{\infty}$
the limit configuration of lamps for this trajectory.
If $t\in\Z_+$ is a cut point of the
trajectory $(y_n)$ attained at time $n$ (so that $y_n=t$), then
$\phi_n$ is the restriction $\phi_{\infty}\big([0,t)\big)$ of the limit
configuration $\phi_{\infty}$ to the random subset $\{0,1,2,\ldots ,t-1\}\subset \Z_+$,
because the random cut point $t$ will be visited only once by  
$(y_n)$. That is, on the trees $\mathcal{T}_0,\mathcal{T}_1,\ldots ,\mathcal{T}_{t-1}$ 
the configuration will not change anymore after time $n$ and on the tree 
$\mathcal{T}_t$ the configuration will not be touched at all during 
the entire process. Therefore $\phi_n=\phi_{\infty}\big([0,t)\big)$ and
\begin{equation}\label{eq:limit_cfg_decomposition}
\phi_{\infty}=\phi_{\infty}\big([0,t)\big)\oplus \phi_0(t)\oplus \phi'\big([t+1,\infty)\big),
\end{equation}
where
\begin{equation*}
\phi_{\infty}([0,t))=\phi_{\infty}(0)\oplus\phi_{\infty}(1)\oplus \cdots 
\oplus \phi_{\infty}(t-1), \text{ for } t\in\Z_+,
\end{equation*}
$\phi_0(t)$ is the initial configuration (all lamps are off) 
on the tree $\mathcal{T}_t$ rooted at $t$, and $\phi'([t+1,\infty))$
is the configuration of lamps on the set $\{t+1,t+2,\ldots\}$,
which can still change at times $\{n+1,n+2,\ldots\}$. 

This proves that, under the condition that $t$ is a cut point for the trajectory $(y_n)$,
the lamp configuration is finalized on the random interval $\{0,1,2,\ldots ,t-1\}$.
Moreover, its value is determined by the limit configuration $\phi_{\infty}$,
and a decomposition of the form \eqref{eq:limit_cfg_decomposition}
is given.

Let $(\phi'_n,y'_n)$ be another trajectory of the ``lamplighter type''
random walk $(\tilde{Z}_n)$ such that the limit configurations 
$\phi_{\infty}$ and $\phi'_{\infty}$ coincide,
and such that $(y_n)$ and $(y'_n)$ have infinitely many 
common cut points. The existence of such paths follows from Remark \ref{rem:inf_cut_points}.
If $t\in\Z_+$ is such a common cut point
(which occurs at different cut time $n'$ with $y'_{n'}=t$), 
then at time $ n'$ the lamp configuration
$\phi'_{n'}$ is finalized and is again given by the restriction of
the limit configuration $\phi_{\infty}$ on the random interval $\{0,1,2,\ldots ,t-1\}$.
Moreover $\phi'_{n'}=\phi_n=\phi_{\infty}([0,t))$ and 
for $\phi_{\infty}$ a decomposition of the form \eqref{eq:limit_cfg_decomposition}
is given. This holds for all common cut points. 

Thus $C_n$ may be partitioned into subevents 
$C_n\cap \{\Phi_n=\Phi_{\infty}\big([0,t)\big), Y_n=t\}$
(with $t$ a cut point of $(Y_n)$) which are in the $\sigma$-algebra 
generated by $\Phi_{\infty}$ and $Y_0$, and the lemma is proved.
\end{proof}

\begin{proposition}\label{prop:exchangeable}
Let $(\tilde{Z}_n)$, with $Z_n=(\Phi_n,Y_n)$, be the extended ``lamplighter type''
random walk on $\tilde{Z}_+$. Then the tail $\sigma$-algebra of $(\tilde{Z}_n)$ is 
generated by the limit configurations of lamps $\Phi_{\infty}$.
\end{proposition}
\begin{proof}
We first describe the tail $\sigma$-algebra of the process
$(\Phi_n,Y_n)$ conditioned on the limit configuration $\Phi_{\infty}$
and on $(\Phi_0,Y_0)=(\mathbf{0},o)$.
By Theorem \ref{thm:inf_cutpoints}, the process $(Y_n)$
has infinitely many cut points, and implicitely also an infinite sequence
$n_1<n_2<\ldots$ of cut times . From Lemma \ref{lem:sigma_alg}, 
at these cut times the values of $Y_{n_1},Y_{n_2},\ldots$ are determined,
and the limit configuration is finalized, that is, it will not
change anymore from now on. Let $t_k$, $t_{k+1}$ be
the respective cut points for two consecutive cut times
$n_k, n_{k+1}$, i.e., $Y_{n_k}=t_k$. 

Between times
$\{n_k,n_k+1,n_k+2,\ldots,n_{k+1}-1\}$ the configuration of lamps 
can only be modified on the set $\{t_k,t_k+1,\ldots,t_{k+1}-1\}$,
that is, only on the trees $\mathcal{T}_{t_k},\mathcal{T}_{t_{k}+1},\ldots,\mathcal{T}_{t_{k+1}-1}$.
For different $k$, these sequences of modified configurations
are mutually independent (conditionally on $\Phi_{\infty}$), 
and hence the tail $\sigma$-algebra
of the process $(\Phi_n,Y_n)_n$, conditioned on $\Phi_{\infty}$ is trivial
by Kolmogorov $0-1$ law. Conditional triviality of a $\sigma$-algebra given 
$\Phi_{\infty}$ means that the $\sigma$-algebra is generated by $\Phi_{\infty}$ 
up to completion. That is, the tail $\sigma$-algebra
of $(\Phi_n,Y_n)_n$ is generated by the limit configurations of 
lamps $\Phi_{\infty}$. 
\end{proof}

\begin{remark}
Note that if for two sample trajectories $(\phi_n,y_n)$ and $(\phi'_n,y'_n)$,
the respective limit configurarions coincide, and the paths
$(y_n)$ and $(y_n')$ have infinitely many common cut points, then they correspond to the 
same point of the Poisson boundary of $(\tilde{Z}_n)$.
\end{remark}

\begin{proposition}\label{prop:poisson_tail}
For the extended ``lamplighter type'' random walk $(\tilde{Z}_n)$
on the state space $\tilde{Z}_+$, with $\tilde{Z}_n=(\Phi_n, Y_n)$,
its tail boundary coincides with its Poisson boundary,
up to sets of measure zero, for any starting point $(\Phi_0,Y_0)$.
\end{proposition}
\begin{proof}

Remark \ref{rem:inf_cut_points} (the density of the cut times) implies that the Poisson boundary of 
$(\tilde{Z}_n)$ covers the
space of limit configurations with finite fibers, that is, the
tail boundary (the tail $\sigma$-algebra), which is, in view of Proposition 
\ref{prop:exchangeable}, generated by the space of limit 
configurations $\Phi_{\infty}$. This is possible,
using the theory of covering Markov operators in {\sc Kaimanovich}
\cite{Kaimanovich1995} (see Theorem $4.3.3$ and Example in $4.4.4$), 
only when the cover is trivial. 
Thus, the Poisson boundary coincides with the tail boundary 
of $(\tilde{Z}_n)$.
\end{proof}

\begin{remark} The arguments in the previous two proofs are similar
to the description of the exchangeable sigma-algebra of bounded range random walks 
on transient groups due to {\sc James and Peres}~\cite{JamesPeres1997}. In both 
cases the argument is based on the existence of cut points at which a certain part
of the limit configuration is finalized. 

In  James and Peres \cite{JamesPeres1997}, coincidence of the limit configurations 
for two trajectories implies that they have the same sequence of cut points and cut times. 
This is not the case in our setup, because of which we have to use a 
special argument dealing with different sequences of cut times
associated to the same limit configuration. 
\end{remark}

\begin{proof}[Proof of Theorem \ref{thm:poisson_fixed_end_zero_drift}]
From the construction of the ``lamplighter-type'' random walk
$(\tilde{Z}_n)$, it follows that the Poisson boundary of $\Zn$
is reduced to the identification of the Poisson boundary for $(\tilde{Z}_n)$.

From Proposition \ref{prop:poisson_tail} 
it follows that the Poisson boundary of $(\tilde{Z}_n)$ coincides
with its tail $\sigma$-algebra, which is by Proposition 
\ref{prop:exchangeable} generated by the space of limit configurations 
$\Phi_{\infty}$. Consequently also the Poisson
boundary of $\Zn$ is isomorphic with the space of limit configurations
with the respective hitting distribution $\nu_{\infty}$. 
\end{proof}

The proof of the Theorem \ref{thm:poisson_fixed_end_zero_drift} is 
done only when the base random walk $\Xn$ is of nearest neighbour type and its associated
LRW $\Zn$ satisfies the local condition \eqref{eq:local_cond}. 
A next step is to generalize this results
to bounded range random walks $\Zn$ on $\Z_2\wr\mathcal{T}$. 

\chapter{Hyperbolic Graphs}
\label{chap:hyperbolic_graphs}

In this chapter we study once more the behaviour at infinity
of lamplighter random walks $\Zn$ on $\lgr$, when the
transitive base graph $\mathsf{G}$ is a hyperbolic graph in the sense of Gromov.
The method developed in Section \ref{sec:half_space_method}
is again applied here. Hyperbolicity of $\mathsf{G}$ and its geometric 
properties will be of important use.

We first lay out the basic definitions and properties 
of the hyperbolic graphs and groups, and then briefly reconsider
random walks $\Xn$ on such structures. Using the results
known for random walks $\Xn$ on hyperbolic graphs $\mathsf{G}$,
we prove similar results for lamplighter random walks $\Zn$, with $Z_n=(\eta_n,X_n)$
on $\lgr$.

\section{Preliminaries}
\label{sec:preliminaries_hyperbolic}

We recall here the most important definitions of hyperbolic 
graphs and their hyperbolic boundary and compactification. 
There is a vast literature on hyperbolic spaces, in particular 
on \textit{hyperbolic groups} (i.e., groups which have a 
hyperbolic Cayley graph). For details, the reader is invited 
to consult the texts by {\sc Gromov}~\cite{Gromov1987}, 
{\sc Ghys and de la Harpe}~\cite{GhysHarpe1990}, 
{\sc Coornaert, Delzant and Papadopoulus}~\cite{CoornaertDelzantPapadopoulos1990}, or, 
for a presentation in the context of random walks on graphs, {\sc Woess}~\cite[Section 22]{WoessBook}.

Let $(\mathsf{G},d)$ be a \index{metric space!proper}\textit{proper metric space}, that is, a space in which 
every closed ball $B(x,r)=\{y\in \mathsf{G}:d(x,y)\leq r\}$ is compact. 
For $x,y\in \mathsf{G}$, a \emph{geodesic arc} $\pi(x,y)$ in $\mathsf{G}$ is the image of
an isometric embedding of the real interval  $[0,d(x,y)]$ 
into $\mathsf{G}$ which sends $0$ to $x$ and $d(x,y)$ to $y$. The geodesic arc may
not be unique. 
 
Suppose that $\mathsf{G}$ is also \index{metric space!geodesic}\textit{geodesic}: for every pair of points 
$x,y \in \mathsf{G}$, there is a geodesic arc $\pi(x,y)$ in $\mathsf{G}$.
A \index{metric space!geodesic!geodesic triangle}\emph{geodesic triangle} 
consists of three points $u,v,w$ together with the geodesic arcs
$\pi(u,v),\pi(v,w),\pi(w,u)$, which are called the \emph{sides}
of the triangle.

\begin{minipage}[b]{0.6\linewidth}
\begin{definition}
A geodesic triangle is called 
\index{metric space!geodesic!thin geodesic triangle}\emph{$\delta$-thin}, with 
$ \delta \geq 0$, if every point on any one of the sides 
is at distance at most $\delta$ from some point on one of the other two sides.
\end{definition}
\begin{definition}
One says that the space $\mathsf{G}$ is \index{graph!hyperbolic}\emph{hyperbolic}, 
if there is $\delta \geq 0$ such that every geodesic triangle in $\mathsf{G}$ is $\delta$-thin.
\end{definition}
\end{minipage}
\hspace{0.5cm} 
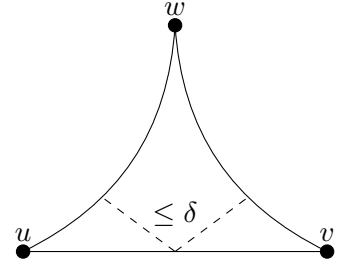
\begin{minipage}[b]{0.4\linewidth}
\begin{tikzpicture}
\draw (-2,0)--(2,0);
\filldraw[black] (-2,0) node [above] {$u$} circle (2.5pt);
\filldraw[black] (2,0) node [above] {$v$} circle (2.5pt);
\filldraw[black] (0,3) node [above] {$w$} circle (2.5pt);

\draw [bend right=30] (-2,0) to (0,3) ;
\draw [bend left=30] (2,0) to (0,3) ;

\draw [dashed] (0,0)  -- (-1,0.75);
\draw [dashed] (0,0) -- (1,0.75);
\node (a) at (0,0.5) {$\leq\delta$};
\end{tikzpicture}
\end{minipage}

The most typical examples of a hyperbolic spaces are trees 
(where $\delta=0$), the hyperbolic
upper half-plane $\mathbb{H}=\mathbb{H}_{2}$ (where $\delta=\log(1+\sqrt{2})$),
and the Poincar\'{e} unit disc.

\paragraph{Gromov Metric. }\label{gromov_distance}Let us 
choose a reference point $o$ in $\mathsf{G}$ and define, for $x,y \in \mathsf{G}$, the \textit{Gromov inner product}
\begin{equation*}
|x\wedge y|=\frac{1}{2}\big[|x|+|y|-d(x,y)\big],
\end{equation*}
where $|x|=d(o,x)$. If $\mathsf{G}$ is a tree (which is $0$-hyperbolic), 
then this is the usual graph distance between $o$ and the geodesic $\pi(x,y)$. 
A good way to think of this is as follows. Using the thin triangles 
property of hyperbolic metric spaces, we see that two walkers moving 
from $o$ to $x$ and $y$ respectively along suitable geodesics will 
remain close together (less than $2\delta$ apart) for a certain distance, 
before beginning to diverge rapidly. The Gromov inner product measures 
approximately the length of time that the two walkers remain close together.

\paragraph{Hyperbolic Boundary and Compactification.}\label{hyperbolic boundary}
Assume that $\mathsf{G}$ is a locally finite, transitive $\delta$-hyperbolic graph, 
with the natural metric (discrete graph metric) $d$ on it. In order to describe the 
\textit{hyperbolic boundary} $\partial_{h}\mathsf{G}$ and the 
\textit{hyperbolic compactification} 
$\widehat{\mathsf{G}}$, we define a new metric on $\mathsf{G}$ which can be extended on the ``boundary''.

Choose a constant $a>0$, such that $(e^{3\delta a}-1)<\sqrt{2}-1$,
and define for $x,y \in \mathsf{G}$ and the fixed vertex $o$
\begin{equation*}
\varrho _{a}(x,y)=
\begin{cases}
0, & x=y  \\
\exp(-a | x\wedge y |), & x \neq y .
\end{cases}
\end{equation*}
This is not a metric unless $\mathsf{G}$ is a tree. We now define
\begin{equation*}\label{hyperbolic_metric}
\theta_{a}(x,y)=\inf \left\{ \ \sum_{i=1}^{n}\varrho _{a}(x_{i-1},x_{i}):
\ n\geq 1,\ x=x_{0},x_{1},\ldots ,x_{n}=y \in \mathsf{G} \right\}.
\end{equation*}
Then $\theta_{a}$ is a metric on $\mathsf{G}$, and the graph $\mathsf{G}$ is discrete 
in this metric. We define $\widehat{\mathsf{G}}$ as the completion of $\mathsf{G}$ 
in the metric $\theta_{a}$. For the hyperbolic graph $\mathsf{G}$, 
the space $\widehat{\mathsf{G}}$ is compact. The space $\widehat{\mathsf{G}}$ 
is called the 
\index{graph!hyperbolic!compactification}\textit{hyperbolic compactification} 
of the graph $\mathsf{G}$. 
Each isometry from $\AUT(\mathsf{G})$ extends to a homeomorphism of $\widehat{\mathsf{G}}$. 
A sequence $(x_{n})$ with $|x_{n}| \to \infty $ is Cauchy, if and only if 
\begin{equation}\label{cauchy1}
\lim_{m,n \to \infty} | x_{m}\wedge x_{n}| =\infty,
\end{equation}
and another Cauchy sequence $(y_{n})$ will define the same boundary point, if and only if 
\begin{equation}\label{cauchy2}
\lim_{n \to \infty } |x_{n} \wedge y_{n}| =\infty.
\end{equation} 
Thus, one can also construct the 
\index{graph!hyperbolic!boundary}\textit{hyperbolic boundary} $\partial_{h}\mathsf{G}$ 
by factoring the set of all sequences in $\mathsf{G}^{\Z_+}$, which satisfy 
\eqref{cauchy1}, with respect to the equivalence relation given by \eqref{cauchy2}. 
The topology of $\widehat{\mathsf{G}}$ does not depend on the choice of $a$, and it is also 
independent of the choice of the base point $o$. 

Similary to trees, a third, equivalent way is to describe the 
hyperbolic boundary $\partial_{h}\mathsf{G}$ via 
equivalence of geodesic rays.
Two rays $\pi=[x_{0},x_{1},\ldots]$ and 
$\pi'=[y_{0},y_{1},\ldots]$ are equivalent if 
\begin{equation*}
\liminf_{n\to\infty }d(y_{n},x_n)< \infty.
\end{equation*}

\paragraph*{Hyperbolic and End Compactification.}
For a hyperbolic graph $\mathsf{G}$, it is easy to understand how its hyperbolic 
boundary $\partial_{h}{\mathsf{G}}$ is related to the space of ends 
$\partial \mathsf{G}$.  The hyperbolic boundary is finer, that is, the identity on $\mathsf{G}$ extends to 
a continuos surjection from the hyperbolic to the end compactification 
which maps $\partial_{h}\mathsf{G}$ onto $\partial \mathsf{G}$. For trees, the two
compactifications are the same. 

The hyperbolic compactification of a hyperbolic graph is a contractive
compactification with respect to $\AUT(\mathsf{G})$. For a proof, see
{\sc Woess} \cite[Theorem 22.14]{WoessBook}. Moreover, we can also prove even the
weaker property of the hyperbolic boundary, namely the weak projectivity.

Recall the Definition \ref{def:weak_proj} of the weak projectivity.
 
\begin{lemma}\label{lem:projective_hyperbolic}
The hyperbolic boundary $\partial_{h}\mathsf{G}$ of a hyperbolic graph $\mathsf{G}$ is a 
weakly projective space.
\end{lemma}
\begin{proof}
Let $(x_n),(y_n)$ be sequences of vertices in $\mathsf{G}$ such that $(x_n)$ converges 
to a hyperbolic boundary point $\mathfrak{u}\in\partial_h \mathsf{G}$ and 
\begin{equation}\label{previous_eq}
\frac{d(x_{n},y_{n})}{|x_n|}\to 0 ,\text{ as }n\to\infty.
\end{equation} 
In order to prove that $(y_n)$ converges to the same boundary point $\mathfrak{u}$, 
we show that equation \eqref{cauchy2} holds. Using the Gromov inner product
\begin{equation*}
|x_n\wedge y_n|=\frac{1}{2}\big[|x_n|+|y_n|-d(x_n,y_n)\big],
\end{equation*}
we obtain
\begin{equation*}
|x_n\wedge y_n|=\frac{1}{2}|x_n|\Bigg[1+\frac{|y_n|}{|x_n|}-\frac{d(x_n,y_n)}{|x_n|}\Bigg].
\end{equation*}
Now equation \eqref{previous_eq} and $|x_n|\to\infty$ implies that $|x_n\wedge y_n|\to \infty$ as $n\to \infty$. 
Therefore the sequence $(y_n)$ converges to the same boundary point $\mathfrak{u}\in\partial_h \mathsf{G}$.
\end{proof}
We shall also use the fact that for every two distinct hyperbolic boundary points
$\mathfrak{u},\mathfrak{v}\in\partial_h \mathsf{G}$, there is an infinite geodesic
$\pi(\mathfrak{u},\mathfrak{v})$ between them, which may not be unique.
For the proof see again {\sc Woess}~\cite{WoessBook}.

The boundary $\partial_h \mathsf{G}$ of an infinite transitive hyperbolic graph $\mathsf{G}$ 
is either infinite or has cardinality $2$. In the latter case, 
it is a graph with two ends that is quasi-isometric with the 
two-way infinite path $\mathbb{Z}$, 
and the Poisson boundary of any homogeneous random walk with finite 
first moment is trivial. In the sequel, we consider hyperbolic graphs
$\mathsf{G}$, which have infinite boundary.

\section{LRW over Hyperbolic Graphs}

Let $\mathsf{G}$ be a transitive hyperbolic graph with infinite hyperbolic boundary
$\partial_{h}\mathsf{G}$ and $\Gamma\subset \AUT(\mathsf{G})$.
Like in the previous sections, $\Zn$ is a homogeneous lamplighter random walk with 
transition matrix $P$ on $\lgr$, such that the projection
of $\Zn$ on $\mathsf{G}$ is the random walk $\Xn$ with transition matrix $P_{\mathsf{G}}$ on $\mathsf{G}$.

Recall that the boundary $\partial (\lgr)$ of $\lgr$ is given by
\begin{equation*}
\partial (\lgr)=(\widehat{\mathcal{C}}\times \widehat{\mathsf{G}})\setminus (\mathcal{C} \times \mathsf{G}) 
\end{equation*}
and the dense subset $\Pi$ of it defined in \eqref{eq:pi_boundary}
is now replaced by
\begin{equation}\label{eq:pi_boundary_hyperbolic}
\Pi_h =\bigcup_{\mathfrak{u} \in \partial_h \mathsf{G}}\mathcal{C}_{\mathfrak{u}}\times \{\mathfrak{u}\}, 
\end{equation}
where the set $\mathcal{C}_{\mathfrak{u}}$ consists of all
configurations $\zeta$ which are either finitely supported,
or infinitely supported with $\supp(\zeta)$  accumulating only
at $\mathfrak{u}\in \partial_h \mathsf{G}$.

Taking into account the connection between the hyperbolic boundary $\partial_h \mathsf{G}$ and
the space of ends $\partial \mathsf{G}$ of a transitive graph $\mathsf{G}$, we have to distinguish
two different cases:
\begin{enumerate}[(a)]
 \item \textit{Infinite hyperbolic boundary and infinitely many ends.}
 \item \textit{Infinite hyperbolic boundary and only one end.}
\end{enumerate}
 
\subsection{Infinite Hyperbolic Boundary and Infinitely Many Ends} 

From the fact that the identity on $\mathsf{G}$ extends to a continuous 
surjection from the hyperbolic to the end compactification, 
which maps $\partial_{h}\mathsf{G}$ onto $\partial \mathsf{G}$, 
it follows that we are in the case of Chapter \ref{inf_many_ends}, 
i.e., of a graph with infinitely many ends. The ends are the 
connected components of the hyperbolic boundary.

The convergence of $\Zn$ is given by Theorem 
\ref{thm:conv_lrw_no_fixedend} and Theorem \ref{thm:conv_LRW_fixed_end} 
respectively (with the hyperbolic boundary $\partial_{h}\mathsf{G}$ 
instead of the space of ends $\partial \mathsf{G}$), depending on whether an 
element of $\partial_h \mathsf{G}$ is fixed under the action of $\Gamma$ or not.
 
The Poisson boundary of the lamplighter random walk $\Zn$ is given by:
\begin{itemize}
\item Theorem \ref{thm:poisson_lrw_gr_infends_nofixed_end}, with
$\Pi_h$ instead of $\Pi$, in the nondegenerate case
when no hyperbolic element of $\partial_h \mathsf{G}$ is fixed by $\Gamma$.
\item In the ``degenerate case'', when $\Gamma$ fixes one point in 
$\partial_h \mathsf{G}$ (which has to be unique, given that $\partial_h \mathsf{G}$
is assumed to be infinite), then $\mathsf{G}$ ``looks'' like in 
Figure \ref{fig:hom_tree_fixed_end}, since the ends are the
connected components of the hyperbolic boundary by {\sc Pavone}
\cite{Pavone1989}. The convergence of $LRW$
is given by Theorem \ref{thm:conv_LRW_fixed_end} 
and the Poisson boundary is described in Theorem 
\ref{thm:poisson_lrw_fixed_end} for the non-zero modular drift case,
and for zero modular drift in Theorem \ref{thm:poisson_fixed_end_zero_drift}.
\end{itemize}

\subsection{Infinite Hyperbolic Boundary and One End}

The problem of the existence of an one-ended hyperbolic graph $\mathsf{G}$
with a transitive group $\Gamma$ that fixes a boundary point
in $\partial_h \mathsf{G}$ is still unsolved. Experts belive that the
answer is negative. See the remarks at the end of the paper 
{\sc Kaimanovich and Woess}~\cite{KaimanovichWoess2002}. Since this
situation is only hypothetical, we shall consider in this section
the case when the graph $\mathsf{G}$ has infinite hyperbolic boundary
$\partial_h \mathsf{G}$ and only one end, and no element of $\partial_h \mathsf{G}$ is
fixed under the action of $\Gamma$.

For homogeneous random walks $\Xn$ on $\mathsf{G}$ with $|\partial_h \mathsf{G}|=\infty$ 
the following holds. See {\sc Woess}~\cite{Woess1993} for the proof.

\newpage

\begin{theorem}
\label{thm:conv_rw_hyperbolic}
If $\mathsf{G}$ is a hyperbolic graph and $\Gamma\subset \AUT(\mathsf{G})$ 
does not fix any element of $\partial_h \mathsf{G}$, then the random walk
$\Xn$ converges almost surely in the hyperbolic topology
to a random point $X_{\infty}\in\partial_h \mathsf{G}$. If $\mu_{\infty}$
is the limit distribution of $\Xn$, then
\begin{enumerate}[(a)]
      \item The support of $\mu_{\infty}$ is the whole $\partial_h \mathsf{G}$.
      \item The measure $\mu_{\infty}$ is continuous on $\partial \mathsf{G}$, that is 
		$\mu_{\infty}(\{\mathfrak{u}\})=0$, $\mathfrak{u}\in \partial_h \mathsf{G}$.
        \end{enumerate}
\end{theorem}
The convergence to the boundary of homogeneous random walks $\Xn$ on 
hyperbolic graphs $\mathsf{G}$ holds without any need of the first moment assumption.
In order to have similar convergence results for lamplighter random walks $\Zn$, 
we need this assumption on the base random walk $\Xn$.
From now on, we suppose that $\Xn$ has finite first moment on 
the transitive base graph $\mathsf{G}$.
Recall also that the hyperbolic boundary $\partial_h \mathsf{G}$ is 
weakly projective, by Lemma \ref{lem:projective_hyperbolic}.

Convergence of lamplighter random walks $\Zn$ on $\lgr$, with
$\mathsf{G}$ a hyperbolic graph, follows easily from Theorem 
\ref{thm_conv_lrw_general_graphs}, since Assumption 
\ref{assumptions_brw} is satisfied. Recall the definition 
\eqref{eq:pi_boundary_hyperbolic} of the boundary $\Pi_h$ of the
graph $\lgr$.

\begin{theorem}
\label{thm:conv_lrw_hyperbolic}
Let $\Zn$ be an irreducible, homogeneous random walk with finite 
first moment on $\lgr$, where $\mathsf{G}$ is a hyperbolic graph and $\Gamma\subset \AUT(\mathsf{G})$ 
does not fix any element of $\partial_h \mathsf{G}$. Then there exists a $\Pi_h$-valued 
random variable $Z_{\infty}$ such that $Z_n\to Z_{\infty}$ almost surely, in the topology of $\widehat{\lgr}$, 
for  every starting point. The distribution of $Z_{\infty}$ is a continuous measure on $\Pi_h$. 
\end{theorem}
\begin{proof}
The proof of this result follows basically the proof of 
Theorem \ref{thm_conv_lrw_general_graphs} and Theorem 
\ref{thm:conv_lrw_no_fixedend},
with the hyperbolic boundary $\partial_h \mathsf{G}$ instead 
of the space of ends $\partial \mathsf{G}$.
\end{proof}

\subsubsection{Poisson Boundary}

Let us first recall the description of the Poisson
boundary of random walks $\Xn$ over transitive hyperbolic graphs $\mathsf{G}$, which will 
be used for the Poisson boundary of lamplighter random walks $\Zn$. For
sake of completeness, we also give here the idea of the proof. For 
a detailed proof, see also {\sc Kaimanovich} \cite{Kaimanovich2000}. 

\begin{theorem}
\label{thm:poisson_rw_hyperbolic}
Let $\mathsf{G}$ be a hyperbolic graph with $|\partial_h \mathsf{G}|=\infty$, 
and $\Gamma\subset \AUT(\mathsf{G})$ does not fix any element of $\partial_h \mathsf{G}$.
If $\Xn$ is an homogeneous random walk with finite first moment on $\mathsf{G}$, then
its Poisson boundary  is $(\partial_h \mathsf{G},\mu_{\infty})$.
\end{theorem}
\begin{proof}
The proof is very similar with the proof of Theorem 
\ref{thm:poisson_rw_inf_ends}. By Theorem \ref{thm:conv_rw_hyperbolic} 
and Definition \ref{def:mu_bndr}, 
the space $(\partial_h \mathsf{G},\mu_{\infty})$ is a $\mu$-boundary for the random walk
$\Xn$. Theorem \ref{thm:conv_rw_hyperbolic} applies also to the reversed 
random walk $(\check{X}_{n})$, with the limit distribution 
$\check{\mu}_{\infty}$, and $(\partial \mathsf{G},\check{\mu}_{\infty})$ is a 
$\mu$-boundary for $(\check{X}_{n})$. Recall that $\mu$ is the probability measure
on $\Gamma$ which is uniquely induced by the transition matrix $P_{\mathsf{G}}$
of $\Xn$ as in \eqref{eq:correspondence_rw}.

Apply the Strip Criterion \ref{thm:strip_crit}
and define the strip $\mathfrak{s}(\mathfrak{u},\mathfrak{v})$, for
$\mathfrak{u},\mathfrak{v}\in\partial_h \mathsf{G}$.
By continuity of $\mu_{\infty}$ and $\check{\mu}_{\infty}$, we have 
\begin{equation*}
\mu_{\infty}\times\ \check{\mu}_{\infty}\big(\{\mathfrak{u},\mathfrak{v}\in\partial_h \mathsf{G}: 
\mathfrak{u}=\mathfrak{v}\}\big) =0.
\end{equation*}
Therefore, we have to construct the strip $\mathfrak{s}(\mathfrak{u},\mathfrak{v})$
only in the case $\mathfrak{u} \neq \mathfrak{v}$. Let
\begin{equation}\label{eq:strip_rw_hyperbolic}
\mathfrak{s}(\mathfrak{u},\mathfrak{v})=\{x\in \mathsf{G}:x
\text{ lies on a two way infinite geodesic }  \text{ between }\mathfrak{u}\text{ and }\mathfrak{v}\}.
\end{equation} 
The strip $\mathfrak{s}(\mathfrak{u},\mathfrak{v})$ is set 
of all points $x$ from all geodesics in $\mathsf{G}$ joining 
$\mathfrak{u}$ and $\mathfrak{v}$. This is a subset of $\mathsf{G}$, and 
\begin{equation*}
\gamma \mathfrak{s}(\mathfrak{u},\mathfrak{v})=\mathfrak{s}(\gamma\mathfrak{u},\gamma\mathfrak{v}), 
\end{equation*}
for every $\gamma\in\Gamma$. In a $\delta$-hyperbolic 
graph any two geodesics with the same endpoints are 
within uniformly bounded distance at most $2\delta$ one from another 
(see \cite{GhysHarpe1990} for details), 
and the geodesics have linear growth. This implies that  
there exists a constant $c>0$, such that 
\begin{equation*}
|\mathfrak{s}(\mathfrak{u},\mathfrak{v})\cap B(o,n)|\leq cn,
\end{equation*}
for all $n$ and distinct $\mathfrak{u},\mathfrak{v}\in\partial_{h} \mathsf{G}$. 
This proves the subexponential growth of 
$\mathfrak{s}(\mathfrak{u},\mathfrak{v})$, which completes 
the proof.
\end{proof}

In order to describe the Poisson boundary of lamplighter 
random walks $\Zn$ over $\lgr$, when $\mathsf{G}$ is a $\delta$-hyperbolic graph
with infinite hyperbolic boundary and only one end, we need some
additional facts.

Consider the hyperbolic graph $\mathsf{G}$ and its hyperbolic boundary 
$\partial_{h}\mathsf{G}$ as being 
described by equivalence of geodesic rays. For $y\in \mathsf{G}$ and 
$\mathfrak{u}\in\partial_{h}\mathsf{G}$ let 
$\pi=[y=y_{0},y_{1},\ldots,\mathfrak{u}]$ be a 
geodesic ray joining $y$ with $\mathfrak{u}$. For every $x\in \mathsf{G}$, let 
\begin{equation*}
\beta_{\mathfrak{u}}(x,\pi)=\limsup_{i\to\infty}\big(d(x,y_{i})-i\big).
\end{equation*} 
Define the \textit{Busemann function}
\begin{equation*}
\beta_{\mathfrak{u}}:\mathsf{G}\times \mathsf{G}\rightarrow \mathbb{R}, \text{ for } \mathfrak{u}\in\partial_{h}\mathsf{G}, 
\end{equation*}
as follows:
\begin{equation*}
\beta_{\mathfrak{u}}(x,y)=\sup\{\beta_{\mathfrak{u}}(x,\pi'):
\pi'\text{ is a geodesic ray from }y\text{ to } \mathfrak{u}\}.
\end{equation*}
The \textit{horosphere} with the centre in $\mathfrak{u}$ and passing through 
$x\in \mathsf{G}$, denoted $H_{x}(\mathfrak{u})$ is the set
\begin{equation*}
H_{x}(\mathfrak{u})=\{ y\in \mathsf{G}:\ \beta_{\mathfrak{u}}(x,y)=0\}.
\end{equation*}
For every $x,y\in \mathsf{G}$ and $\mathfrak{u}\in\partial_{h}\mathsf{G}$ the 
distances $d(x,\mathfrak{u})$ and $d(y,\mathfrak{u})$ are not defined, 
but the Busemann function $\beta_{\mathfrak{u}}(x,y)$ gives sense to the expression 
$d(x,\mathfrak{u})-d(y,\mathfrak{u})$, 
which is of the type $\infty -\infty$. 
One can think of $\beta_{\mathfrak{u}}(x,y)$ as being the distance 
between the horospheres $H_{x}(\mathfrak{u})$ and $H_{y}(\mathfrak{u})$, 
where $\beta_{\mathfrak{u}}(x,y)>0$ if $x$ is at the 
exterior of the horoball limited by $H_{y}(\mathfrak{u})$.
For properties of the Busemann function, see 
\cite[Chapter 8]{GhysHarpe1990}.

\begin{theorem}
\label{thm:poisson_lrw_hyperbolic_graphs}
Let $\Zn$ be an irreducible, homogeneous random walk with 
finite first moment on $\lgr$, where $\mathsf{G}$ is a hyperbolic 
transitive graph with $|\partial_e\mathsf{G}|=1$ and $|\partial_h \mathsf{G}|=\infty$.
If $\Gamma\subset \AUT(\mathsf{G})$ acts transitively on $\mathsf{G}$ and does not 
fix any element in $\partial_h \mathsf{G}$, and $\nu_{\infty}$ is 
the limit distribution on $\Pi_h$,
with $\Pi_h$ defined in \eqref{eq:pi_boundary_hyperbolic},
then $(\Pi_h,\nu_{\infty})$ is the Poisson boundary of $\Zn$.
\end{theorem}
\begin{proof}
We apply Theorem \ref{PoissonTheorem}. First of all, we 
check that the conditions required in the 
\textit{Half Space Method} are satisfied for
$\Xn$ and $(\check{X}_n)$ on $\mathsf{G}$, which have finite first
moments. From Theorem \ref{thm:conv_rw_hyperbolic}
and Lemma \ref{lem:projective_hyperbolic} it follows that
the Assumption \ref{assumptions_brw} holds. 

For the second requirement in the \textit{Half Space Method}
consider the strip $s(\mathfrak{u},\mathfrak{v})\subset \mathsf{G}$
defined in \eqref{eq:strip_rw_hyperbolic}, which has
sub-exponential growth.

Finally, let us partition $\mathsf{G}$ into half-spaces. Actually, 
this is one of the examples where the partition is made into 
two half-spaces and another ``degenerate'' set on 
which the lamplighter configuration will be set up to be  $0$. 
For every $x\in\mathfrak{s}(\mathfrak{u},\mathfrak{v})$, 
let $H_{x}(\mathfrak{u})$ (respectively, $H_{x}(\mathfrak{v})$) 
be the horosphere with center $\mathfrak{u}$ (respectively, $\mathfrak{v}$) 
and passing through $x$. Remark that the two horospheres 
may have non empty intersection. Consider the 
partition of $\mathsf{G}$ into the subsets $\mathsf{G}_{+}$, $\mathsf{G}_{-}$, and $\mathsf{G}\setminus(\mathsf{G}_{+}\cup \mathsf{G}_{-})$, 
where 
\begin{equation*}
\mathsf{G}_{+}(x)=H_{x}(\mathfrak{u}) 
\end{equation*}
contains a neighbourhood of $\mathfrak{u}$, and
\begin{equation*}
\mathsf{G}_{-}(x)=H_{x}(\mathfrak{v})\setminus H_{x}(\mathfrak{u}), 
\end{equation*}
contains a neighbourhood of $\mathfrak{v}$. This partition is $\Gamma$-equivariant. 

Up to now, we have checked that the assumptions required in the
\emph{Half-Space Method }\ref{sec:half_space_method} are fulfilled,
when $\mathsf{G}$ is a hyperbolic graph. Apply now Theorem \ref{PoissonTheorem}.

By Theorem \ref{thm:conv_lrw_hyperbolic} each of the random walks 
$\Zn$ and $(\check{Z}_{n})$ 
starting at $(\mathbf{0},o)$ converges almost surely to a 
$\Pi_h$-valued random variable, with $\Pi_h$ given in 
\eqref{eq:pi_boundary_hyperbolic}, with limit distributions 
$\nu_{\infty}$ and $\check{\nu}_{\infty}$ respectively.
Then $(\Pi_h,\nu_{\infty})$ and 
$(\Pi_h,\check{\nu}_{\infty})$ are $\nu$- and $\check{\nu}$- boundaries 
of the respective random walks. Take
\begin{equation*}
b_{+}=(\phi_{+},\mathfrak{u}),\text{ and } b_{-}=(\phi_{-},\mathfrak{v})\in\Pi_h ,
\end{equation*} 
where $\phi_{+}$ and $\phi_{-}$ are the limit configurations of $(Z_{n})$ and 
$(\check{Z}_{n})$, respectively, and $\mathfrak{u},\mathfrak{v}\in\partial \mathsf{G}$ 
are their only respective accumulation points. 

Define the configuration 
$\Phi(b_{+},b_{-},x)$ like in \eqref{StripConfiguration}, that is,
\begin{equation*}
\Phi(b_{+},b_{-},x)=
\begin{cases}
\phi_{-}, & \text{on}\  H_{x}(\mathfrak{u})\\
\phi_{+}, & \text{on}\  H_{x}(\mathfrak{v})\setminus H_{x}(\mathfrak{u})\\
0,        & \text{on}\  \mathsf{G}\setminus \big(H_{x}(\mathfrak{u})\cup H_{x}(\mathfrak{v})\big)
\end{cases} 
\end{equation*}
and the strip $S(b_{+},b_{-})$ exactly like in \eqref{eq:lamplighter_strip}, i.e.,
\begin{equation*}
S(b_{+},b_{-})=\{\left(\Phi,x\right) :\  x\in\mathfrak{s}(\mathfrak{u},\mathfrak{v})\}.
\end{equation*}
For a graphic visualization of the above construction of the strip
and lamps configuration $\Phi$, see Figure \ref{fig:horosphere_strip}.

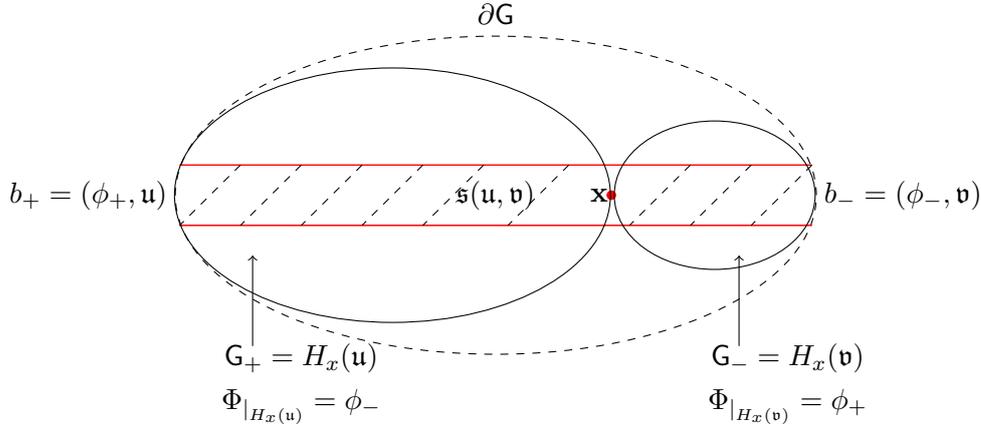
\begin{figure}[h]

\begin{tikzpicture}[scale=0.8]

\draw[dashed,thin] (0,0) ellipse (150pt and 75pt);
\node at (6.7,0) {$b_{-}=(\phi_{-},\mathfrak{v})$};
\node at (-6.7,0) {$b_{+}=(\phi_{+},\mathfrak{u})$};
\node at (0,3) {$\partial \mathsf{G}$};

\draw[thin,red] (-5.2,0.5) to (5.2,0.5);
\draw[thin,red] (-5.2,-0.5) to (5.2,-0.5);

\draw[thin, red] (-5.2,0.5)--(5.2,0.5);
\draw[thin, red] (-5.2,-0.5)--(5.2,-0.5);
\draw[thin,dashed] (-4.2,0.5)--(-5.2,-0.5);
\draw[thin,dashed] (-3.2,0.5)--(-4.2,-0.5);
\draw[thin,dashed] (-2.2,0.5)--(-3.2,-0.5); 
\draw[thin,dashed] (-1.2,0.5)--(-2.2,-0.5);
\draw[thin,dashed] (-0.2,0.5)--(-1.2,-0.5);
\draw[thin,dashed] (1.2,0.5)--(0.2,-0.5);

\draw[thin,dashed] (3.2,0.5)--(2.2,-0.5);
\draw[thin,dashed] (4.2,0.5)--(3.2,-0.5);
\draw[thin,dashed] (5.2,0.5)--(4.2,-0.5);

\node at (1.7,0) {$\mathbf{x}$};
\filldraw[red] (1.9,0) circle (2pt);
\node at (0,0) {$\mathbf{\mathfrak{s}(\mathfrak{u},\mathfrak{v})}$};
\draw (3.6,0) ellipse (47 pt and 35 pt);
\draw (-1.7,0) ellipse (102pt and 60 pt );

\draw[->] (4,-2.5) -> (4,-1);
\node at (4.8,-2.7) {$\mathsf{G}_{-}=H_x(\mathfrak{v})$};
\node at (4.8,-3.5) {$\Phi_{|_{H_x(\mathfrak{v})}}=\phi_{+}$};;
\draw[->] (-4,-2.5) -> (-4,-1);
\node at (-3.2,-2.7) {$\mathsf{G}_{+}=H_x(\mathfrak{u})$};
\node at (-3.2,-3.5) {$\Phi_{|_{H_x(\mathfrak{u})}}=\phi_{-}$};;

\end{tikzpicture}

\caption{The strip construction using horospheres}
\label{fig:horosphere_strip}
\end{figure}

From Theorem \ref{PoissonTheorem}, $S(b_{+},b_{-})$ satisfies the 
conditions from Theorem \ref{thm:strip_crit}, and it follows 
that the space $(\Pi_h,\nu_{\infty})$ is the Poisson boundary of 
the lamplighter random walk $\Zn$ over $\lgr$.
\end{proof}
 
\chapter{LRW over Euclidean Lattices}

The wreath product $\Z_2\wr\Z^d$ was first considered in
{\sc Kaimanovich and Vershik} \cite{KaimanovichVershik1983}
as a source of several examples and counterexamples
illustrating the relationship between growth, amenability
and the Poisson boundary for random walks on groups.

In this chapter we consider lamplighter random walks over Euclidean lattices
$\mathbb{Z}^d$. The associated lamplighter graph is $\mathbb{Z}_2\wr\mathbb{Z}^d$.
We show how to apply the \textit{Half Space Method} \ref{sec:half_space_method} 
in order to describe the Poisson boundary
of lamplighter random walks $\Zn$, in the case when the random walk 
$\Xn$ of $\Z^d$ has non-zero drift.

We emphasize that the results in this
section were earlier obtained by Kaimanovich {\sc Kaimanovich}
\cite{Kaimanovich2001} in the non-zero drift case. Nevertheless,
we still recall them, as another application of our methods.
For the zero drift case, the description of the Poisson boundary 
was recently done in {\sc Erschler}
\cite{Erschler2010}, by using a modified version of the Ray Criterion.

\section{Random Walks on $\Z_2\wr\Z^d$}

Let now $\mathsf{G}= \Z^{d}$, $d\geq 3$, be the $d$-dimensional lattice, 
with the Euclidean metric $|\cdot|$ on it. For $\Z^{d}$, there 
are also natural boundaries and compactifications. A nice example 
of compactification is obtained by embedding $\Z^{d}$ into the 
$d$-dimensional unit disc via the map $x\mapsto x/(1+|x|)$, 
and taking the closure. In this compactification, the boundary 
$\partial\Z^{d}$ is the unit sphere $S_{d-1}$ in $\mathbb{R}^d$, 
and a sequence $(x_{n})$ in $\Z^{d}$ converges to $\mathfrak{u}\in S_{d-1}$ 
if and only if 
\begin{equation*}
|x_{n}|\to\infty \text{ and } \frac{x_{n}}{|x_{n}|}\to\mathfrak{u} \text{ as } n\to\infty.
\end{equation*}
Recall first the Definition \ref{def:weak_proj} of a weakly projective boundary.

\newpage

\begin{lemma}
The boundary $S_{d-1}$ is a weakly projective boundary.  
\end{lemma}
\begin{proof}
Let $(x_{n})$ be a sequence which converges to $\mathfrak{u}\in S_{d-1}$, 
and $(y_{n})$ be another sequence in $\Z^d$, such that 
\begin{equation*}
\frac{x_{n}}{|x_{n}|}\to\mathfrak{u}, \text{ and } \frac{|x_{n}-y_{n}|}{|x_{n}|}\to 0 \text{ as }n\to\infty.
\end{equation*}
Since 
\begin{equation*}
\Big{|}\dfrac{x_{n}}{|x_{n}|}-\dfrac{y_{n}}{|x_{n}|}\Big{|} = \dfrac{|x_{n}-y_{n}|}{|x_{n}|}\to 0,
\text{ as }n\to\infty,
\end{equation*}
it follows that $\frac{y_{n}}{|x_{n}|}\to\mathfrak{u}$.  Now
\begin{equation*}
\frac{y_{n}}{|x_{n}|}=\frac{y_{n}}{|y_{n}|}\cdot\frac{|y_{n}|}{|x_{n}|} 
\end{equation*}

and the sequence $|y_{n}|/|x_{n}|$ of real numbers converges to $1$, 
since we can bound it from above and from below by two sequences both converging to 1. 
Therefore $y_{n}/|y_{n}|\to\mathfrak{u}$, and this proves the desired.
\end{proof}

Consider the random walk $\Xn$ with law $\mu$ on $\Z^d$
and the lamplighter random walk $\Zn$ with law $\nu$ on $\Z_2\wr\Z^d$.
If the law $\mu$ of $\Xn$ has non-zero first moment (drift)
\begin{equation*}
m=\sum_{x}x\mu(x)\in\mathbb{R}^{d}, 
\end{equation*}
then the law of large numbers implies that $\Xn$ converges to the 
boundary $S_{d-1}$ in this compactification with deterministic limit $m/|m|$. 
In particular, the limit distribution $\mu_{\infty}$ is the Dirac mass at this point. 

Let us now state the result on the Poisson boundary of lamplighter random 
walks $\Zn$ on $\Z_2\wr\Z^d$ in the case of non-zero drift.

\begin{theorem}\label{thm:poisson_lrw_lattices}
Let $\Zn$, with $Z_{n}=(\eta_{n},X_{n})$, be a random walk
with law $\nu$ on $\Z_{2}\wr \mathbb{Z}^{d}$, $d\geq 3$,
such that $\supp(\nu)$ generates 
$\Z_{2}\wr \mathbb{Z}^{d}$, and $\Xn$ has non-zero drift on 
$\mathbb{Z}^{d}$. If $\nu$ has finite first moment,
and $\Pi$ is defined as in \eqref{eq:pi_boundary}
with the unit sphere $S_{d-1}$ instead of $\partial \mathsf{G}$, then
$(\Pi,\nu_{\infty})$ is the Poisson boundary of $\Zn$, where 
$\nu_{\infty}$ is the limit distribution of $\Zn$ on $\Pi$.
\end{theorem}
Thus, the Poisson boundary of lamplighter random walks
is described by the space of infinite limit configurations of lamps.
\begin{proof}
Here it is easy to check the requirements in the \textit{Half Space Method}.

The random walk $\Xn$ (respectively $(\check{X}_n)$) converges to 
the boundary $S_{d-1}$ with deterministic limit 
$\mathfrak{u}=m/|m|$ (respectively, $\mathfrak{v}=-m/|m|$), 
in the case of non-zero drift $m$. The limit distributions 
$\nu_{\infty}$ and $\check{\nu}_{\infty}$ are the Dirac-masses at the
respective limit points.

Define the strip $\mathfrak{s}(\mathfrak{u},\mathfrak{v})=\Z^{d}$. 
It does not depend on the limit points, it is $\Z^{d}$-equivariant, 
and it has polynomial growth of order $d$, that is, also subexponential growth. 

Next, let us partition $\Z^{d}$ into half-spaces. 
Denote by $\pi(\mathfrak{u},\mathfrak{v})$ the geodesic 
of $S_{d-1}$ joining the two deterministic boundary points 
$\mathfrak{u},\mathfrak{v}\in S_{d-1}$. 
This is exactly the diameter of the ball, since the points $\mathfrak{u}$ and 
$\mathfrak{v}$ are antipodal points, i.e., they are opposite through the centre. 
For every $x\in\mathfrak{s}(\mathfrak{u},\mathfrak{v})=\Z^{d}$, 
consider the hyperplane which passes through $x$ and is orthogonal to 
$\pi(\mathfrak{u},\mathfrak{v})$. This hyperplane cuts $\Z^{d}$ into two disjoint spaces 
$\Z_{+}$ and $\Z_{-}$, containing $\mathfrak{u}$ and $\mathfrak{v}$, 
respectively. The half-spaces $\Z_{+}$ and $\Z_{-}$ are $\Z^{d}$-equivariant.
Apply now Theorem \ref{PoissonTheorem}.

By Theorem \ref{thm_conv_lrw_general_graphs} each of the random walks $(Z_{n})$ and 
$(\check{Z}_{n})$ converges almost surely to a $\Pi$-valued random variable, 
where $\Pi$ is defined as in \eqref{eq:pi_boundary}, with $S_{d-1}$ instead of $\partial \mathsf{G}$. 
Nevertheless, the only ``active'' points of non-zero $\mu_{\infty}$- and 
$\check{\mu}_{\infty}$-measure on $S_{d-1}$ are  $\mathfrak{u}=m/|m|$ and 
$\mathfrak{v}=-m/|m|$, respectively. More precisely, $\Pi$ can be written as
\begin{equation*}
\Pi =\Big{(}\mathcal{C}_{\mathfrak{u}}\times\{\mathfrak{u}\}\Big{)}\cup
\Big{(}\mathcal{C}_{\mathfrak{v}}\times\{\mathfrak{v}\}\Big{)}, 
\end{equation*}
where $\mathcal{C}_{\mathfrak{u}}$ (respectively, $\mathcal{C}_{\mathfrak{v}}$) 
is the set of all configurations accumulating only at $\mathfrak{u}$ (respectively, $\mathfrak{v}$).

If $\nu_{\infty}$  and $\check{\nu}_{\infty}$ are the limit distributions of $\Zn$ 
and $(\check{Z}_{n})$ on $\Pi$, then the spaces 
$(\Pi,\nu_{\infty})$ and $(\Pi,\check{\nu}_{\infty})$ are 
$\nu$- and $\check{\nu}$- boundaries of the respective random walks. Take 
\begin{equation*}
b_{+}=(\phi_{+},\mathfrak{u})\text{ and } b_{-}=(\phi_{-},\mathfrak{v})\in\Pi, 
\end{equation*}
where $\phi_{+}$ and $\phi_{-}$ are the limit configurations of $\Zn$ and 
$(\check{Z}_{n})$, respectively, and $\mathfrak{u},\mathfrak{v}$ 
are their only respective accumulation points. Define the configuration 
$\Phi(b_{+},b_{-},x)$ like in \eqref{StripConfiguration}, and the strip 
$S(b_{+},b_{-})$ exactly like in \eqref{eq:lamplighter_strip}.
From Theorem \ref{PoissonTheorem}, $S(b_{+},b_{-})$ satisfies the 
conditions from Theorem \ref{thm:strip_crit}, and it follows that 
the space $(\Pi,\nu_{\infty})$ is the Poisson boundary of the lamplighter 
random walk $\Zn$ over $\Z_{2}\wr \Z^d$.
\end{proof}

\chapter{Open Problems on LRW}

The goal of this chapter is to give a brief overview 
on some problems that are related to the first part of 
the thesis. This is only a small personal selection of 
the vast questionings concerning Lamplighter Random Walks.

\section{Poisson Boundary of LRW}

Let $\mathcal{T}$ be the oriented tree in Figure \ref{fig:tree_zero_drift}
with a fixed end $\omega$, and $\Xn$ be a random walk with zero modular 
drift $\delta(P)$ on $\mathcal{T}$. Consider the associated lamplighter random 
walk $Z_n=(\eta_n,X_n)$ on $\Z_2\wr\mathcal{T}$. In Theorem 
\ref{thm:poisson_fixed_end_zero_drift},
we have proved that the Poisson boundary of $\Zn$ is the space
of limit configurations of lamps together with the
respective hitting distribution, only 
when $\Xn$ and $\Zn$ are both of nearest neighbour type. This means 
that the configuration of the lamp can be changed only at the current position. This assumption
cannot be avoided in our proof.

It will be interesting to generalize this result when the base random walk 
$\Xn$ has bounded range (not range $1$ like in our case)
and the lamp configuration can be changed in a bounded neighbourhood of the
current position (not only at the current position like in our settings).

\begin{conjecture}
For any random walk $\Zn$ with bounded range on $\Z_2\wr\mathcal{T}$, such that
the projection $\Xn$ on $\mathcal{T}$ has zero drift, its 
Poisson boundary is isomorphic with the space of infinite limit configurations 
of lamps, endowed with the respective hitting distribution. 
\end{conjecture} 

\section{Return Probability Asymptotics of LRW}

Let $\mathsf{G}$ be an infinite graph, $\Z_2$ the finite set of lamp states,
and $\lgr$ the associated lamplighter graph.
Consider the lamplighter random walk $\Zn$ on $\lgr$. Recall first a known result 
due to {\sc Varopoulos}~\cite{Varopoulos1983}
and {\sc Pittet and Saloff-Coste}~\cite{Pittet_Saloff2001}
on the return probabilities of random walks on $\lgr$.

\begin{theorem}
If $\Zn$, with $Z_n=(\eta_n,X_n)$, is the Switch-Walk-Switch random walk
on $\lgr$ and the transition
matrix on $\Z_2$ is uniform, i.e., $p(\cdot,\cdot)=1/2$, then the
$n$-step return probabilities are
\begin{equation*}
 q^{(n)}\big((\eta,x),(\eta,x)\big)=\mathbb{E}_x\Big(2^{-R_n}\mathbf{1}_{\{X_n=x\}}\Big),
\end{equation*}
where $\mathbb{E}_x$ is the expectation on the trajectory space of $\Xn$
starting at $x$, and $R_n$ represents the \index{random walk!range}\textit{range} 
of the random walk $\Xn$ on $\mathsf{G}$. 
\end{theorem}

The \index{random walk!range of a~}\textit{range} $R_n$ of a random walk $\Xn$
is defined as the number of distinct
visited points up to time $n$ by the random walk, that is
\begin{equation*}
R_n=|\{X_1,X_2,\ldots ,X_n\}|.
\end{equation*}
So, in order to derive asymptotics for the return probabilites $q^{(n)}$
of $\Zn$ on $\lgr$, as $n$ goes to infinity, 
it is enough to study the asymptotics for the 
range $R_n$ of its underlying walk $\Xn$ on $\mathsf{G}$. 

\paragraph*{Asymptotics for $\mathsf{G}=\Z^d$.}
There are several results on asymptotics of return probabilites of LRW
over base groups which have polynomial growth, for instance on $\Z^d$, for all $d\geq 1$.
We state here some of them.

For random walks on $\Z^{d}$, 
{\sc Donsker and Varadhan}~\cite{DonskerVaradhan1979} studied
the asymptotic behaviour
of the Laplace transform of the range $R_n$
\begin{equation*}
\mathbb{E}[\exp\{-t R_n\}], \text{ as }n\to\infty
\end{equation*}
for $t>0$. This behaviour depends on what is assumed
about the one-step transition probabilities. 
They proved the following important theorem.
\begin{theorem}\label{thm:varadhan}
For simple random walks on $\Z^{d}$,
\begin{equation*}
-\log\mathbb{E}[\exp\{-t R_n\}]\sim c(d)\cdot t^{d/(d+2)}\cdot n^{d/(d+2)}, \text{ as } n\to\infty,
\end{equation*}
where 
\begin{equation*}
c(d)=2^{-1}(d+2)\omega_d^{2/(d+2)}(\lambda_d/d)^{d/(d+2)} 
\end{equation*}
and $\lambda_d$ is the lowest eigenvalue of the Laplacian with Dirichlet boundary condition
in the Euclidean ball of radius $1$, and $\omega_d=\pi^{d/2}/\Gamma(d/2+1)$ its volume.
\end{theorem}
Here, for sequences $(a_n)$ and $(b_n)$ of real numbers, we indicate by $a_n\sim b_n$
that their quotient tends to $1$. We say that $a_n\preccurlyeq b_n$ if there are 
$C\geq c > 0$, such that for all sufficiently large $n$,
\begin{equation*}
a_n\leq C \sup\{ b_k:cn\leq k\leq Cn\}. 
\end{equation*}
If also $b_n \preccurlyeq a_n$,
then we write $a_n \approx b_n$. An equivalence class of sequences under 
this relation is called an \textit{asymptotic type}. Note that the asymptotic
type of a sequence is not as sharp as asymptotic equivalence.
For example, sequences of the form $(e^{-\lambda n}Q(n))$ 
(where $\lambda>0$ and $Q$ is a polynomial) are all of asymptotic type $(e^{-n})$.

Using the previous Theorem, one gets asymptotic equivalence 
for the return probabilities $q^{(n)}$ of the LRW on $\Z_2\wr\Z^d$. 

{\sc Revelle}~\cite{Revelle2003}, computed precise asymptotics for 
Switch-Walk-Switch lamplighter walks on $\Z_2\wr\Z$. He obtained that
\begin{equation*}
q^{(n)} \sim c_1 n^{1/6} \cdot \exp\{-c_2n^{1/3}\},
\end{equation*}
using the relation with the one-dimensional trapping problem.

\paragraph*{Asymptotics for $\mathsf{G}=\mathcal{T}$.}
Consider now the underlying graph as being an oriented tree
with a fixed end $\omega$ like in Figure \ref{fig:tree_zero_drift}
and $\Xn$ the random walk with transition probabilities $P=(p(x,y))$ 
given in \eqref{eq:tr_pb_fixed_end}, which has zero modular drift $\delta(P)$.

We are interested in asymptotics (asymptotic type or asymptotic equivalence) 
of  the return probabilities $q^{(n)}$ for SWS
lamplighter random walks $\Zn$ on $\Z_2\wr\mathcal{T}$, such that the base random walk
$\Xn$ has zero drift on $\mathcal{T}$. Such a situation was not considered until now.
The precise asymptotics for the range $R_n$ are hard to determine.

An easy upper estimate can be obtained if we look 
at the horocyclic projection $h(X_n)$ of $X_n$ on the integers, which is
a simple random walk on $\Z$. Moreover, the range $\tilde{R}_n$ of the 
horocyclic projection is much smaller than the range of $\Xn$ on $\mathcal{T}$. 
Using the large deviation estimate 
of Theorem \ref{thm:varadhan} for $\tilde{R}_n$ on $\Z$, one has
\begin{equation*}
-\log\mathbb{E}[\exp\{-t \tilde{R}_n\}]\sim c\cdot t^{2/3}\cdot n^{1/3}, \text{ as } n\to\infty.
\end{equation*}
Since $\tilde{R}_n\leq R_n$, we get the following upper estimate
for the return probabilities of LRW on $\Z_2\wr\mathcal{T}$:
\begin{equation*}
 q^{(n)}\big((\mathbf{0},x),(\mathbf{0},x)\big)\leq C \cdot\rho(Q)^n\cdot \exp\{-cn^{1/3}\},
\end{equation*}
where $Q$ is the transition matrix of the Switch-Walk-Switch random walk on 
$\Z_2\wr\mathcal{T}$, and $\rho(Q)$ its spectral radius. Recall that $\mathbf{0}$
is the trivial configuration, where all lamps are off.

\textbf{Question 1:} How can one find a good lower bound for 
$q^{(n)}\big((\mathbf{0},x),(\mathbf{0},x)\big)$?
If a lower estimate of the same order $\exp\{-cn^{1/3}\}$ can be found, then one
would have the asymptotic type for the return probabilities, which is 
weaker than the precise asymptotics.

\textbf{Question 2:} How can one find asymptotics for the range $R_n$
of random walks on graphs with exponential growth, say trees? This 
problem was not considered up to now. Such estimates are well studied
for random walks on groups and graphs with polynomial growth, where
gaussian estimates are available.

\part{Entropy-Sensitivity of Languages via Markov Chains}
\chapter{Languages on Labelled Graphs}\label{sec:entropy}

This part of the thesis is based on the paper by {\sc Huss, Sava and Woess} 
\cite{HussSavaWoess2010}.

A language $L$ over a finite alphabet $\Si$ is called \emph{growth sensitive} 
(or \emph{entropy sensitive}) if 
forbidding any finite set of factors $F$ yields a sublanguage $L^F$ whose 
exponential growth rate (entropy) is smaller than that of $L$. Let 
$(X, E, \ell)$ be an infinite, oriented, edge-labelled graph with label 
alphabet $\Si$. Considering the edge-labelled graph as an (infinite) automaton,  
we associate with any pair of vertices $x,y \in X$ the language $L_{x,y}$ 
consisting of all words that can be read as labels along some path 
from $x$ to $y$. Under suitable general assumptions we prove that these 
languages are growth sensitive. This is based on using Markov chains with 
forbidden transitions. 

\section{Introduction}\label{sec:intro}

Let $\Si$ be a finite \textit{alphabet} and $\Si^{*}$ the set of all finite 
words over $\Si$, including the empty word $\epsilon$. 
A \index{language}\textit{language} $L$ over $\Si$ is a subset of $\Si^{*}$. 
All our languages will be infinite. 
We denote by $\abs{w}$ the length of the word $w$. 
A \textit{factor} of a word $w=a_1a_2\ldots a_n$ is
a word of the form $a_ia_{i+1}\ldots a_j$, with $1\leq i\leq j\leq n$. 
The \index{language!growth}\textit{growth} or 
\index{language!entropy}\textit{entropy} of $L$ is 
\begin{equation*}
\entr(L) 
= \limsup_{n\to\infty}
\frac{1}{n} \log \bigl|\{w \in L:\: \abs{w} = n\}\bigr|.
\end{equation*}
For a finite, non-empty set 
$F\subset\Si^+ = \Si^*\setminus \{\epsilon\}$ consisting of factors 
of elements of $L$, we let 
\begin{equation*}
 L^{F} 
 = \{w\in L:\:\text{no}\; v\in F\; \text{is a factor of}\; w\}.
\end{equation*}
The issue addressed here is to provide conditions under which, for a
class of languages associated with infinite graphs,
$\entr(L^{F})<\entr(L)$. If this holds for \textit{any}
set $F$ of \textit{forbidden factors}, then the language $L$ is called
\textit{growth sensitive} (or \textit{entropy sensitive}). 

Questions related to growth sensitivity have been considered in different
contexts.

In \textit{group theory} in relation to regular normal forms of finitely 
generated groups, the study of growth-sensitivity has been proposed by 
{\sc Grigorchuk and de la Harpe}~\cite{Grigorchuk_Harpe1997} as a tool for proving 
the Hopfianity of a given group or class of groups; see also
{\sc Arzhantseva and Lysenok}~\cite{ArzhantsevaLysenok2002} and
{\sc Ceccherini-Silberstein and Scarabotti}~\cite{CeccheriniScarabotti2004}. 
A group is called \textit{Hopfian} if it is not isomorphic with a proper quotient of
itself. The basic example were this tool applies is the free group. 

In \textit{symbolic dynamics}, the number $\entr(L)$ associated with a regular
language accepted by a finite automaton with suitable properties
appears as the \textit{topological entropy} of a \textit{sofic system}; see
{\sc Lind and Marcus}~\cite[Chapters 3 \& 4]{LindMarcus1995}. 
Entropy sensitivity appears as the strict inequality between the entropies
of an irreducible sofic shift and a proper subshift~\cite[Cor. 4.4.9]{LindMarcus1995}.

Motivated by these bodies of work, {\sc Ceccherini-Silberstein and 
Woess} \cite{Ceccherini_Woess2002}, \cite{Ceccherini2007} have elaborated
practicable criteria that guarantee the  growth sensitivity of \textit{context-free}
languages. 

The main result of this chapter can be seen as a direct extension of 
\cite[Cor. 4.4.9]{LindMarcus1995} to the entropies of infinite sofic systems; see below 
for further comments and references.  This will be done using a 
probabilistic approach, namely considering Markov chains with forbidden
transitions.

Our basic object is an \index{graph!labelled graph}\textit{infinite directed 
graph} $(X,E, \ell)$ 
whose edges are labelled by elements of a finite alphabet $\Si$. Each edge 
has the form $e=(x,a,y)$, where $e^-=x$ and $e^+= y \in X$ are the initial 
and the terminal vertices of $e$, and $\ell(e) = a \in \Si$ is its label. 
We will also write $x \xrightarrow{a} y$ for the edge $e=(x,a,y)$, or just 
$x\rightarrow y$ in situations where we do not care about the label.
Multiple edges and loops are allowed, but two edges with the same end 
vertices must have distinct labels. 

A \textit{path} of length $n$ in $(X,E, \ell)$ is a sequence $\pi=e_{1}e_{2}\ldots e_{n}$ 
of edges such that $e_{i}^{+}=e_{i+1}^{-}$, for $i=1,2,\ldots n-1$. 
We say that it is a path from $x$ to $y$,	 if $e_1^-=x$ and $e_n^+ =y$.
The label $l(\pi)$ of $\pi$ is the word
$\ell(\pi)=\ell(e_{1})\ell(e_{2})\ldots \ell(e_{n})\in \Si^{*}$ that we read 
along the path. We also allow the empty path from $x$ to $x$, whose label is 
the \textit{empty word} $\epsilon \in \Si^*$. 
For $x,y\in X$, denote by $\Pi_{x,y}$ the set of all paths $\pi$ from $x$ to $y$ in $(X,E, \ell)$. 

The languages which we consider here are
\begin{equation*} 
L_{x,y}
=\{\ell(\pi)\in\Si^{*}: \pi\in \Pi_{x,y}\}, 
\text{ where } x,y\in X. 
\end{equation*}
That is, we consider the edge-labelled graph $(X,E, \ell)$ as an infinite automaton (labelled 
digraph) with initial
state $x$ and terminal state $y$, so that $L_{x,y}$ is the language
accepted by the automaton. 

\begin{minipage}[b]{0.55\linewidth}
The languages under study will be infinite. 
Also, we shall require that the growth is very fast (exponential),
since in the subexponential growth case the entropy $\entr$
is zero.
\begin{example}
Let us consider the following finite labelled graph $(X,E, \ell)$ given in the Figure.
For $x,y\in X$, the language
$L_{x,y}$ is the set of all labeles of paths from $x$ to $y$, 
that is
\begin{equation*}
 L_{x,y}=\{a(a)^*bb,(ba)^*bb,bb,\ldots\}
\end{equation*}

\end{example} 
\end{minipage}
\hspace{0.5cm}
\begin{minipage}[b]{0.45\linewidth}
\begin{tikzpicture}[scale=0.5]\label{fig:det_graph}
\begin{scope}
[yshift=-3.65cm,vertex/.style={circle, draw=black, thick, inner sep=2pt, minimum size=7mm},
 pre/.style={->, shorten >=1pt, >=stealth', semithick},
 post/.style={<-, shorten <=1pt, >=stealth', semithick},
 node distance = sqrt(2.5^2/2)
]

\node (origin) {};
\node[vertex] (a_2) [left=of origin] {$x$};
\node[vertex] (b_1) [right=of origin] {$y$};
\node[vertex] (a_1) [above=of origin] {$z$}
   edge [post] node[auto, swap] {$a$} (a_2)
   edge [post] node[auto] {$a$} (b_1)
   edge [post, loop above] node[auto] {$a$} ();
\node[vertex] (b_2) [below=of origin] {$t$}
   edge [post] node[auto] {$b$} (a_1)
   edge [pre, bend left=20] node[auto] {$a$} (a_2)
   edge [post, bend right=20] node[auto, swap] {$b$} (a_2)
   edge [pre, bend left=20] node[auto] {$b$} (b_1)
   edge [post, bend right=20] node[auto, swap] {$b$} (b_1);
\end{scope}	
\end{tikzpicture}
\end{minipage}

\begin{definition}
We say that $(X,E,\ell)$ is 
\index{graph!labelled graph!deterministic}\emph{deterministic} if, for every vertex
$x$ and every $a \in \Sigma$, there is at most one edge with initial point $x$
and label $a$. 
\end{definition}
Any automaton (finite or infinite) can be transformed into a 
deterministic one that accepts the same language, by the well-known powerset construction.
See, for example \cite[Prop. 1.4.1]{BerstelPerrinaReutenauer2010}.

As in the finite case, we need an irreducibility assumption. 
The graph $(X,E, \ell)$ is called 
\index{graph!labelled graph!strongly connected}\textit{strongly connected} if, for every
pair of vertices $x$, $y$, there is an (oriented) path from $x$ to $y$. 

\begin{definition}
The graph $(X,E, \ell)$ is called  \index{graph!labelled graph!uniformly connected}\emph{uniformly connected}
if it is strongly connected and the following holds:
there is a constant $K$, such that for every edge $x\rightarrow y$ there is
a path from $y$ to $x$ with length at most $K$.
\end{definition}
In the finite case, the two notions coincide as one can take $K=|X|$.
The \textit{forward distance}
$d^+(x,y)$ of $x,y \in X$ is the minimum length of a path from $x$ to $y$.
We write 
\begin{equation*}
\entr(X) = \entr(X,E,\ell) = 
\sup_{x,y\in X} \entr(L_{x,y})
\end{equation*}
and call this the \index{graph!labelled graph!entropy}\textit{entropy} of our oriented, labelled graph.
It is a well-known and easy to prove fact, that for a strongly connected graph,
\begin{equation*}
\entr(L_{x,y}) = \entr(X), \text{ for all } x, y \in X. 
\end{equation*}
We also need a reasonable assumption on the set of forbidden factors.

\begin{definition}
We say that a finite set $F \subset \Si^+$ is 
\index{language!relatively dense subset}\emph{relatively dense}
in the graph $(X,E, \ell)$ if there is a constant $D$ such that, for every
$x \in X$, there are $y \in X$ and $w \in F$, such that $d^+(x,y) \le D$
and there is a path starting at $y$ which has label $w$. 
\end{definition}

Note that the assumptions of uniformly connectedness and relatively denseness
cannot be avoided, since they play an important role in the proof of
the main result. This fails withous these assumptions.

\begin{theorem}\label{thm:A}
Suppose that $(X,E,\ell)$ is uniformly connected and deterministic with
label alphabet $\Si$. Let $F \subset \Si^+$ be a finite, non-empty set
which is relatively dense in $(X,E, \ell)$. Then 
\begin{equation*}
\sup_{x,y\in X}\entr(L_{x,y}^F) < \entr(X).
\end{equation*}
\end{theorem}

We say that $(X,E, \ell)$ is \textit{fully deterministic} if, for every
$x \in X$ and $a \in \Sigma$, there is precisely one edge with initial point $x$
and label $a$. We remark that, in automata theory, the classical terminology is deterministic and
complete, instead of fully deterministic. Since in graph theory 
a complete graph is one in which every pair a distinct vertices is connected 
by an unique edge, we shall use the notion of fully deterministic
throughout this work.

As a consequence of Theorem \ref{thm:A} one can easily prove the
following.

\begin{corollary}\label{cor:B}
If $(X,E,\ell)$ is uniformly connected and fully deterministic then
$L_{x,y}$ is growth sensitive for all $x,y \in X$.
\end{corollary}

Indeed, in this case, for every $x \in X$ and every $w \in \Si^*$, there
is precisely one path with label $w$ starting at $x$.

With our edge-labelled graph $(X,E,\ell)$, we can consider the \textit{full shift space} 
which consists of all bi-infinite words over $\Si$ that can be read along the edges of 
some bi-infinite path in $(X,E, \ell)$. When $(X,E,\ell)$ is strongly connected,
the entropy $\entr(L_{x,y})$ is independent of $x$ and $y$ and 
equals the topological entropy of the full shift space of the graph.
See, for example,  {\sc Gurevi\v c}~\cite{Gurevic1969}, 
{\sc Petersen}~\cite{Petersen1986} or {\sc Boyle, Guzzi and G\'omez}~\cite{Boyle_Buzzi_Gomez2006}
for a selection of related work and references, and also
the discussion in \cite[\S 13.9]{LindMarcus1995}.  

If we consider the shift space consisting of all those
bi-infinite words as above that do not contain any factor in $F$,
then the interpretation of Corollary \ref{cor:B} is that
the associated entropy is
strictly smaller than $\entr(X)$.

Theorem \ref{thm:A}, once approached in the right way, is not hard to prove.
It is based on a classical tool, a version of the Perron-Frobenius 
theorem for infinite
non-negative matrices; see, for example, {\sc Seneta}~\cite{Seneta2006}. We shall first
reformulate things in terms of Markov chains (random walks)
and forbidden transitions.

\section{Markov Chains and Forbidden Transitions}\label{sec:Markov}

We now equip the oriented, edge-labelled graph $(X,E,\ell)$ with additional data:
with each edge $e=(x,a,y)$, we associate a probability 
\begin{equation*}
p(e) = p(x,a,y) \ge \alpha > 0, 
\end{equation*}
where $\alpha$ is a fixed constant, such that
\begin{equation}\label{eq:substoch}
\sum_{e \in E \,:\,e^-=x} p(e) \le 1 \quad\text{for every}\; x \in X\,.
\end{equation}
Our assumption to have the uniform lower bound $p(e) \ge \alpha$ for each 
edge implies that the outdegree (number of outgoing edges) of each vertex
is bounded by $1/\alpha$.
We interpret $p(e)$ as the probability that a particle with current position
$x=e^-$ moves in one (discrete) time unit along $e$ to its end vertex $y = e^+$.
Observing the successive random positions of the particle at the time 
instants $0,1,2,\dots$, we obtain a Markov chain with state space $X$ whose
one-step transition probabilities are 
\begin{equation*}
p(x,y) = \sum_{a \in \Si : (x,a,y) \in E} p(x,a,y).
\end{equation*}
We shall also want to record the edges and their labeles used in each step,
which means considering a Markov chain on a somewhat larger 
state space, but we will not need to formalize it in detail.
In \eqref{eq:substoch}, we admit the possibility that $1 - \sum_y p(x,y) > 0$
for some $x$. This number is then interpreted as the probability that a
particle positioned at $x$ dies at the next step.

We write $p^{(n)}(x,y)$ for the probability that the particle starting at
$x$ is at position $y$ after $n$ steps. This is the $(x,y)$-element of
the $n$-power $P^n$ of the transition matrix 
$P = \bigl( p(x,y) \bigr)_{x,y \in X}\,$. If $(X,E, \ell)$ is strongly connected,
then $P$ is irreducible, and it is well-known that the number
\begin{equation*}
\rho(P) = \limsup_{n \to \infty} p^{(n)}(x,y)^{1/n}
\end{equation*}
is independent of $x$ and $y$. See once more \cite{Seneta2006}.
The quantity $\rho(P)$ is called the spectral radius of $P$. It is the parameter
of exponential decay of the transition probabilities.

Let once more $F \subset \Si^+$ be finite. We interpret the elements of
$F$ as sequences of \textit{forbidden transitions}. That is, we restrict 
the motion of the particle: at no time is it allowed to traverse any path  
$\pi$ with $\ell(\pi) \in F$ in $k$ successive steps, where $k$ is
the length of $\pi$. The words in $F$ are forbidden for the Markov chain.
We write $p^{(n)}_F(x,y)$ for the probability that the particle starting at
$x$ is at position $y$ after $n$ steps, without having made any such sequence
of forbidden transitions. Let
\begin{equation*}
\rho_{x,y}(P_F) = \limsup_{n \to \infty} p_F^{(n)}(x,y)^{1/n}, \quad
x,y \in X\,.
\end{equation*}
These numbers are not necessarily independent of $x$ and $y$, and \textit{they are not}
the elements of the $n$-matrix power of some substochastic matrix. 

\begin{definition}
A transition matrix $Q = \big(q(x,y)\big)_{x,y\in X}$ on the state space $X$
is called \textit{substochastic} if there exists a constant
$\varepsilon > 0$ such that, for all $x\in X$,
\begin{equation*}
\sum_{y\in X} q(x,y) \leq 1-\varepsilon.
\end{equation*}
That is, all row sums are bounded by $1-\varepsilon$.
\end{definition}

The restricted matrix $P_F$ does not represent the transition matrix
of a Markov chain.
In order to give an upper bound for the restricted transition 
probabilities $p^{(n)}_F(x,y)$, we first show the following.  

\begin{theorem}\label{lem:q_str_substoch}
Suppose that $(X,E,l)$ is strongly 
connected with label alphabet $\Si$ and equipped with transition 
probabilities $p(e) \ge \alpha > 0$, $e \in E$. Let $F \subset \Si^+$ be a 
finite, non-empty set which is relatively dense in $(X,E, \ell)$. Then there are
$k\in\mathbb{N}$ and $\eps_0>0$ such that  
$$
\sum_{y \in X} p_F^{(k)}(x,y) \le 1-\eps_0 \quad \text{for all}\; x \in X\,.
$$
In other words, the transition matrix 
$Q = \bigl(p_F^{(k)}(x,y)\bigr)_{x,y \in X}$ is strictly substochastic, with all row sums bounded by $1-\eps_0\,$.
\end{theorem}
\begin{proof}
Let $R = \max_{w \in F} |w|$, and let $D \in \N$ be the constant from the
definition of relative denseness of $F$. Set 
\begin{equation*}
k=D+R.
\end{equation*}
For each $x \in X$, we can find a path $\pi_1$ from $x$ to some 
$y \in X$ with length $d \le D$ and a path $\pi_2$ starting at $y$ which 
has label $w\in \Si^*$. Let $z$ be the endpoint of $\pi_2$, and choose any 
path $\pi_3$ that starts at $z$ and has length $k - d - |w|$. 
Such a path exists by strong connectedness. Then let $\pi$ be the path 
obtained by concatenating $\pi_1$, $\pi_2$ and $\pi_3$.

The probability that the Markov chain starting at $x$ makes its first $k$
steps along the edges of $\pi$ is
\begin{equation*}
\mathbb{P}(\pi)\geq \alpha ^{k}=\eps_0 > 0. 
\end{equation*}
Hence
\begin{equation*}
\sum_{y\in X} p^{(k)}_F(x,y)\leq \sum_{y\in X} p^{(k)}(x,y) -\mathbb{P}(\pi)
\leq 1-\eps_0, 
\end{equation*}
and this upper bound holds for every $x$.
\end{proof}
Given that the transition matrix $Q$ is substochastic, it is 
an easy exercise to prove that also its $n$-matrix power $Q^n$ is
also substochastic and the row sums of $Q^n$ are bounded from above
by $(1-\eps_0)^n$.

The matrix $P$ acts on functions $h: X \to \R$ by 
$Ph(x) = \sum_y p(x,y) h(y)$.
Next, we state two key results due to  {\sc Pruitt}~\cite[Lemma 1]{pruitt_1964} and 
\cite[Corollary to Theorem 2]{pruitt_1964}, which will be used in
the proof of the main result.

\begin{lemma}\label{lem:pruitt1}
If the transition matrix $P$ is irreducible and $Ph \leq s h$ for some 
$s > 0$ and $h \not= 0$, then $h > 0$.
\end{lemma}

\begin{lemma}\label{lem:pruitt2}
If the transition matrix $P = \big( p(x,y)\big)_{x,y\in X}$ is such that 
for every $x\in X$ the entries $p(x,y) = 0$ for all $y\in X$
except finitely many, then the equation
\begin{equation*}
P h = s h
\end{equation*}
has a solution for all $s \geq \rho(P)$.
\end{lemma}
Based on these lemmatas, we prove the following result on sensitivity
of the Markov chain with respect to forbidding the transitions in $F$.

\begin{theorem}\label{thm:C}
Suppose that $(X,E,\ell)$ is uniformly 
connected with
label alphabet $\Si$ and equipped with transition probabilities
$p(e) \ge \alpha > 0$, $e \in E$. Let $F \subset \Si^+$ be a 
finite, non-empty set which is relatively dense in $(X,E, \ell)$.
Then 
\begin{equation*}
\sup_{x,y \in X} \rho_{x,y}(P_F) < \rho(P) \quad\text{strictly.}
\end{equation*}
\end{theorem}

\begin{proof} We shall proceed in two steps.

\emph{Step 1. We assume that $P = \bigl( p(x,y) \bigr)_{x,y \in X}$ 
is stochastic and that $\rho(P)=1$.} 

Consider the matrix $Q$ of Lemma \ref{lem:q_str_substoch}. Let
\begin{equation*}
Q^n = \bigl( q^{(n)}(x,y) \bigr)_{x,y \in X} 
\end{equation*}
be its $n$-th matrix power. The quantity $q^{(n)}(x,y)$ is 
the probability that the Markov chain starting at $x$
is in $y$ at time $nk$ and does not make any forbidden sequence of transitions
in each of the discrete time intervals
\begin{equation*}
[(j-1)k,jk] \text{ for } j \in \{1,\dots, n\}. 
\end{equation*}
Therefore
\begin{equation*}
p^{(nk)}_F(x,y) \le q^{(n)}(x,y)\,,
\end{equation*}
and also, by the same reasoning,  for $i=0, \dots, k-1$,
\begin{align*}
p^{(nk+i)}_F(x,y)& = \sum_{z\in X}p^{(nk)}_F(x,z)p^{(i)}_F(z,y) \\
& \le \sum_{z \in X} q^{(n)}(x,z)p^{(i)}_F(z,y)\,,\quad
i=0\, \dots, k-1.
\end{align*}
Therefore, for every $x \in X$ and $i=0, \dots, k-1$,
\begin{equation*}
\sum_{y \in X} p^{(nk+i)}_F(x,y) 
\le \sum_{z \in X} q^{(n)}(x,z) 
\underbrace{\sum_{y \in X}p^{(i)}_F(z,y)}_{\displaystyle\le 1}
\le (1-\eps_0)^n\,,
\end{equation*}
since Lemma \ref{lem:q_str_substoch} implies that the row sums of the
matrix power $Q^n$ are bounded above by $(1-\eps_0)^n$.
We conclude that  
\begin{equation*}
\limsup_{n\to\infty} p_F^{(nk+i)}(x,y)^{1/(nk+i)} \leq (1-\eps_0)^{1/k}\,,
\end{equation*}
so $\rho_{x,y}(P_F) \le (1-\eps_0)^{1/k} = 1-\eps$,
where $\eps > 0$. 

\smallskip\noindent
\emph{Step 2. General case. } We reduce this case to the previous one. 

Since $P$ is irreducible and
every row of $P$ has only finitely many non-zero entries,
Lemma \ref{lem:pruitt1} and Lemma \ref{lem:pruitt2}
guarantee the existence of a strictly
positive solution $h:X\to\R$ for the equation 
\begin{equation*}
P h = \rho(P) \cdot h,
\end{equation*}
that is, $h$ is \emph{$\rho(P)$-harmonic.}
Consider now the $h$-transform of the transition probabilities 
$p(e)$ of $P$, $e = (x,a,y) \in E$, given by
\begin{equation*}
p^h(e) = p^h(x,a,y) = \frac{p(x,a,y) h(y)}{\rho(P) h(x)},
\end{equation*}
and the associated transition matrix $P^h$ with entries 
\begin{equation*}
p^h(x,y) = \sum_{a\,:\, (x,a,y) \in E} p^h(x,a,y)\,.
\end{equation*}

The Markov chain associated with $P^h$ is called the \emph{$h$-process.}

Then $\rho(P^h)=1$. Using uniform connectedness, we show that there is a 
constant $\bar \alpha > 0$ such that $p^h(e) \ge \bar\alpha$ for each
$e =(x,a,y)\in E$. Indeed, for such an edge, there is $k \le K$ such
that $d^+(y,x)=k$, whence
\begin{equation*}
\rho(P)^k h(y)  = \sum_{z \in X} p^{(k)}(y,z) h(z) \ge \alpha^k h(x),
\end{equation*}
so
\begin{equation*}
p^h(x,a,y)\ge \bigl(\alpha/\rho(P)\bigr)^{k+1}.
\end{equation*}
Recall that $K$ is the constant used in the definition of the uniform connectedness.
We can now choose 
\begin{equation*}
\bar \alpha = \bigl(\alpha/\rho(P)\bigr)^{K+1}. 
\end{equation*}
We see that with $P^h$ we are now in the situation of \textit{Step 1}. Thus, forbidding
the transitions of $F$ for the Markov chain with transition matrix $P^h$, 
we get 
\begin{equation*}
\rho_{x,y}(P^h_F) \le 1 - \eps, \text{ for all  } x,y \in X, 
\end{equation*}
where $\eps > 0$. We now show that
\begin{equation*}
\rho_{x,y}(P^h_F) = \rho_{x,y}(P_F)/\rho(P), 
\end{equation*}
which will conclude the proof.

For a path $\pi = e_1 \dots e_n$ from $x$ to $y$, let (as above) 
$\mathbb{P}(\pi)$ be the probability that the original Markov chain traverses
the edges of $\pi$ in $n$ successive steps, and let $\mathbb{P}^h(\pi)$
be the analogous probability with respect to the $h$-process.
Then
$$
\mathbb{P}^h(\pi) = \frac{\mathbb{P}(\pi)h(y)}{\rho(P)^nh(x)}\,.
$$
Let us write $\Pi_{x,y}^n(\neg F)$ for the set of all paths $\pi$ from
$x$ to $y$ with length $n$ for which $\ell(\pi)$ does not
contain a factor in $F$. Then the $n$-step transition probabilities
of the $h$-process with the transitions in $F$ forbidden are
\begin{equation*}
{p^h}^{(n)}_F(x,y) = \sum_{\pi \in \Pi_{x,y}^n(\neg F)} \mathbb{P}^h(\pi)
= \sum_{\pi \in \Pi_{x,y}^n(\neg F)}\frac{\mathbb{P}(\pi)h(y)}{\rho(P)^nh(x)}
= \frac{p^{(n)}_F(x,y)h(y)}{\rho(P)^nh(x)}.
\end{equation*}
Taking $n$-th roots and passing to the upper limit, 
we obtain the required identity.
\end{proof}

With this result, it is now easy to deduce Theorem \ref{thm:A}.

\begin{proof}[Proof of Theorem \ref{thm:A}]
Since $(X,E,l)$ is deterministic with label alphabet $\Si$, 
the outdegree of every $x\in X$ is at most $|\Si|$. Equip the edges of $(X,E, \ell)$ 
with the transition probabilities 
\begin{equation*}
p(x,a,y) = \frac{1}{|\Si|}, \text{ when } (x,a,y)\in E.
\end{equation*}
Then the $n$-step transition probabilities of the resulting Markov chain are given by
\begin{equation*}
p^{(n)}(x,y)= \dfrac{\bigl|\{w \in L_{x,y}\,:\, \abs{w} = n\}\bigr|}{|\Si|^n}.
\end{equation*}
Therefore 
(because $(X,E, \ell)$ is uniformly connected)
\begin{align*}
\entr(X)=\entr(L_{x,y}) & =\limsup_{n\to\infty} \frac{1}{n} \log \bigl|\{w \in L_{x,y}:\: \abs{w} = n\}\bigr|\\
& =\limsup_{n\to\infty}\frac{1}{n} \log \bigl(p^n(x,y)|\Si|^n\bigr)
=\log \bigl(\rho(P)\cdot|\Si|\bigr).
\end{align*}
Analogously,
\begin{equation*}
\entr(L^F_{x,y})=\log\bigl(\rho_{x,y}(P_F)\cdot |\Si|\bigr).
\end{equation*}
By Theorem \ref{thm:C}
\begin{equation*}
\sup_{x,y \in X} \rho_{x,y}(P_F) < \rho(P),
\end{equation*}
and this implies that 
\begin{equation*}
\sup_{x,y\in X}\entr(L_{x,y}^F) < \entr(X)
\end{equation*}
strictly.
\end{proof}

\section{Application to Schreier Graphs}

Let $G$ be a finitely generated group and $K$ a (not necessary finitely 
generated) subgroup.  Also, let $\Si$ be a finite alphabet and let
$\psi:\Si\rightarrow G$ be such that the set $\psi(\Si)$ generates $G$ as 
a semigroup. We extend $\psi$ to a monoid homomorphism from $\Si^*$
to $G$ by $\psi(w) = \psi(a_1)\cdots \psi(a_n)$ if $w = a_1 \dots a_n$
with $a_i \in \Si$ (and $\psi(\epsilon) = 1_G$). The mapping $\psi$ is
called a \emph{semigroup presentation} of $G$ in \cite{Ceccherini_Woess}.  

The \index{graph!Schreier graph}\textit{Schreier graph}  $X=X(G,K,\psi)$ has vertex set
\begin{equation*}
X=\{Kg: g\in G\}, 
\end{equation*}
the set of all right $K$-cosets in $G$, and the set of all labelled, 
directed edges $E$ is given by
\begin{equation*}
 E=\{e=(x,a,y): x=Kg, y=Kg \psi(a)\,,\; \text{where}\; g\in G\,,\;a\in\Si\}.
\end{equation*}
Note that the graph $X$ is fully deterministic and uniformly connected. 

The \textit{word problem} of $(G,K)$ with respect to $\psi$ is the language 
\begin{equation*}
 L(G,K,\psi)=\{w\in\Si^*: \psi(w)\in K\}.
\end{equation*}
The \index{group!word problem}\textit{word problem} for a 
recursively presented  group $G$ is the algorithmic problem 
of deciding whether two words represent the same element. Also, this
terminology is used in the context of formal language theory and goes back
at least to the seminal paper of {\sc Muller and Schupp}~\cite{MullerSchupp1983}.
For additional information, see also {\sc Muller and Schupp}~\cite{MullerSchupp1985}.
In their work, for a finitely generated group $G$ the \textit{word problem}
$W(G)$ is the set of all words on the generators and their inverses which
represent the identity element of $G$.

If we consider the ``root'' vertex $o=K$ of the Schreier graph, then
in the notation of 
the introduction, 
we have $L(G,K,\psi)=L_{o,o}$; compare with
\cite[Lemma 2.4]{Ceccherini_Woess}.

We can therefore apply Theorem \ref{thm:A} and Corollary \ref{cor:B} to the 
graph $X(G,K,\psi)$ in order to deduce the following.

\begin{corollary}
The word problem of the pair $(G,K)$ with respect to any semigroup 
presentation $\psi$ is growth sensitive (with respect to forbidding an 
arbitrary non-empty finite subset $F\subset\Si^*$). 
\end{corollary}

\begin{example}
Let $G=\mathbb{Z}_2=\{1,t\}$ be the group of order two and $K=\{1\}$
the trivial subgroup. Let $\Si=\{a\}$ and consider the presentations
$\psi:\Si\to G$ such that $\psi(a)=t$. Then
\begin{equation*}
 L(G,K,\psi)=\{a^{2n}: n\geq 0 \}.
\end{equation*}
\end{example}
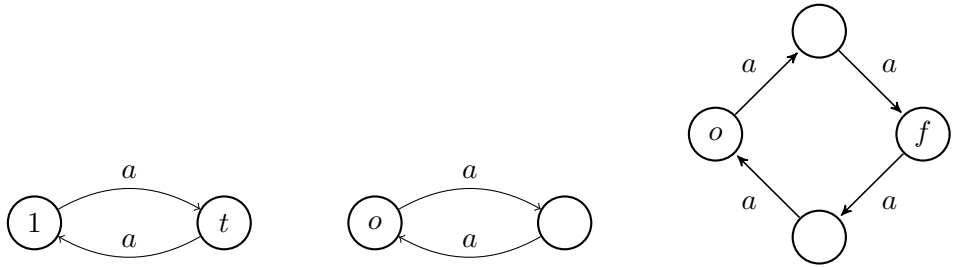
\begin{tikzpicture}
[yshift=-3.65cm,vertex/.style={circle, draw=black, thick, inner sep=2pt, minimum size=7mm},
 pre/.style={->, shorten >=1pt, >=stealth', semithick},
 post/.style={<-, shorten <=1pt, >=stealth', semithick},
 node distance = sqrt(2.5^2/2)
]
\node[vertex] (1) {$1$};
\node[vertex] (t) [right=of 1] {$t$};
\path[->] (1) edge [bend left] node [above] {$a$} (t);
\path[->] (t) edge [bend left] node [above] {$a$} (1);
\end{tikzpicture}
\hspace{1cm}
\begin{tikzpicture}
[yshift=-3.65cm,vertex/.style={circle, draw=black, thick, inner sep=2pt, minimum size=7mm},
 pre/.style={->, shorten >=1pt, >=stealth', semithick},
 post/.style={<-, shorten <=1pt, >=stealth', semithick},
 node distance = sqrt(2.5^2/2)
]
\node[vertex] (o) {$o$};
\node[vertex] (t) [right=of o] {};
\path[->] (o) edge [bend left] node [above] {$a$} (t);
\path[->] (t) edge [bend left] node [above] {$a$} (o);
\end{tikzpicture}
\hspace{1cm}
\begin{tikzpicture}[scale=0.05]
\begin{scope}
[yshift=-3.65cm,vertex/.style={circle, draw=black, thick, inner sep=2pt, minimum size=7mm},
 pre/.style={->, shorten >=1pt, >=stealth', semithick},
 post/.style={<-, shorten <=1pt, >=stealth', semithick},
 node distance = sqrt(1.5/2)
]

\node (origin) {};
\node[vertex] (o) [left=of origin] {$o$};
\node[vertex] (f) [right=of origin] {$f$};
\node[vertex] (b) [above=of origin] {}
   edge [pre] node[auto] {$a$} (f);
\node[vertex] (c) [below=of origin] {}
   edge [pre] node[auto] {$a$} (o);
\path[->] (o) edge [pre] node [auto] {$a$} (b);
\path[->] (f) edge [pre] node [auto] {$a$} (c);
\end{scope}	
\end{tikzpicture}

In the figure above are represented in order, the Schreier graph
$X(G,K,\psi)$ (which is the Cayley graph of $G$ with respect to
$\psi$), and two automata $\mathcal{A}_1$ and $\mathcal{A}_2$.
As usual $o$ denoted the origin, while the set of final states
are $F_1=\{o\}$ and $F_2=\{o,f\}$ respectively. We have
\begin{equation*}
 L(\mathcal{A}_1)=L(\mathcal{A}_2)=L(G,K,\psi).
\end{equation*}
For detailed information on properties of pairs of groups
and their Schreier graphs, the reader may have a look
at {\sc Ceccherini and Woess}~\cite{Ceccherini_Woess}.

\cleardoublepage
\phantomsection
\addcontentsline{toc}{chapter}{Acknowledgements}
\chapter*{Acknowledgements}

This work would not have been possible without the support
of my advisor, Wolfgang Woess. I am very grateful to him for
the numerous fruitful disscusions on topics of this thesis
and on mathematics in general. Acknowledgements go also to 
Wilfried Huss and all collegues from the Department of Mathematical Structure
Theory, at Graz University of Technology for several disscusions
on different mathematical problems and not only.  

My thanks go also to Vadim Kaimanovich, who gave me a lot of hints
and insights on this work. The main result and the idea of the proof 
in Section \ref{sec:zero_drift} would have not been possible without his help.

I also thank to NAWI Graz research program for supporting 
me with a Graduate Research Fellowship during much of the time when
this work was carried out. I want to thank to Prof. Bernd Thaller
for accepting to act as a co-advisor of my thesis.

Finally, I would like to thank to my family and my friends 
for supporting me and believing in me through the best and the worst of time.

\cleardoublepage
\phantomsection
\bibliography{mybib}{}
\bibliographystyle{amsalpha}

\cleardoublepage
\phantomsection
\addcontentsline{toc}{chapter}{Index}
\printindex

\addcontentsline{toc}{chapter}{Index on Notation}
\cleardoublepage
\phantomsection
\chapter*{Index of Notation}

\subsubsection{Part I}

\begin{tabular}{ l  l }
  $\mathsf{G}$              &    locally finite, connected, infinite transitive graph \\[1pt]
  $d(\cdot,\cdot)$ &    graph metric on $\mathsf{G}$ \\[1pt]
  $\partial \mathsf{G}$     &    geometric boundary of $\mathsf{G}$\\[1pt]
  $\widehat{\mathsf{G}}=\mathsf{G}\cup\partial \mathsf{G}$  & compactification of $\mathsf{G}$\\[1pt]
  $\Omega=\mathsf{G}^{\Z_+}$&    the trajectory space of $\mathsf{G}$\\[1pt]
  $P_{\mathsf{G}}$              &    transition matrix on $\mathsf{G}$\\[1pt]
  $AUT(\mathsf{G})$         &    the set of automorphisms (or isometries) of $\mathsf{G}$\\[1pt]
  $\Gamma$         &    subgroup of $\AUT(\mathsf{G})$ which acts transitively on $\mathsf{G}$\\[1pt]
  $\mu$            &    probability measure on $\Gamma$\\[1pt]
  $(\mathsf{G},P_{\mathsf{G}})$          &    Markov chain with state space $\mathsf{G}$ and transition $P_{\mathsf{G}}$\\[1pt]
  $\Xn$            &    random walk with transition matrix $P_{\mathsf{G}}$ on $\mathsf{G}$\\[1pt]
  $\mu_{\infty}$   &    limit distribution of $\Xn$ on $\partial \mathsf{G}$\\[1pt]
  $\rho(P_{\mathsf{G}})$        &    spectral radius of $(\mathsf{G},P_{\mathsf{G}})$\\[1pt]
  $\delta(P_{\mathsf{G}})$      &    modular drift of $P_{\mathsf{G}}$\\[1pt]
  $l(P_{\mathsf{G}})$           &    rate of escape (drift) of $P_{\mathsf{G}}$\\[1pt]
  $\Z_2$           &    cyclic group with two elements (or a set with two elements)\\[1pt]
  $\lgr$           &    lamplighter graph with base $\mathsf{G}$\\[1pt]
  $P$            &    transition matrix on $\lgr$\\[1pt]
  $\partial(\lgr)$ &    geometric boundary of $\lgr$\\[1pt]
  $\lgrp$          &    lamplighter group, subgroup of $\AUT(\lgr)$\\[1pt]
  $\nu$            &    probability measure on $\lgrp$\\[1pt]
  $\eta_n$         &    random configuration of lamps at time $n$, with finite support\\[1pt]
  $\mathcal{C}$    &    set of finitely supported configurations over $\mathsf{G}$\\[1pt]
  $\C_{\mathfrak{u}},\mathfrak{u}\in\partial \mathsf{G}$ & set of lamps configurations accumulating at $\mathfrak{u}$\\[1pt]
  $\Zn$            &    lamplighter random walk (LRW) over $\lgr$ with transition matrix $P$, with\\
                   &    $Z_n=(\eta_n,X_n)$ the position at time $n$, and $\Xn$ the base random walk on $\mathsf{G}$\\[1pt]
  $\Pi$            &    dense subset of $\partial(\lgr)$ toward $\Zn$ converges\\[1pt]
  $\nu_{\infty}$   &    limit distribution of $\Zn$ on $\Pi$\\[1pt]
  $\mathcal{T}_q$  &    homogeneous tree of degree $q$\\[1pt]
\end{tabular} 

\subsubsection{Part II}

\begin{tabular}{ l  l  }
  $\Si$            & a finite alphabet\\[1pt]
  $\Si^*$          & the set of all finite words over $\Si$ \\[1pt]
  $L$              & a language over $\Si$\\[1pt]
  $(X,E,\ell)$     & infinite, oriented, edge-labelled graph with label alphabet $\Si$\\[1pt]
  $\entr(L)$       & the entropy of $L$\\[1pt]
  $L_{x,y}$        & the set of all labeles of paths from $x$ to $y$\\[1pt]
  $p(e)$           & the probability of the edge $e\in E$\\[1pt]
  $P$              & transition matrix over $(X,E,\ell)$\\[1pt]
  $P^h$            & the $h$-transform of the matrix $P$\\[1pt]
  $P^F$            & restricted matrix on $(X,E,\ell)$, where $F\subset \Si^*$\\[1pt]

\end{tabular}
\end{document}